  \documentclass{article}
  \usepackage[longnamesfirst]{natbib}
 
  \usepackage{amssymb}
  \usepackage{amsmath}% if you are using this package,
  \usepackage{amsthm}
  
\usepackage{color}

  \theoremstyle{plain}% default
  \newtheorem{theorem}{Theorem}%[chapter]

  \theoremstyle{definition}
  \newtheorem{definition}[theorem]{Definition}

  \newtheorem{example}[theorem]{Example}

  \theoremstyle{remark}

\usepackage{multirow}

\usepackage[usenames,dvipsnames]{xcolor}

\long\global\def\C#1\F{{}}

\def\term#1{\textsl{#1}}

\usepackage{makeidx}
\makeindex   
\def\subjindex#1{\index{#1}}
\def\authindex#1{\index{~#1}}
\def\indexedterm#1{\term{#1}\subjindex{#1}}

\usepackage{tikz}
\usetikzlibrary{decorations.pathreplacing,arrows}
\usetikzlibrary{shapes.symbols}
\usetikzlibrary{decorations}
\usetikzlibrary{decorations.pathmorphing}

\newcommand{\N}{\mathbb{N}}

\def\st{{\;\vrule height8pt width1pt depth2.5pt\;}}

\def\BBB{\mathcal{B}}
\def\CCC{\mathcal{C}}
\def\DDD{\mathcal{D}}

\def\FFF{\mathcal{F}}
\def\GGG{\mathcal{G}}
\def\HHH{\mathcal{H}}
\def\III{\mathcal{I}}
\def\LLL{\mathcal{L}}
\def\KKK{\mathcal{K}}

\def\OOO{\mathcal{O}}
\def\PPP{\mathcal{P}}
\def\QQQ{\mathcal{Q}}

\def\SSS{\mathcal{S}}

\def\WWW{\mathcal{W}}
\def\ss{\sigma}
\def\es{\varnothing}

\usepackage[shortlabels]{enumitem}   
\setlist[enumerate]{leftmargin=25pt,topsep=3pt}

\def\EEX{\null}  %{\hfill$\diamond$} indicates the end of an exercise
\def\EDE{\null} %\hfill$\heartsuit$} indicates the end of a definition

\def\Prob{\mathbb{P}}

\def\aleks{\texttt{ALEKS}}
\def\rath{\texttt{RATH}}
\def\query{\texttt{QUERY}}

\def\Calc{\textbf{Calc}}   
\def\PreCalc{\textbf{Pre-Calc}}
\def\BusCalc{\textbf{BusCalc}}
\def\CalcExp{\textbf{CalcExp}}

\def\EQ{\Longleftrightarrow}
\def\Span{\mathbb S}

\begin{document}
%%%%%%%%%%%%%%%%%%%%%%%%%%%%%%%%%%%%%%%%%%%%%%%%%%%%%%%%%%%%%%%%%%%%%%

\centerline{\bf\Large Knowledge Spaces and Learning Spaces\footnote{A different version is to appear in \textsl{The New Handbook of Mathematical Psychology}, Cambridge University Press.  The authors thank Laurent Fourny and Keno Merckx for their careful reading of a preliminary draft.}
}

\vskip3mm

\centerline{\textsc{Jean-Paul Doignon}\footnote{%	
    Universit\'e Libre de Bruxelles,
	D\'epartement de Math\'ematique c.p.~216, 
	B-1050~Bruxelles, Belgium.
	\texttt{doignon@ulb.ac.be}}
\qquad \textsc{Jean-Claude Falmagne}\footnote{%
	Department of Cognitive Sciences,
	University of California,
	Irvine,~CA~92617,~USA,
	\texttt{jcf@uci.edu}}
}
	
\vskip3mm

%\date{November 20, 2015}

% \tableofcontents
% \listoffigures
% \listoftables
% \vfill

%%%%%%%%%%%%%%%%%%%%%%%%%%%%%%%%%%%%%%%%%%%%%%%%%%%%%%%%%%%%%%%%%%%%%%
\begin{abstract}
How to design automated procedures which (i)~accurately assess the knowledge of a student, and (ii)~efficiently provide advices for further study?  To produce well-founded answers, Knowledge Space Theory relies on a combinatorial viewpoint on the assessment of knowledge, and thus departs from common, numerical evaluation.  
Its assessment procedures fundamentally differ from other current ones (such as those of S.A.T.\ and A.C.T.).  They are adaptative (taking into account the possible correctness of previous answers from the student) and they produce an outcome which is far more informative than a crude numerical mark.
This chapter recapitulates the main concepts underlying Knowledge Space Theory and its special case, Learning Space Theory. 
We begin by describing the combinatorial core of the theory, in the form of two basic axioms and the main ensuing results (most of which we give without proofs).  In practical applications, learning spaces 
are huge combinatorial structures which may be difficult to manage.  We outline methods providing efficient and comprehensive summaries of such large structures.
We then describe the probabilistic part of the theory,
especially the Markovian type processes which are instrumental in uncovering the knowledge states of individuals. In the guise of the ALEKS system, which includes a teaching component, these methods have been used by millions of students in schools and colleges, and by home schooled students. We summarize some of the results of these applications. 
\end{abstract}

\noindent MSC-class: 91E45

%%%%%%%%%%%%%%%%%%%%%%%%%%%%%%%%%%%%%%%%%%%%%%%%%%%%%%%%%%%%%%%%%%%%%%
\section[Origin and Motivation]{Origin and Motivation
\sectionmark{Origin and Motivation}}
\sectionmark{Origin and Motivation}
\label{origin}

Knowledge Space Theory (abbreviated as KST) originated with a paper by \authindex{Falmagne, J.-Cl.}\authindex{Doignon, J.-P.}\citet{Doignon_Falmagne_1985}.  This work was motivated by the shortcomings of the psychometric approach to the assessment of competence.
The psychometric models are based on the notion that competence can be measured, which the two authors thought was at least debatable.
Moreover, a typical application of a psychometric model in the form of a standardized test results in placing an individual in one of a few dozen ordered categories, which is far too coarse a classification to be useful. In the case of the S.A.T.\footnote{A test for college admission and placement in the U.S. The acronym S.A.T\ used to mean ``Scholastic Aptitude Test''.  A few years ago the meaning of S.A.T.\  was changed into ``Scholastic Assessment Test''.  Today this acronym stands alone, without any associated meaning.}, for example,
the result of the test is a number between 200 and 800 with only multiples of 10 being possible scores.

In the cited paper, Doignon and Falmagne proposed a fundamentally different theory.  The paper was followed by many others, written by them and other researchers (see the Bibliographical Notes in Section~\ref{bibliography}).

The basic idea is that an assessment in a scholarly subject should uncover the individual's `knowledge state', that is,
 the exact set of concepts mastered by the individual.  Here, `concept' means a type of problem that the individual has learned to master, such as, in Beginning Algebra:
\begin{quote} \begin{center}{\sl solving a quadratic equation with integer coefficients;}\end{center}\end{quote}
or, in Basic Chemistry
\begin{quote}\begin{center} 
\textsl{balance a chemical equation using the smallest\\ whole number stoichiometric coefficients\footnote{For example: 
Fe$_2$O$_3(s)\to\,\,\, $Fe$(s) + $O$_2(g)$, which is not balanced. The correct response is: \,\,2Fe$_2$O$_3(s)\,\to\,\,\,$4Fe$(s) + 3$O$_2(g)$.}.}
\end{center}\end{quote}
In KST, a problem type is referred to as an `item'.  Note that this usage differs to that in psychometric, where an item is a particular problem, such as: {\sl Solve the quadratic  equation} $x^2 -x -12= 0$.  In our case, the examples of an item are called \subjindex{instance}\term{instances}\footnote{So, the instances of knowledge space theory are the items of psychometrics.}. 
  
The items or problem types form a possibly quite large set, which we call the `domain' of the body of knowledge.  A knowledge state is a subset of the domain, but not any subset is a state: the knowledge states form a particular collection of subsets, which is called the `knowledge structure' or more specifically (when certain requirements are satisfied) the `knowledge space'  or the `learning space'.  The collection of states captures the whole structure of the domain.
As in Beginning Algebra, which will be our lead example in this chapter, the domain may contain as many as $650$ items, and the learning space may contain many millions states, in sharp contrast with the few dozen scoring categories of a psychometric test.  Despite this large number of possible knowledge states, an efficient assessment is feasible in the course of 25--35 questions. 
 
In Sections~\ref{core_concepts} to \ref{projection_theorem}, we review the fundamental combinatorial concepts and the main axiomatizations.  We introduce the important special case of KST, Learning Space Theory (LST). As the collection of all the feasible, realistic knowledge states may be very large, it is essential to find efficient summaries. We describe two such summaries in Sections~\ref{atoms_base} and  \ref{learning_strings}.  The theory discussed in this chapter has been extensively applied in the schools and universities.  The success of the applications is due in large part to a feature of the knowledge state produced by the assessment: the state is predictive of what a student is ready to learn.  The reason lies in a formal result, the `Fringe Theorem' (see Section~\ref{fringe_theorem}). 
In some situations, it is important to focus on a part of a knowledge structure. We call the relevant concept a `projection' of a knowledge structure on a subset of the items. This is the subject of Section~\ref{projection_theorem}.
 
A next section is devoted to the description of the Markovian type stochastic assessment procedure (Section~\ref{assessment}).   
It relies on the notion of a probabilistic knowledge structure, introduced in Section~\ref{Probabilistic Knowledge Structures}. 
In Section~\ref{sec_Building}, we give an outline of our methods for building the fundamental structure of states for a particular scholarly domain, such as  
Beginning Algebra, Pre-Calculus, or Statistics. Such constructions are enormously demanding and time consuming. They rely not only on dedicated mathematical algorithms, but also on huge bases of assessment data.  The most extensive applications of KST are in the form of the web based system called \subjindex{\aleks}\aleks\footnote{\subjindex{\aleks}\aleks{} is an acronym for ``{\bf A}ssessment and {\bf LE}arming in {\bf K}nowledge {\bf S}paces''.}, which includes a teaching component.  Millions of students have used the system, either at home, or in schools and universities.  Section~\ref{applications} reports some results of these applications.  

This chapter summarizes key concepts and results from two books. One is the monograph of \authindex{Doignon, J.-P.}\authindex{Falmagne, J.-Cl.}\citet{Falmagne_Doignon_LS}\footnote{This is a much expanded reedition of \authindex{Doignon, J.-P.}\authindex{Falmagne, J.-Cl.}\citet{Doignon_Falmagne_KS}.}. The other book is the edited volume of \authindex{Falmagne, J.-Cl.}\authindex{Albert, D.}\authindex{Doble, C.W.}\authindex{Eppstein, D.}\authindex{Hu, X.}\citet{Falmagne_Albert_Doble_Eppstein_Hu}, which contains recent data on the applications of the theory and also some new theoretical results.
A few additional results appear here in Section~\ref{learning_strings} and \ref{sec_Building}.

%%%%%%%%%%%%%%%%%%%%%%%%%%%%%%%%%%%%%%%%%%%%%%%%%%%%%%%%%%%%%%%%%%%%%
\section[Knowledge Structures and Learning Spaces]{Knowledge Structures and Learning Spaces
\sectionmark{Knowledge Structures and Learning Spaces}}
\sectionmark{Knowledge Structures and Learning Spaces}
\label{core_concepts}

We formalize the cognitive structure of a  scholarly subject  as a collection $\KKK$ of subsets of a basic set $Q$ of items. In the case of Beginning Algebra, the items forming $Q$ are the types of problems a student must master to be fully conversant in the subject.   We suppose that the collection $\KKK$ contains at least two subsets: the empty set, which is that of a student knowing nothing at all in the subject, and the full set $Q$ of problems. The next definition cast these basic notions in set-theoretic terms.  We illustrate the definition by a few examples.   

\begin{definition}\label{defn_ks}
A \indexedterm{knowledge structure} is a pair $(Q,\KKK)$ consisting of a  nonempty set $Q$ and a collection $\KKK$ of subsets of $Q$; we assume  $\varnothing\in\KKK$ and $Q\in\KKK$.  The set $Q$ is called the \indexedterm{domain} of the knowledge structure $(Q,\KKK)$.  The elements of $Q$ are the \subjindex{item}\term{items}, and the elements of $\KKK$ are the \subjindex{knowledge state}\term{knowledge states}, or just the \subjindex{state}\term{states}.   
A knowledge structure 
$(Q,\KKK)$ is \subjindex{finite (knowledge structure)}\term{finite} when its domain $Q$ is a finite set.
 
For any item $q$ in $Q$, we write $\KKK_q$ for $\{K\in\KKK\st q\in K\}$, the subcollection of $\KKK$ consisting of all the states containing $q$. A knowledge structure $(Q,\KKK)$ is  
\subjindex{discriminative (knowledge structure)}\term{discriminative} when for any two items $q$ and  $r$ in the domain, we have $\KKK_q=\KKK_r$ only if $q=r$. 
\EDE\end{definition}

We often abbreviate $(Q,\KKK)$ into $\KKK$ (with no loss of information because $Q=\cup\KKK$).

\pagebreak

\begin{example}\label{ex-5 examples}
Here are five examples of knowledge structures all on the same domain
$Q=\{a,b,c,d\}$:
\begin{align*}
\KKK^{(1)} &= \{\es, \{a\}, \{d\}, \{a,b\}, \{a,d\}, \{a,b,c\}, \{a,b,d\},  Q \},\\
\KKK^{(2)} &= \{\es, \{a\}, \{b\}, \{c\},\, \{a,b\}, \{a,c\},\, \{b,c\}, \{a,b,d\},\, \{a,b,c\}, \{a,c,d\},  Q \},\\
\KKK^{(3)} &= \{\es, \{a\}, \{d\}, \{a,b\}, \{c,d\}, \{a,b,c\}, \{b,c,d\},  Q \},\\
\KKK^{(4)} &= \{\es, \{c\}, \{d\}, \{c,d\}, \{a,b,c\}, \{a,b,d\},  Q \},\\\KKK^{(5)} &= \{\es, \{a\}, \{c\}, \{a,b\}, \{c,d\}, \{a,b,c\}, \{a,c,d\},  Q \}.
\end{align*}
The five knowledge structures are finite, and all but $\KKK^{(4)}$ are discriminative: we have 
\begin{equation}
\KKK_a^{(4)} = \KKK_b^{(4)} = \{\{a,b,c\}, \{a,b,d\},\, Q\}\qquad\text{with}\quad a\neq b.
\end{equation}
%%%%%%%%%%% figure
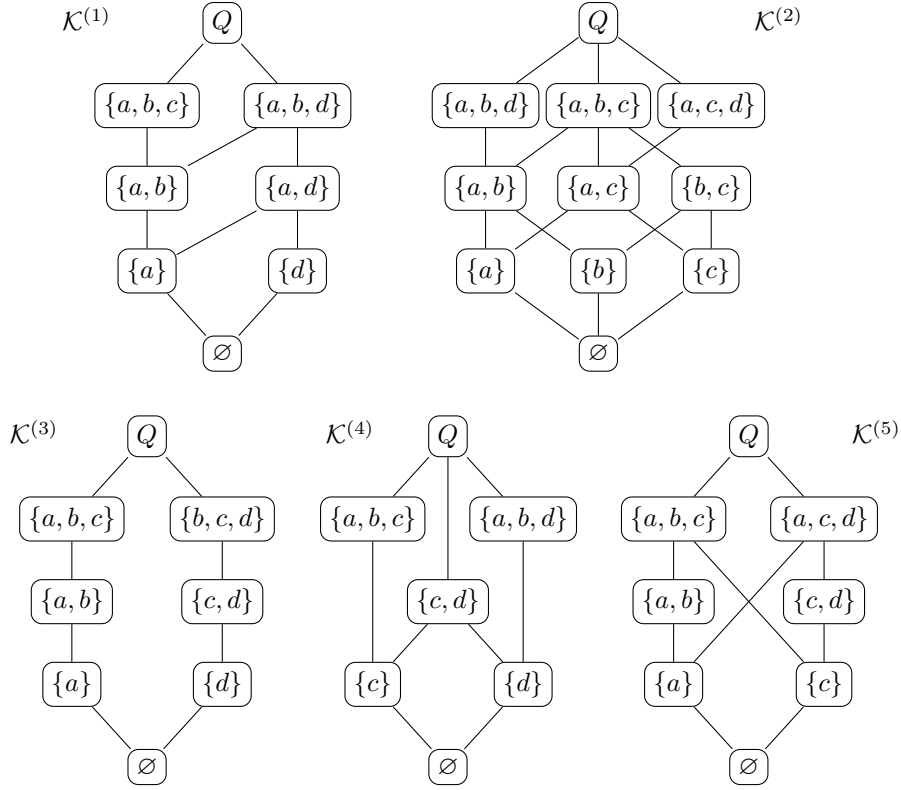
\begin{figure}
\begin{tikzpicture}[yscale=1.1,baseline=60pt,every node/.style={rectangle, draw, rounded corners}]

\begin{scope}[xshift=-1cm]
\path coordinate[label= $\KKK^{(1)}$] (KKK1) at (-1.8,3.8);
\node (es) at (0,0) {$\es$};
\node (l1) at (-1,1) {$\{a\}$};
\node (r1) at ( 1,1) {$\{d\}$};
\node (l2) at (-1,2) {$\{a,b\}$};
\node (r2) at ( 1,2) {$\{a,d\}$};
\node (l3) at (-1,3) {$\{a,b,c\}$};
\node (r3) at ( 1,3) {$\{a,b,d\}$};
\node (Q)  at ( 0,4) {$Q$};
\draw (es) -- (l1) -- (l2) -- (l3) -- (Q);
\draw (es) -- (r1) -- (r2) -- (r3) -- (Q);
\draw (l1) -- (r2);
\draw (l2) -- (r3);
\end{scope}

\begin{scope}[xshift=4cm]
\node[label= $\KKK^{(2)}$,shape=coordinate] at (2.4,3.8) {};
\node (es) at (0,0) {$\es$};
\node (l1) at (-1.5,1) {$\{a\}$};
\node (c1) at ( 0,1) {$\{b\}$};
\node (r1) at ( 1.5,1) {$\{c\}$};
\node (l2) at (-1.5,2) {$\{a,b\}$};
\node (c2) at ( 0,2) {$\{a,c\}$};
\node (r2) at ( 1.5,2) {$\{b,c\}$};
\node (l3) at (-1.5,3) {$\{a,b,d\}$};
\node (c3) at ( 0,3) {$\{a,b,c\}$};
\node (r3) at ( 1.5,3) {$\{a,c,d\}$};
\node (Q)  at ( 0,4) {$Q$};
\draw (es) -- (l1) -- (l2) -- (l3) -- (Q);
\draw (es) -- (c1) -- (r2);
\draw (c2) -- (c3) -- (Q);
\draw (es) -- (r1) -- (r2);
\draw (r3) -- (Q);
\draw (l1) -- (c2) -- (r3);
\draw (c1) -- (l2);
\draw (r1) -- (c2);
\draw (l2) -- (c3);
\draw (r2) -- (c3);
\end{scope}

\begin{scope}[xshift=-2cm,yshift=-5cm]
\node[label= $\KKK^{(3)}$,shape=coordinate] at (-1.5,3.8) {};
\node (es) at (0,0) {$\es$};
\node (l1) at (-1,1) {$\{a\}$};
\node (r1) at ( 1,1) {$\{d\}$};
\node (l2) at (-1,2) {$\{a,b\}$};
\node (r2) at ( 1,2) {$\{c,d\}$};
\node (l3) at (-1,3) {$\{a,b,c\}$};
\node (r3) at ( 1,3) {$\{b,c,d\}$};
\node (Q)  at ( 0,4) {$Q$};\
\draw (es) -- (l1) -- (l2) -- (l3) -- (Q);
\draw (es) -- (r1) -- (r2) -- (r3) -- (Q);
\end{scope}

\begin{scope}[xshift=2cm,yshift=-5cm]
\node[label= $\KKK^{(4)}$,shape=coordinate] at (-1.3,3.8) {};
\node (es) at (0,0) {$\es$};
\node (l1) at (-1,1) {$\{c\}$};
\node (r1) at ( 1,1) {$\{d\}$};
\node (c2) at ( 0,2) {$\{c,d\}$};
\node (l3) at (-1,3) {$\{a,b,c\}$};
\node (r3) at ( 1,3) {$\{a,b,d\}$};
\node (Q)  at ( 0,4) {$Q$};\
\draw (es) -- (l1) -- (l3) -- (Q) ;
\draw (es) -- (r1) -- (r3) -- (Q)  -- (c2);
\draw (l1) -- (c2) -- (r1);
\end{scope}

\begin{scope}[xshift=6cm,yshift=-5cm]
\node[label= $\KKK^{(5)}$,shape=coordinate] at (1.7,3.8) {};
\node (es) at (0,0) {$\es$};
\node (l1) at (-1,1) {$\{a\}$};
\node (r1) at ( 1,1) {$\{c\}$};
\node (l2) at (-1,2) {$\{a,b\}$};
\node (r2) at ( 1,2) {$\{c,d\}$};
\node (l3) at (-1,3) {$\{a,b,c\}$};
\node (r3) at ( 1,3) {$\{a,c,d\}$};
\node (Q)  at ( 0,4) {$Q$};\

\draw (es) -- (l1) -- (l2) -- (l3) -- (Q);
\draw (es) -- (r1) -- (r2) -- (r3) -- (Q);
\draw (l1) -- (r3);
\draw (r1) -- (l3);
\end{scope}
\end{tikzpicture}
\caption[Five examples of knowledge structures]{The five examples of knowledge structures in Example~\ref{ex-5 examples}.}
\label{fig-five-examples}
\end{figure}

In the graphs of these knowledge structures displayed in Figure~\ref{fig-five-examples}, the ascending lines show the covering relation of the states; that is, we have an ascending line from the point representing the state $K$ to the one representing the state $L$ exactly when $K$ is \subjindex{covered (state covered by)}\term{covered} by $L$, that is when $K \subset L$ and moreover there is no state $A$ in $\KKK$ such that $K \subset A \subset L$.  
\hfill\EEX\end{example}

We call a representation of a knowledge structure as exemplified in Figure~\ref{fig-five-examples} a \indexedterm{covering diagram} of the structure.
Note in passing two extreme cases of knowledge structures on a given domain $Q$.  One is $(Q,2^Q)$, where $2^Q$ denotes the \indexedterm{power set} of $Q$, that is the collection of all the subsets of $Q$.  The other one is $(Q,\{\es,Q\})$, in which the knowledge structure contains only the two required states  $\es$ and $Q$.  These two examples are trivial and uninteresting because they 
entail a complete lack of organization in the body of information covered by the items in $Q$.  

Two requirements on a knowledge structure make good pedagogical sense. One is that there should be no gaps in the organization of the material: the student should be able to master the items one by one. Also, there should be some consistency in the items:  an advanced student should have less trouble learning a new item than another, less competent student has. The two conditions incorporated in the next definition formalize the two ideas.

\begin{definition}\label{def_ls_axioms} 
A \indexedterm{learning space} $(Q,\KKK)$ is a knowledge structure which satisfies the two following conditions:
\begin{enumerate}[{[}L1{]}]
\item {\sc Learning smoothness.} For any two states $K$, $L$ with $K \subset L$, there exists a finite chain of states 
\begin{equation}\label{L1 chain}
K=K_0 \subset  K_1\subset\cdots\subset K_p = L
\end{equation}
such that  $|K_{i}\setminus K_{i-1}|=1$ for $1\leq i\leq p$ (thus we have $|L\setminus K| = p$).\\[1mm]
In words: {\sl If the learner is in some state $K$ included in some state $L$, then the learner can reach state $L$ by mastering items one by one.}\\[2mm] 
\item {\sc Learning consistency.}  For any two states $K$, $L$ with $K \subset L$,
if $q$ is an item such that $K \cup \{q\}$ is a state, then  $L \cup \{q\}$ is also a state.\\[1mm]
In words: {\sl Knowing more does not prevent learning something new.} 
\end{enumerate}
\EDE\end{definition}

Notice that any learning space is finite.  Indeed Condition~[L1] applied to the two states $\es$ and $Q$ yields a finite chain of states from $\es$ to $Q$.
In Example~\ref{ex-5 examples} (see also Figure~\ref{fig-five-examples}), only the two structures $\KKK^{(1)}$ and $\KKK^{(2)}$ are learning spaces.  The knowledge structure $\KKK^{(3)}$ satisfies Condition~[L1] in Definition~\ref{def_ls_axioms} but not Condition~[L2] (take $K=\es$, $L=\{a\}$ and $q=d$).  As for $\KKK^{(4)}$, it  satisfies [L2] but not [L1].  Finally, $\KKK^{(5)}$ does not satisfy either condition.  A simpler way to check whether a covering diagram as in Figure~\ref{fig-five-examples} represents a learning space is provided in the next section, just after Theorem~\ref{pro_learning_spaces}.

We give one more example of a learning space, which is realistic in that its ten items belong to the domain of the very large learning space of Beginning Algebra used in the \subjindex{\aleks}\aleks{} system. A remarkable feature of a learning space is that any subset of the items of a learning space also defines a learning space (see below Definition~\ref{def_projection} and Theorem~\ref{projection_theo}).  Thus we have a learning space on these ten items.  This example  will be used repeatedly later on, and in particular to illustrate the assessment mechanism, that is, the questioning algorithm uncovering the knowledge state of a student.

%%%%%%%%%%%%%%%%%%% table
\begin{table}[h!]
%\begin{minipage}{150pt} %{80mm} %
\caption[The ten-item example]{The items of the ten-item example of Figure~\ref{ten_items_figure}.}
\label{ten_items}
\renewcommand{\arraystretch}{1.8}
\begin{tabular}{|l@{~}p{50mm}|@{~}l@{~}p{55mm}|}
\hline
\bf a. & Quotients of expressions involving exponents & \bf b. & Multiplying two binomials\\
\bf c. & Plotting a point in the coordinate plane using a virtual pencil on a Cartesian graph & \bf d. & Writing the equation of a line given the slope and a point on the line\\
\bf e. & Solving a word problem using a system of linear equations\newline(advanced problem) & \bf f. & Graphing a line given its equation \\
\bf g. & Multiplication of a decimal by a whole number & \bf h. & Integer addition \newline(introductory problem)\\
\bf i. & Equivalent fractions: fill the blank in the equation 
$\frac ab = \frac {\boxed{\phantom{.}}} {c}$, 
where $a$, $b$ and $c$ are whole numbers & \bf j. & Graphing integer functions\\
\hline
\end{tabular}
%\end{minipage}
\end{table} 
%%%%%%%%%%%%%%%%%%%%%%%% end of table

\begin{example}\label{ten_items_example}
{\bf Ten items in Beginning Algebra.}  Table~\ref{ten_items} lists ten items.
Remember that items are types of problems, and not particular cases (which are called instances).  Here is an instance of item {\bf d}: {\sl A line passes through the point $(x,y) =(-3,2)$ and has a slope of $6$. Write an equation
 for this line.}  Figure~\ref{ten_items_figure} shows 34 knowledge states in $Q=\{a$, $b$, \dots, $j\}$ which together form a learning space.
 
\pgfdeclarelayer{niveau sommets}
\pgfdeclarelayer{niveau traits}
\pgfsetlayers{niveau traits,niveau sommets}
%%%%%%%%%%%%%%%%%%%%%%%%% figure
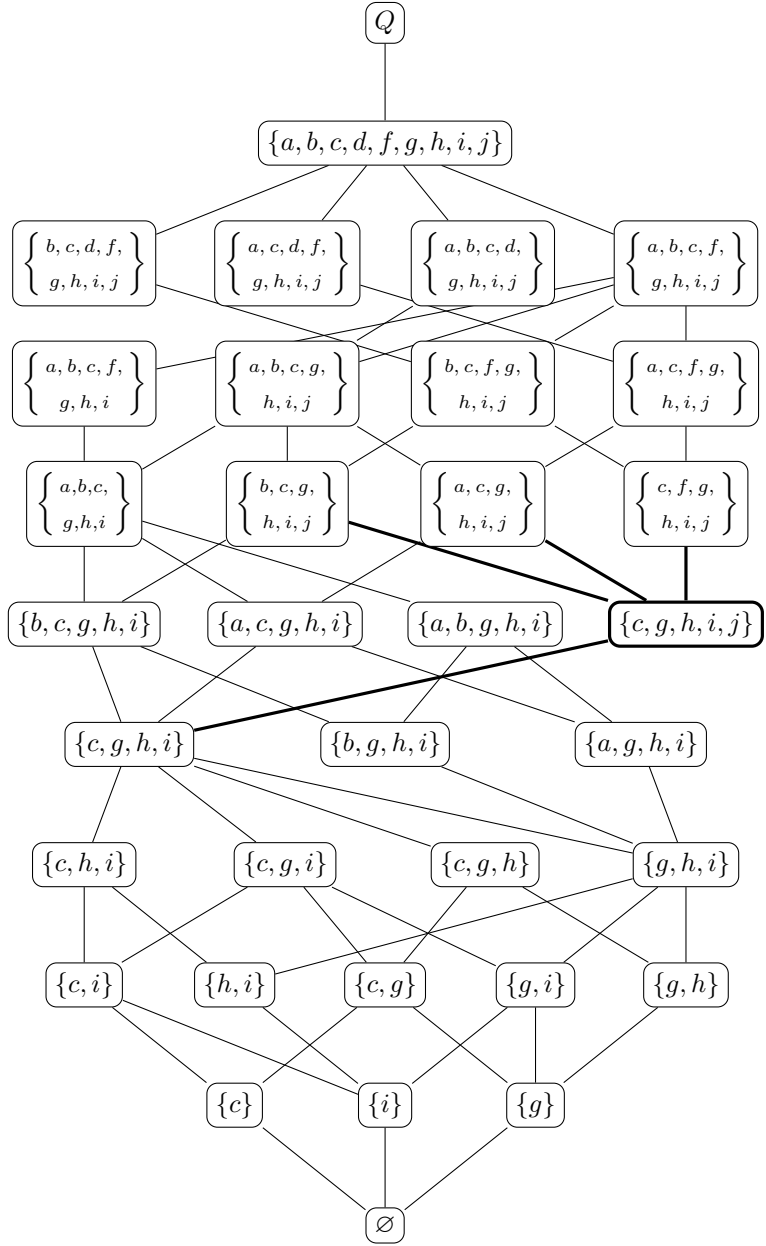
\begin{figure}[h!]
\begin{center}
\begin{tikzpicture}[yscale=1.6]%10-item example

\begin{pgfonlayer}{niveau sommets}
\begin{scope}[every node/.style={rectangle, draw, rounded corners,
fill=white,}]
%%%% 000
\node (es)  at ( 0,0) {$\es$};
%%%% 111
\node (c) at (-2,1) {$\{c\}$};
\node (i) at ( 0,1) {$\{i\}$};
\node (g) at ( 2,1) {$\{g\}$};
%%%% 222
\node (ci) at (-4,2) {$\{c,i\}$};
\node (gi) at ( 2,2) {$\{g,i\}$};
\node (hi) at ( -2,2) {$\{h,i\}$};
\node (gh) at ( 4,2) {$\{g,h\}$};
\node (cg) at ( 0,2) {$\{c,g\}$};
%%%% 333
\node (cgh) at (1.33,3) {$\{c,g,h\}$};%
\node (chi) at (-4,3) {$\{c,h,i\}$};
\node (cgi) at (-1.33,3) {$\{c,g,i\}$};%
\node (ghi) at ( 4,3) {$\{g,h,i\}$};
%%%% 444
\node (cghi) at (-3.4,4) {$\{c,g,h,i\}$};
\node (bghi) at ( 0,4) {$\{b,g,h,i\}$};
\node (aghi) at ( 3.4,4) {$\{a,g,h,i\}$};
%%%% 555
\node (bcghi) at (-4,5) {$\{b,c,g,h,i\}$};%
\node (acghi) at ( -1.33,5) {$\{a,c,g,h,i\}$};
\node[very thick] (cghij) at (  4,5) {$\{c,g,h, i,j\}$};
\node (abghi) at (1.33,5) {$\{a,b,g,h,i\}$};
%%%% 666
\node (cfghij) at (4,6) {$\left\{\begin{smallmatrix}\strut c,\, f,\, g,\\\strut h,\, i,\, j\end{smallmatrix}\right\}$};
\node (bcghij) at (-1.3,6) {$\left\{\begin{smallmatrix}\strut b,\,c,\,g,\\\strut h,\, i,\,j\end{smallmatrix}\right\}$};%
\node (acghij) at ( 1.3,6) {$\left\{\begin{smallmatrix}\strut a,\,c,\,g,\\\strut h,\,i,\,j\end{smallmatrix}\right\}$};
\node (abcghi) at (-4,6) {$\left\{\begin{smallmatrix}\strut a,b,c,\\\strut g,h,i\end{smallmatrix}\right\}$};
%%%% 777
\node (abcfghi) at (-4,7) {$\left\{\begin{smallmatrix}\strut a,\,b,\,c,\,f,\\\strut g,\,h,\,i\end{smallmatrix}\right\}$};
\node (abcghij) at (-1.3,7) {$\left\{\begin{smallmatrix}\strut a,\,b,\,c,\,g,\\\strut h,\, i,\,j\end{smallmatrix}\right\}$};
\node (bcfghij) at (1.3,7) {$\left\{\begin{smallmatrix}\strut b,\,c,\, f,\,g,\\\strut h,\, i,\,j\end{smallmatrix}\right\}$};%
\node (acfghij) at ( 4,7) {$\left\{\begin{smallmatrix}\strut a,\,c,\,f,\,g,\\\strut h,\,i,\,j\end{smallmatrix}\right\}$};
%%%% 888
\node (abcfghij) at ( 4,8) {$\left\{\begin{smallmatrix}\strut a,\,b,\,c,\,f,\\\strut g,\, h, \,i,\,j\end{smallmatrix}\right\}$};
\node (bcdfghij) at (-4,8) {$\left\{\begin{smallmatrix}\strut b,\,c,\,d,\,f,\\\strut g,\,h,\, i,\,j\end{smallmatrix}\right\}$};%
\node (acdfghij) at ( -1.3,8) {$\left\{\begin{smallmatrix}\strut a,\,c,\,d,\,f,\\\strut g,\, h,\,i,\,j\end{smallmatrix}\right\}$};
\node (abcdghij) at ( 1.3,8) {$\left\{\begin{smallmatrix}\strut a,\,b,\,c,\,d,\\\strut g,\,h,\,i,\,j\end{smallmatrix}\right\}$};
%%%% 999
\node (abcdfghij) at (0,9) {$\{a,b,c,d,f,g,h,i,j\}$};
%%%%
\node (Q) at (0,10) {$Q$};
\end{scope}
\node (schtroumpf) at (0,10.5){};
\end{pgfonlayer}

\begin{pgfonlayer}{niveau traits}
\draw (es) -- (c);
\draw (es) -- (i);
\draw (es) -- (g);
\draw (c) -- (ci);
\draw (c) -- (cg);
\draw  (i) -- (hi);
\draw (i) -- (gi);
\draw (g) -- (cg);
\draw (i) -- (ci);
\draw (g) -- (gi);
\draw (g) -- (gh);
\draw (ci) -- (chi);
\draw (hi) -- (chi);
\draw (cg) -- (cgi);
\draw (hi) -- (ghi);
\draw (gi) -- (cgi);
\draw (gh) -- (cgh);
\draw (ci) -- (cgi);
\draw (cg) -- (cgh);
\draw (gh) -- (ghi);
\draw (gi) -- (ghi);
\draw (chi) -- (cghi);
\draw (cgi) -- (cghi);
\draw (cgh) -- (cghi);
\draw (ghi) -- (cghi);
\draw (ghi) -- (bghi);
\draw (ghi) -- (aghi);
\draw (cghi) -- (bcghi);
\draw (cghi) -- (acghi);
\draw[very thick] (cghi) -- (cghij);
\draw (bghi) -- (abghi);
\draw (aghi) -- (acghi);
\draw (bghi) -- (bcghi);
\draw (aghi) -- (abghi);
\draw[very thick] (cghij) -- (acghij);
\draw[very thick] (cghij) -- (bcghij);
\draw[very thick] (cghij) -- (cfghij);
\draw (bcghi) -- (bcghij);
\draw (bcghi) -- (abcghi);
\draw (acghi) -- (acghij);
\draw (acghi) -- (abcghi);
\draw (abghi) -- (abcghi);
\draw (abcghi) -- (abcghij);
\draw (abcghi) -- (abcfghi);
\draw (bcghij) -- (bcfghij);
\draw (bcghij) -- (abcghij);
\draw (acghij) -- (acfghij);
\draw (acghij) -- (abcghij);
\draw (cfghij) -- (bcfghij);
\draw (cfghij) -- (acfghij);
\draw (bcfghij) -- (bcdfghij);
\draw (bcfghij) -- (abcfghij);
\draw (acfghij) -- (acdfghij);
\draw (acfghij) -- (abcfghij);
\draw (abcghij) -- (abcfghij);
\draw (abcghij) -- (abcdghij);
\draw (abcfghi) -- (abcfghij);
\draw (abcfghij) -- (abcdfghij);
\draw(bcdfghij) -- (abcdfghij);
\draw(acdfghij) -- (abcdfghij);
\draw (abcdghij) -- (abcdfghij);
\draw (abcdfghij) -- (Q);
\end{pgfonlayer}

\end{tikzpicture}
\end{center}
\caption[The covering diagram of the ten-item example]{The covering diagram of the ten-item learning space $\LLL$  of Example~\ref{ten_items_example}. The meaning of the four bold lines joining the state $\{c,g,h,i,j\}$  to four other states is explained in Section~\ref{fringe_theorem}.
\label{ten_items_figure}
}
\end{figure}
%%%%%%%%%%%%%%%%%%% end of figure
\EEX\end{example}

\clearpage % ??

%%%%%%%%%%%%%%%%%%%%%%%%%%%%%%%%%%%%%%%%%%%%%%%%%%%%%%%%%%%%%%%%%%%%%
\section[Knowledge Spaces and Wellgradedness]{Knowledge Spaces and Wellgradedness
\sectionmark{Knowledge Spaces and Wellgradedness}}
\sectionmark{Knowledge Spaces and Wellgradedness}
\label{other_axioms}
The name ``learning space'' specifically refers to 
% etait Axioms
Conditions~[L1] and [L2] of Definition~\ref{def_ls_axioms}.  As mentioned earlier, the two conditions have an interesting, pedagogical interpretation.  Other characterizations of the same combinatorial concept focus on some other key concepts, which we review in the present section.  The \indexedterm{symmetric difference} between two sets $K$ and $L$ is defined by $K\bigtriangleup L = (K\setminus L)\cup (L\setminus K)$.

%\newpage % ??
 
\begin{definition}\label{space_wg_downgrad} 
A \indexedterm{knowledge space} $\KKK$ is a knowledge structure which is  closed under union, or \term{$\cup$-closed}\subjindex{union-closed (knowledge structure)}, that is, $\cup \CCC\in \KKK$ for any subcollection $\CCC$ of~$\KKK$. 
The knowledge structure $\KKK$ is \term{well-graded}
\subjindex{well-graded (knowledge structure)}
 if for any two states $K$ and $L$ in $\KKK$, there exists a natural number $h$ such that $|K\bigtriangleup L| = h$ and a finite sequence of states $K=K_0$, $K_1$, \dots, $K_h=L$ such that $|K_{i-1}\bigtriangleup K_i|=1$ for $1\leq i\leq h$.  
The knowledge structure $\KKK$ is \subjindex{accessible (knowledge structure)}\term{accessible} or \subjindex{downgradable (knowledge structure)}\term{downgradable}\footnote{See \citet{Doble_Doignon_Falmagne_Fishburn2001} for the latter term.\authindex{Doble, C.W.}\authindex{Doignon, J.-P.}\authindex{Falmagne, J.-Cl.}\authindex{Fishburn, P.C.}} if for any nonempty state $K$ in $\KKK$, there is some item $q$ in $K$ such that $K\setminus{q}\in \KKK$.  A downgradable, finite knowledge space is called an \indexedterm{antimatroid\,}\footnote{See \citet{Korte_Lovasz_Schrader_1991} for this use of the word. \authindex{Korte, B.}\authindex{Lov\'asz, L.}\authindex{Schrader, R.}  For other authors an antimatroid is closed under intersection rather than under union (and ``upgradable'' rather than downgradable), see for example \authindex{Edelman, P.H.}\authindex{Jamison, R.E.}\citet{Edelman_Jamison_1985}.}.
\EDE\end{definition}

The closure under union is a critical property because it enables a (sometime highly efficient) summary of a knowledge space by a minimal subcollection of its states. The subcollection is called the `base' of the knowledge space and is one of the topics of our next section.  Note that a well-graded knowledge structure $(Q,\KKK)$ is necessarily accessible (given the state $K$, take $L=\es$ in Definition~\ref{space_wg_downgrad}).  
However, an accessible knowledge structure is not necessarily finite nor well-graded.

\begin{example}
Take as items all the natural numbers, thus the domain is~$\N$.  As states, take the empty set plus all the subsets of $\N$ whose complement is finite.  We denote by $\GGG$ the collection of states:
\begin{equation}
\GGG \;=\; \{\es\} \cup \{ K\in 2^\N \;\st\; |\N \setminus K| < +\infty\}.
\end{equation}
The resulting structure $(\N,\GGG)$ is accessible.  It is infinite, and not well-graded (consider for instance the two states $\es$ and $\N$).
\EEX\end{example}

\begin{theorem}\label{big_equivalence}
\sl For any knowledge structure $(Q,\KKK)$, the following three statements are equivalent.
\begin{enumerate}[\quad\rm(i)]
\item $(Q,\KKK)$ is a learning space.
\item $(Q,\KKK)$ is an antimatroid.
\item $(Q,\KKK)$ is a well-graded knowledge space.
\end{enumerate}
\end{theorem}

\authindex{Cosyn, E.}\authindex{Uzun, H.}\citet{Cosyn_Uzun_2009} proved the equivalence of Conditions~(i) and (ii) in Theorem~\ref{big_equivalence}, while \authindex{Korte, B.}\authindex{Lov\'asz, L.}\authindex{Schrader, R.}\citet{Korte_Lovasz_Schrader_1991} established still another characterization of antimatroids (or learning spaces): Theorem~\ref{pro_learning_spaces} below is Lemma~1.2 of their Chapter~3.  We provide a combined proof of Theorems~\ref{big_equivalence} and \ref{pro_learning_spaces} below.

\begin{theorem}\label{pro_learning_spaces}
\sl A knowledge structure $(Q,\KKK)$ is a learning space if and only if its collection $\KKK$ of states satisfies the following three conditions:
\begin{enumerate}[\quad\rm(a)]
\item $Q$ is finite;
\item $\KKK$ is downgradable, that is: any nonempty state $K$ contains some item $q$ such that $K \setminus \{q\} \in \KKK$;
\item for any state $K$ and any items $q$, $r$, if $K \cup \{q\}$, 
$K \cup \{r\}\in \KKK$, then $K \cup \{q,r\}\in \KKK$.
\end{enumerate}
\end{theorem}

Theorem~\ref{pro_learning_spaces} makes it easy to check whether a (finite) covering diagram such as that pictured in Figure~\ref{fig-five-examples} represents a learning space.  Assume that the points representing two states are at the same level (or height) if and only if they have the same number of items; then it suffices to check that: (i)~any ascending line connects points at two successive levels; (ii)~any point representing a nonempty state is the end of at least one ascending line; (iii)~if two ascending lines start from the same point, then their endpoints are the origins of ascending lines having the same endpoint.

\begin{proof}[Proofs of Theorems~\ref{big_equivalence} and \ref{pro_learning_spaces}]
For a given knowledge structure $(Q,\KKK)$, we show that
\begin{center}
(i) $\Rightarrow$ (ii) $\Rightarrow$ (iii) $\Rightarrow$ \big((a), (b) and (c)\big) $\Rightarrow$ (i).
\end{center}
First notice that any of the three conditions (i), (ii), (iii) entails the finiteness of $Q$, that is, Condition~(a).

\medskip

\noindent(i) $\Rightarrow$ (ii). Let $(Q,\KKK)$ be a learning space as in Definition~\ref{def_ls_axioms}. We prove that $(Q,\KKK)$ is an antimatroid as in Definition~\ref{space_wg_downgrad}, in other words that $(Q,\KKK)$ is a finite knowledge space which is moreover downgradable.  Suppose that $K$, $L$ are states.  We first apply Learning Smoothness to $\es$ and $L$ and derive a sequence $L_0=\es$, $L_1$, \dots, $L_\ell=L$ of states such that $|L_i\setminus L_{i-1}|=1$ for $1 \le i \le \ell$.  Then applying Learning Consistency to the states $\es$ and $K$ and the item forming $L_1$, we derive $K \cup L_1\in \KKK$.  Next, we apply Learning Consistency to the states $L_1$ and $K\cup L_1$ and the item forming $L_2\setminus L_1$ to derive $K \cup L_2\in \KKK$.  The general step, for $i=1$, $2$, \dots, $\ell$, applies Learning Consistency to $L_{i-1}$ and $K\cup L_{i-1}$ and the item forming $L_i\setminus L_{i-1}$ to derive $K \cup L_i\in \KKK$.  At the last step ($i=\ell$), we get $K \cup L \in \KKK$.  On the other hand, downgradability of $\KKK$ at the state $K$ is just a particular case of Learning Smoothness at the states $\es$ and $K$. 

\medskip

\noindent(ii) $\Rightarrow$ (iii). If $(Q,\KKK)$ is an antimatroid, then $\KKK$ is closed under union by definition.  To prove the wellgradedness of $\KKK$ (Definition~\ref{space_wg_downgrad}), we take two states $K$ and $L$.  By assumption, $K\cup L$ is also a state, and moreover by downgradability there exist a sequence of states $M_0=\es$, $M_1$, \dots, $M_h=K\cup L$  with $|M_i\setminus M_{i-1}|=1$ for $1 \le i \le h$.  Then $K \cup M_0$, $K \cup M_1$, \dots, $K \cup M_h$, after deletion of repetitions, becomes an increasing sequence $K_0$, $K_1$, \dots, $K_k$ from $K$ to $K \cup L$ with increments consisting of one item.  We  derive a similar sequence $L_0$, $L_1$, \dots, $L_\ell$ from $L$ to $K\cup L$.  Finally, $K_0$, $K_1$, \dots, $K_k=L_\ell$, $L_{\ell-1}$, \dots, $L_0$ is the required sequence from $K$ to $L$ (indeed, $k+\ell=|K\bigtriangleup L|$).
   
\medskip 

\noindent(iii) $\Rightarrow$ ((a), (b) and (c)).  Downgradability~(b) is a direct consequence of wellgradedness.
To prove (c), we only  need to notice $K\cup\{q,r\}= (K\cup\{q\}) \cup (K\cup\{r\})$ and apply the assumed closure under union.

\medskip

\noindent((a), (b) and (c)) $\Rightarrow$ (i). To prove Learning Smoothness,  we consider two states $K$ and $L$ such that $K \subset L$. 
By downgradability, there exist two sequences $K_0=\es$, $K_1$, \dots, $K_k=K$ and $L_0=\es$, $L_1$, \dots, $L_\ell=L$ of states such that $|K_i\setminus K_{i-1}|=1$ for $1 \le i \le k$ and $|L_j\setminus L_{j-1}|=1$ for $1 \le j \le \ell$.  By repeated applications of (c), we derive  
$K_1 \cup L_1$, $K_2 \cup L_1$, \dots, $K_k \cup L_1 \in \KKK$, 
next 
$K_1 \cup L_2$, $K_2 \cup L_2$, \dots, $K_k \cup L_2 \in \KKK$,
etc.,
and finally 
$K_1 \cup L_{l-1}$, $K_2 \cup L_{l-1}$, \dots, $K_k \cup L_{l-1} \in \KKK$.  
Thus $K \cup L_0=K$, $K \cup L_1$, \dots, $K \cup L_{l-1}$, $L$ are all in $\KKK$ (remember $K_k=K$, $L_l=L$ and $K \subset L$).  After deletion of repetitions, we obtain the desired sequence from $K$ to $L$.  

To prove Learning Consistency, we again consider two states $K$ and $L$ with $K \subset L$ together with an item $q$ such that $K \cup \{q\}\in\KKK$.  In the previous paragraph, we proved the existence of a sequence $M_0=K$, $M_1$, \dots, $M_h=L$ of states such that $|M_i\setminus M_{i-1}|=1$ for $1 \le i \le h$.   Applying (c) repeatedly, we obtain $M_1 \cup \{q\} \in \KKK$, $M_2 \cup \{q\} \in \KKK$, \dots, $M_h \cup \{q\} \in \KKK$, the last one being $L \cup \{q\} \in \KKK$ as desired. 
\end{proof}
  
A simple case of a learning space arises when the collection of states is closed under both union and intersection. 

\begin{definition} \label{quasi_ordinal}
A \term{quasi ordinal space}\subjindex{quasi ordinal (space)} is a knowledge space closed under intersection.  A \subjindex{ordinal (space)}\subjindex{partially ordinal (space)}\term{(partially) ordinal space} is a quasi ordinal space which is discriminative.
\EDE\end{definition}

In Example~\ref{ex-5 examples}, only the structure $\KKK^{(1)}$ is a quasi ordinal space, and it is even an ordinal space.  The reason for the terminology in Definition~\ref{quasi_ordinal} lies in Theorem~\ref{Birkhoff} below, due to \authindex{Birkhoff, G.}\citet{Birkhoff37}. 
We recall that a \indexedterm{quasi order} on $Q$ is a reflexive and transitive relation on $Q$.  A \indexedterm{partial order} on $Q$ is a quasi order on $Q$ which is an \indexedterm{antisymmetric relation} (that is, for all $q$ and $r$ in $Q$, it holds that $q R r$ and $r R q$ implies $q=r$).

\begin{theorem}[Birkhoff, 1937]\label{Birkhoff}  
\sl There exists a one-to-one correspondence between the collection of all quasi ordinal spaces $\KKK$ on a set $Q$ and the collection of all quasi orders $\QQQ$ on $Q$.  One such correspondence  is specified by the two equivalences\footnote{We recall the formula $\KKK_q=\{K\in \KKK \st q\in K\}$.}
\begin{gather}
\text{for all }q, r \text{ in } Q:\quad q\QQQ r\; \EQ\; \KKK_q\supseteq \KKK_r;\\
\text{for all } K \subseteq Q:\quad K\in\KKK\; \EQ\; (\forall (q,r)\in \QQQ:\, r\in K \Rightarrow q\in K).\label{eq_ideals}
\end{gather}
Its restriction to discriminative spaces links partially ordinal spaces to partial orders.
\end{theorem}

Note in passing that the closure under intersection does not make good pedagogical sense.
A variant of Theorem~\ref{Birkhoff}\authindex{Birkhoff, G.} for knowledge spaces appears below as Theorem~\ref{surmise_theo}; a variant for learning spaces follows from Theorem~\ref{pro_LS_through_surmise}.
  
%%%%%%%%%%%%%%%%%%%%%%%%%%%%%%%%%%%%%%%%%%%%%%%%%%%%%%%%%%%%%%%%%%%%%%
\section[The Base and the Atoms]{The Base and the Atoms\sectionmark{The Base and the Atoms}}
\sectionmark{The Base and the Atoms}
\label{atoms_base}
In practice, learning spaces tend to be very large, counting millions of states.  For various purposes---for example, to store the structure in a computer's memory---such huge structures need to be summarized. One such summary is the `base' of the structure, which we define below. 

\begin{definition}\label{base_etc}
The \indexedterm{span} of a collection of sets $\FFF$ is the collection of sets $\FFF'$ containing exactly those sets that are unions of sets in $\FFF$.  We say then that $\FFF$ spans $\FFF'$ and we write $\Span(\FFF)=\FFF'$. So, $\Span(\FFF)$ is necessarily $\cup$-closed. A \indexedterm{base} of a $\cup$-closed collection~$\SSS$ of sets is a minimal subcollection $\BBB$ of $\SSS$ spanning $\SSS$---where ``minimal'' refers to inclusion, that is, if $\Span(\HHH)=\SSS$ for some $\HHH\subseteq \BBB$, then necessarily $\BBB\subseteq \HHH$.  
\EDE\end{definition}

Note that by a common convention, the empty set is the union of zero set in $\BBB$.  Accordingly, the empty set never belongs to a base.
It is easily shown that the base of a knowledge space is unique when it exists.  Also, any finite knowledge space has a base \authindex{Doignon, J.-P.}\authindex{Falmagne, J.-Cl.}\citep[see Theorems 3.4.2 and 3.4.4 in][]{Falmagne_Doignon_LS}.  However, some $\cup$-closed collection of sets have no base; an example is the collection of all the open subsets of the set of real numbers. Any learning space has a base because it is finite and $\cup$-closed (cf.~Theorem~\ref{big_equivalence}, (i) $\Leftrightarrow$ (iii)).  
  
For example, the base of the learning space $\KKK^{(2)}$ displayed in Figure~\ref{fig-five-examples} is
\begin{equation}
\{\; \{a\},\, \{b\},\; \{c\},\; \{a,\, b,\, d\},\; \{a,\, c,\, d\}\;\}. 
\end{equation}
The economy is not great in  this little example but may become spectacular in the case of the very large structures encountered in practice.  
  
Another reason for the importance of the base stems from a pedagogical concept. The relevant question is: ``\textsl{Given some item~$q$, which minimal state, or states, must be mastered for $q$ to be mastered?}''.
In more direct words: ``\textsl{what are the minimal states containing a given item $q$?}''.   As one might guess, these minimal sets coincide with the elements of the base.

\begin{definition}\label{atoms}
Let $\KKK$ be a knowledge space. For any item $q$, an \term{atom at} $q$ is a minimal state of $\KKK$ containing $q$.  A state $K$ is called an \indexedterm{atom} if $K$ is an atom at $q$ for some item $q$.  A knowledge space is \subjindex{granular (knowledge space)}\term{granular} if for any item $q$ and any state $K$ containing $q$, there is an atom at $q$ which is included in~$K$. 
\EDE\end{definition}

Clearly, any finite knowledge space is granular.  On the other hand, a state $K$ is an atom in a knowledge space $\KKK$ if and only if any subcollection of states $\FFF$ such that $K=\cup \FFF$ contains $K$ \authindex{Doignon, J.-P.}\authindex{Falmagne, J.-Cl.}\citep[cf.~Theorem 3.4.7 in][]{Falmagne_Doignon_LS}.  Note also that any granular knowledge space has a base \citep[cf.~Proposition~3.6.6 in][]{Falmagne_Doignon_LS}.
 
\begin{example}\label{exa_atoms}
For the ten-item learning space pictured in Figure~\ref{ten_items_figure} (page~\pageref{ten_items_figure}), there are two atoms at item $f$, namely
\begin{equation}
\{c,\, f,\, g,\, h,\, i,\, j\},\quad \{a,\, b,\, c,\, f,\, g,\, h,\, i\}.
\end{equation} 
You can check from the figure that these two sets are indeed minimal states containing $f$ and that they are the only ones with that property.  Note that there is just one atom at the item $b$, which is $\{b$, $g$, $h$, $i\}$, while there are three atoms at $d$.  Table~\ref{ta_atoms} displays the full information on the atoms.
\begin{table}
%\begin{minipage}{300pt}
\caption[Atoms in the ten-item example]{The atoms at all the items of the ten-item learning space from Example~\ref{ten_items_example} (see Example~\ref{exa_atoms}).}
\label{ta_atoms}
\begin{equation*}
\renewcommand{\arraystretch}{1.5}
\renewcommand{\tabcolsep}{1mm}
\begin{array}{c@{\qquad}c}
\hline
\text{Items} & \text{Atoms}\\
\hline
a & \{a,\, g,\, h,\, i\}\\
b & \{b,g,h,i\}\\
c & \{c\} \\
d & \{b,c,d,f,g,h,i,j\},\; \{a,c,d,f,g,h,i,j\},\; \{a,b,c,d,g,h,i,j\}\\
e & Q\\
f & \{c,f,g,h,i,j\},\; \{a,b,c,f,g,h,i\}\\
g & \{g\}\\
h & \{h,i\},\; \{g,h\}\\
i & \{i\}\\
j & \{c,g,h,i,j\}\\
\hline
\end{array}
\renewcommand{\arraystretch}{1}
\end{equation*}
%\end{minipage}
\end{table}
\EEX\end{example}  

We conclude this section with the expected result (a proof is given in  \authindex{Doignon, J.-P.}\authindex{Falmagne, J.-Cl.}\citealp{Falmagne_Doignon_LS}\footnote{The reader should refer to that monograph for most of the proofs omitted in this chapter.}).

\begin{theorem} \label{atom_is_base_set}
\sl Suppose that a knowledge space has a base. Then this base is exactly the collection of all the atoms.
\end{theorem}

A simple algorithm, due to \authindex{Dowling, C.E.}\citet{Dowling:1993} and grounded on the concept of an atom, constructs the base of a finite knowledge space given the states. In the same paper, she also describes a more elaborate algorithm for efficiently building the span of a collection of subsets of a finite set. Both algorithms are sketched in \authindex{Doignon, J.-P.}\authindex{Falmagne, J.-Cl.}\citet[][pages 49--50]{Falmagne_Doignon_LS} (another algorithm for the second task, in another context, is due to Ganter; see \authindex{Ganter, B.}\authindex{Wille, R.}\citealp{Ganter_Reuter91})\authindex{Reuter, K.}. The concept of an atom is closely related to that of the `surmise system', which is the topic of our next section.  We complete the present section with a characterization of learning spaces through their atoms (\authindex{Koppen, M.}\citealp{Koppen1998}).

\clearpage

\begin{theorem}\label{pro_LS_through_atoms}
\sl For any finite knowledge structure $(Q,\KKK)$, the following three statements are equivalent:
\begin{enumerate}[\qquad{\rm(i)}]
\item $(Q,\KKK)$ is a learning space;
\item for any atom $A$ at item $q$, the set $A\setminus\{q\}$ is a state;
\item any atom is an atom at only one item.
\end{enumerate}
\end{theorem}

%%%%%%%%%%%%%%%%%%%%%%%%%%%%%%%%%%%%%%%%%%%%%%%%%%%%%%%%%%%%%%%%%%%%%%
\section[Surmise Systems]{Surmise Systems\sectionmark{Surmise Systems}}
\sectionmark{Surmise Systems}
\label{Surmise Systems}

In a finite knowledge space, a student masters an item $q$ only when his state includes some atom $C$ at $q$. So, the collection of all the atoms at the various items may provide a new way to specify a knowledge space. We illustrate this idea by the following example of a knowledge space.

%%%%%%%%%%% figure
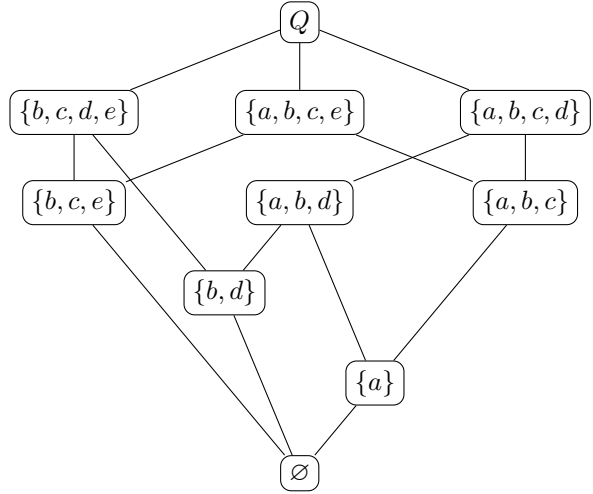
\begin{figure}[h]
\begin{center}
\begin{tikzpicture}[yscale=1.2,baseline=60pt,every node/.style={rectangle, draw, rounded corners}]
\node (es)   at (0,0)    {$\es$};
\node (a)    at ( 1,1)   {$\{a\}$};
\node (bd)   at (-1,2)   {$\{b,d\}$};
\node (bce)  at (-3,3)   {$\{b,c,e\}$};
\node (abd)  at ( 0,3)   {$\{a,b,d\}$};
\node (abc)  at ( 3,3)   {$\{a,b,c\}$};
\node (bcde) at (-3,4)   {$\{b,c,d,e\}$};
\node (abce) at ( 0,4)   {$\{a,b,c,e\}$};
\node (abcd) at ( 3,4)   {$\{a,b,c,d\}$};
\node (Q)    at ( 0,5) {$Q$};
\draw (Q) -- (bcde) -- (bce) -- (es) -- (bd) -- (bcde);
\draw (Q) -- (abce) -- (bce);
\draw (Q) -- (abcd) -- (abd) -- (a) -- (es);
\draw (abcd) -- (abc) -- (a);
\draw (bd) -- (abd);
\draw (abc) -- (abce);
\end{tikzpicture}
\end{center}
\caption[The covering diagram of Example~\ref{ex_HHH}]{The covering diagram of the knowledge space in Example~\ref{ex_HHH}.}
\label{fig_HHH}
\end{figure}
%%%%%%%%%%%%%%%%%%%%%%%%%% end of figure

\begin{example}\label{ex_HHH}
Consider the knowledge space
\begin{align}\nonumber
\HHH = \bigl\{ \es, \{a&\}, \{b,d\}, \{a,b,c\}, \{a,b,d\}, \{b,c,e\},
  \\ \label{srmeqa}
&\{a,b,c,d\},\{a,b,c,e\},\{b,c,d,e\}, \{a,b,c,d,e\} \bigr\}
\end{align}
on the domain $Q = \{a,b,c,d,e\}$.  Figure~\ref{fig_HHH} provides its covering diagram, while Table~\ref{ta_HHH} lists its atoms.
%%%%%%%%%%%%%%%%%%%%%%%%%% table
\begin{table}[h!]
%\begin{minipage}{300pt}
\caption[The atoms in Example~\ref{ex_HHH}]{Items and their atoms in the knowledge space of Equation~\eqref{srmeqa} (see also Figure~\ref{fig_HHH}).}
\label{ta_HHH}
\renewcommand{\arraystretch}{1.5}
\renewcommand{\tabcolsep}{1mm}
\begin{center}
\begin{tabular}{c@{\qquad}c}
\hline
Items&Atoms\\
\hline
$a$ &  $\{a\}$\\
$b$ &  $\{b,d\},\,\{a,b,c\},\,\{b,c,e\}$\\
$c$  & $\{a,b,c\},\,\{b,c,e\}$\\
$d$ & $\{b,d\}$\\
$e$  & $\{b,c,e\}$\\
\hline
\end{tabular}
\end{center}
%\end{minipage}
\end{table}
%%%%%%%%%%%%%%%%%%%%%%%%%% end of table
Table~\ref{ta_HHH} links each of items $a$, $d$ and $e$ to a single atom of $\HHH$, and items $b$ and $c$ to three and two atoms, respectively. So, to master item $b$, one must first master either item $d$, or items $a$ and $c$, or items $c$ and $e$.
\EEX\end{example}

Example~\ref{ex_HHH} illustrates the following definition. 

\begin{definition} \label{srmthreeb}
Let $Q$ be a nonempty set of items. A function $\ss: Q \to 2^{2^Q}$ mapping each item $q$ in $Q$ to a nonempty collection $\ss(q)$  of subsets of $Q$ (so, $\ss(q)\neq\es$) is called an \indexedterm{attribution function} on the set $Q$. 
For each $q$ in $Q$, any $C$ in $\ss(q)$ is called a \subjindex{clause}\term{clause for} $q$ (in $\ss$). 
A \indexedterm{surmise function} $\ss$ on $Q$ is an attribution function on $Q$ which satisfies the three additional conditions, for all  $q,q' \in Q$, and $C,C' \subseteq Q$:
\begin{enumerate}
\item[]
\begin{enumerate}[(i)]
\item if $C \in \ss (q)$, then $q \in C$;
\item if $ q' \in C \in \ss (q)$, then $C' \subseteq C$ for some $C' \in
\ss (q')$;
\item if $C, C' \in \ss (q)$ and  $C' \subseteq C $, then $C = C'$.
\end{enumerate}
\end{enumerate}
In such a case, the pair $(Q,\ss )$ is a \indexedterm{surmise system}. A surmise system $(Q,\ss)$ is \subjindex{discriminative (surmise system)}\term{discriminative} if $\sigma$ is injective (that is: whenever $\ss(q) = \ss(q')$ for some $q,q'\in Q$, then $q =q'$). Then the surmise function $\ss$ is also called \term{discriminative}.
\EDE\end{definition}

It is easily shown that any attribution function $\ss$ on a set $Q$ defines a knowledge space $(Q,\KKK)$ via the equivalence
\begin{equation}\label{attribution_build_space}
K\in \KKK\quad \EQ\quad\forall q\in K,\,\exists C\in\ss(q):\, C\subseteq K. 
\end{equation}
In fact, we have the following extension of Birkhoff's Theorem~\ref{Birkhoff}\authindex{Birkhoff, G.} (it is an extension in the sense that the one-to-one correspondence we obtain extends the correspondence in Birkhoff's Theorem).  The result is due to \authindex{Doignon, J.-P.}\authindex{Falmagne, J.-Cl.}\citet{Doignon_Falmagne_1985}, who derive it from an appropriate `Galois connection'.\authindex{Galois, E.}

\begin{theorem} \label{surmise_theo} 
\sl There exists a one-to-one correspondence between the collection of all granular knowledge spaces $\KKK$ on a set $Q$ and the collection of all the surmise functions $\ss$ on $Q$.  One such correspondence is specified by the equivalence, for all $q$ in $Q$ and $A$ in $2^Q$,
\begin{equation}
A\,\,\text{is an atom at $q$ in $\KKK$}\quad\EQ\quad A\in\ss(q).  
\end{equation}
This one-to-one correspondence links discriminative knowledge spaces to discriminative surmise functions.
\end{theorem}

The correspondence between  knowledge spaces and surmise functions is suggestive of a practical method for building a knowledge space or even a learning space, based on analyzing large sets of learning data. We describe such a method in Section~\ref{sec_Building}.

A characterization of learning spaces through their surmise functions derives directly from Theorem~\ref{pro_LS_through_atoms}. 

\begin{theorem}\label{pro_LS_through_surmise}
\sl A finite knowledge space $(Q,\KKK)$ is a learning space if and only if the corresponding surmise system (as in Theorem~\ref{surmise_theo}) has the property that any clause is a clause for only one item. 
\end{theorem}

In the case of finite, partially ordinal spaces, a highly efficient summary of the space takes the form of the `Hasse diagram' of the partial order.  Attempts to extend the notion of a Hasse diagram from partially ordered sets to surmise systems are reported in \authindex{Doignon, J.-P.}\authindex{Falmagne, J.-Cl.}\citet{Doignon_Falmagne_KS} and \citet{Falmagne_Doignon_LS}.

%%%%%%%%%%%%%%%%%%%%%%%%%%%%%%%%%%%%%%%%%%%%%%%%%%%%%%%%%%%%%%%%%%%%%
\section[The Fringe Theorem]{The Fringe Theorem\sectionmark{The Fringe Theorem}}
\sectionmark{The Fringe Theorem}
\label{fringe_theorem}
The final result of a standardized test is a numerical 
score\footnote{Or a couple of such scores, in the case of a multidimensional model.}.  In the case of an assessment in the framework of a learning space, the result is a knowledge state which may contain hundreds of items.  Fortunately, a meaningful summary of that state can be given in the form of its `inner fringe' and `outer fringe'.
 
In the ten-item Example~\ref{ten_items_example}, consider the state 
$\{c$, $g$, $h$, $i$, $j\}$, which is printed in bold in the covering graph of Figure~\ref{ten_items_figure} (page \pageref{ten_items_figure}).  Figure~\ref{part_ten_items_figure_extract} reproduces the relevant part of the graph, in particular all the adjacent states.  From the state $\{c$, $g$, $h$, $i$, $j\}$, only three items are learnable\footnote{We mean directly learnable without requiring the mastery of any other item outside $\{c$, $g$, $h$, $i$, $j\}$.}, which are $a$, $b$ and $f$ (we mark them on their respective lines).  On the other hand, the only way to reach state $\{c$, $g$, $h$, $i$, $j\}$ is to learn item $j$ from the state $\{c$, $g$, $h$, $i\}$ (which is the unique state giving access to $\{c$, $g$, $h$, $i$, $j\}$).  The two sets of items $\{j\}$ and $\{a$, $b$, $f\}$ completely specify the state $\{c$, $g$, $h$, $i$, $j\}$ among all the states in the structure\footnote{In this particular case, the two-set summary is barely more concise than the original state.  However, in realistic learning spaces containing millions of states, the summary may be considerably smaller than the state.}; this is a remarkable property of learning spaces which we now formalize.

%%%%%%%%%%%%%%%%%%%%%%%%%% figure
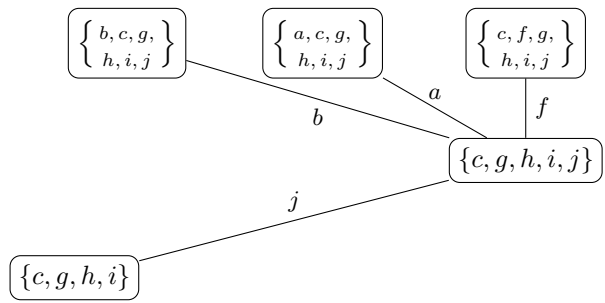
\begin{figure}[t]
\begin{center}
\begin{tikzpicture}[yscale=1.56]
\begin{scope}[every node/.style={rectangle, draw, rounded corners}]
%%%% 444
\node (cghi) at (-2,4) {$\{c,g,h,i\}$};
%%%% 555
\node (cghij) at (4,5) {$\{c,g,h, i,j\}$};
%%%% 666
\node (cfghij) at (4,6) {$\left\{\begin{smallmatrix}\strut c,\, f,\, g,\\ h,\, i,\, j\end{smallmatrix}\right\}$};
\node (bcghij) at (-1.3,6) {$\left\{\begin{smallmatrix}\strut b,\,c,\,g,\\ h,\, i,\,j\end{smallmatrix}\right\}$};%
\node (acghij) at ( 1.3,6) {$\left\{\begin{smallmatrix}\strut a,\,c,\,g,\\ h,\,i,\,j\end{smallmatrix}\right\}$};
\end{scope}
\begin{scope}[every node/.style={font=\small, midway,anchor=text}] %fill=white,anchor=text}]
\draw (cghi) -- node[above]{$j$} (cghij);
\draw (cghij) --node[above]{$a$} (acghij);
\draw (cghij) --node[below]{$b$} (bcghij);
\draw (cghij) --node[right]{$f$} (cfghij);
\end{scope}
\end{tikzpicture}
\end{center}
\caption[In the ten-item example, inner fringe and outer fringe of the state $\{c$, $g$, $h$, $i$, $j\}$]{Part of Figure~\ref{ten_items_figure} showing the items in the inner fringe and outer fringe of the state $\{c$, $g$, $h$, $i$, $j\}$ (see Definition~\ref{fringe_def}).}
\label{part_ten_items_figure_extract}
\end{figure}
%%%%%%%%%%%%%%%%%%%%%%%%%% end of figure

\begin{definition}\label{fringe_def}
Let $(Q,\KKK)$ be a knowledge structure. The \indexedterm{inner fringe} of a state $K$ in $\KKK$ is the set of items 
\begin{equation}
K^\III=\{ q\in K \;\st\; K \setminus \{q\} \in \KKK \}.
\end{equation}
The \indexedterm{outer fringe} of a state $K$ is the set of items
\begin{equation}
K^\OOO=\{q\in Q \setminus K \;\st\; K \cup \{q\} \in\KKK \}.
\end{equation}
\EDE\end{definition}

Note that the empty state $\es$ always has an empty inner fringe, and that the whole domain $Q$ has an empty outer fringe.

\begin{theorem}\label{fringe_theo}  
\sl In a learning space, any state is specified by 
%its two fringes
the pair formed of its inner fringe and its outer fringe.
\end{theorem}

Theorem~\ref{fringe_theo} has the following important consequence.  In any learning space, the knowledge state uncovered by an assessment can be reported as two sets of items: those in its inner fringe and those in its outer fringe.
The outer fringe is especially important because, assuming that the learning space is a faithful representation of the cognitive organization of the material, it tells us exactly what the student is ready to learn. We will see in Section~\ref{applications} that this information is accurate in real-life: the probability that a student actually succeeds in learning an item picked in the outer fringe of his or her state, estimated on the basis of hundreds of thousand \subjindex{\aleks}\aleks{} assessments, is about $.93$ (see page \pageref{outerfringe_success}). 

%%%%%%%%%%%%%%%%%%%%%%%%%%%%%%%%%%%%%%%%%%%%%%%%%%%%%%%%%%%%%%%%%%%%%
\section[Learning Words and Learning Strings]{Learning Words and Learning strings
\sectionmark{Learning Words and Learning Strings}}
\sectionmark{Learning Words and Learning Strings}
\label{learning_strings}
In a learning space, a learner can reach any state by learning its items one at a time---but not in any order. Let us look at an example.

\begin{example}\label{ex_another_example}
A learning space on the domain $Q=\{a,b,c,d\}$ is described by its covering diagram in Figure~\ref{fig_another_example}.

%%%%%%%%%%% figure
\begin{figure}[ht]
\begin{center}
\begin{tikzpicture}[yscale=1.2,baseline=60pt,every node/.style={rectangle, draw, rounded corners}]
%\node[label= $\KKK$,shape=coordinate] at (-2.8,3.8) {};
\node (es)  at (0,0) {$\es$};
\node (a)   at (-1,1) {$\{a\}$};
\node (b)   at ( 1,1) {$\{b\}$};
\node (ab)  at (-1,2) {$\{a,b\}$};
\node (bd)  at ( 1,2) {$\{b,d\}$};
\node (abc) at (-2,3) {$\{a,b,c\}$};
\node (abd) at ( 0,3) {$\{a,b,d\}$};
\node (bcd) at ( 2,3) {$\{b,c,d\}$};
\node (Q)   at ( 0,4) {$Q$};
\draw (es) -- (a) -- (ab) -- (abc) -- (Q);
\draw (es) -- (b) -- (ab) -- (abd) -- (Q);
\draw (b) -- (bd) -- (abd);
\draw (bd) -- (bcd) -- (Q);
\end{tikzpicture}
\end{center}
\caption[The learning space in Example~\ref{ex_another_example}]{The covering diagram of the learning space in Example~\ref{ex_another_example}.}
\label{fig_another_example}
\end{figure}
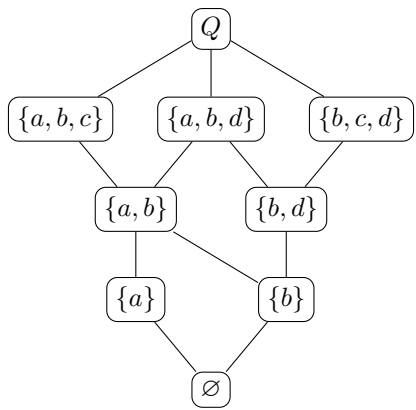
%%%%%%%%%%%%%%%%%%%%%%%%% end of figure

The state $\{a$, $b$, $d\}$ can be reached by mastering the items in three possible successions, which we also call `words' (see Definition~\ref{def_learning_string}):
\begin{align*}
& a\, b\, d,\\
& b\, a\, d,\\
& b\, d\, a.
\end{align*}
For the mastery of the whole domain $Q$, there are 6 `strings' in all:
\begin{align*}
& a\, b\, c\, d,\\
& a\, b\, d\, c,\\
& b\, a\, c\, d,\\
& b\, a\, d\, c,\\
& b\, d\, a\, c,\\
& b\, d\, c\, a.
\end{align*}
\hfill\EEX\end{example}

All of the words and strings in Example~\ref{ex_another_example} share a self-explanatory property:  for any of their `prefixes', the items appearing in the prefix form a knowledge state.  
Let us define the new terminology.

\begin{definition}\label{def_learning_string}
Given some finite set $Q$, a \indexedterm{word} on $Q$ is any injective mapping $f$ from $\{1$, $2$, \dots, $k\}$ to $Q$, for some $k$ with $0\le k \le|Q|$; the case $k=0$ produces the \indexedterm{empty word}.  With $f(i)=w_i$ for $1 \le i \le k$, we write the word $f$ as $w=w_1\,w_2\,\cdots w_k$, and we call $k$ the \subjindex{length (of a word)}\term{length} of the word $w$.   A \subjindex{prefix (of a word)}\term{prefix} of $w$ is a word $w_1\,w_2\,\cdots w_i$, where $0 \le i \le k$ (thus any word is a prefix of itself).  If an item $q$ does not appear in $w$, the \subjindex{concatenation (of a word and an item)}\term{concatenation} of $w$ with $q$ is the word $w\,q=w_1\,w_2\,\cdots\,w_k\,q$.  A  
\indexedterm{string} on $Q$ is a word of length $|Q|$.  Notice that our words (and strings) do not involve repetitions (because of the required injectivity of $f$), a reason for which some authors rather speak of ``simple words'' (as for example \authindex{Boyd, E.A.} \authindex{Faigle, U.}\citealp{Boyd_Faigle_1990}).  
Any word $w= w_1\,w_2 \cdots\,w_k$ \subjindex{determines (a word determines)}\term{determines} the set $\widetilde w=\{w_1$, $w_2$, \dots, $w_k\}$ (if all words are counted, $k!$ of them determine the same set of size $k$).  

Let $(Q,\KKK)$ be a finite knowledge structure with $|Q|=m$. 
A \indexedterm{learning word} (in $(Q,\KKK)$) is a word $w$ on $Q$ such that for each of its prefixes $v=w_1\,w_2\, \cdots\, w_i$ the subset $\widetilde v=\{w_1$, $w_2$, \dots, $w_i\}$ is a state in $\KKK$; here $0 \le i \le k$ if $k$ is the length of $w$.  
A \indexedterm{learning string} is such a learning word with $k=m$.
\EDE\end{definition}

\authindex{Korte, B.}\authindex{Lov\'asz, L.}\authindex{Schrader, R.}\citet{Korte_Lovasz_Schrader_1991} use the expression ``shelling sequence'' for our ``learning word'', and the expression ``basic word'' for our ``learning string'', while \authindex{Eppstein, D.}\citet{Eppstein_2013a} uses ``learning sequence'' for  ``learning string''.

%v The latter term is due to \authindex{Eppstein, D.}\citet{Eppstein_2013a}.  

General knowledge spaces can be without any learning string (for instance, it is the case when $\KKK=\{\es,Q\}$ as soon as $|Q|\ge 2$).  The axioms of learning spaces are consistent with the existence of (many) learning strings and words.  In fact, learning spaces can be recognized from properties of the collection of their learning strings (see next theorem) or the collection of their learning words (see Theorem~\ref{theo_learning_word}). 

\begin{theorem}\label{theo_learning_string}  
\sl Let $Q$ be a finite set, with $|Q|=m$.  A nonempty collection $\SSS$ of strings on $Q$ is the collection of all learning strings of some learning space on $Q$ if and only if $\SSS$ satisfies the three conditions below:
\begin{enumerate}[\quad\rm(i)]
\item any item appears in some string of $\SSS$;
\item 
if $u$ and $v$ are two strings in $\SSS$ such that for some $k$ in $\{1$, $2$, \dots, $m-1\}$ we have 
\begin{equation}
\{\, u_1,\, u_2,\, \dots,\, u_{k-1}\, \} \;=\; \{\, v_1,\, v_2,\, \dots,\, v_{k-1}\,\} \quad \text{and} \quad u_{k}\neq v_k,
\end{equation}
then 
\begin{equation}
u_1\, u_2\, \cdots\, u_{k-1}\, u_{k}\, v_{k}
\end{equation}
is a prefix of some string in $\SSS$;
\item if $u$ and $v$ are two strings in $\SSS$ such that for some $k$ in $\{0$, $1$, \dots, $m-1\}$ and some item $q$ we have 
\begin{equation}
\{v_1,\, v_2,\, \dots, \, v_{k-1},\, v_k,\, v_{k+1}\} \setminus \{u_1,\, u_2,\, \dots,\, u_k\} \;=\; \{q\},
\end{equation}
then $u_1\,u_2\,\cdots\,u_k\,q$ is a prefix of some string in $\SSS$.
\end{enumerate}
\end{theorem}

\begin{proof}
(Necessity.)  Assume $(Q,\LLL)$ is a learning space with $|Q|=m$, and denote by $\SSS$ the collection of its learning strings.   
By Learning Smoothness $\SSS$ is nonempty and each item of $Q$ appears in some string in $\SSS$, so Condition~(i) is true.  If the hypothesis of Condition~(ii) holds, then $\{u_1$, $u_2$, \dots, $u_{k-1}\}$, $\{u_1$, $u_2$, \dots, $u_{k-1}$, $u_{k}\}$ and $\{u_1$, $u_2$, \dots, $u_{k-1}$, $v_k\}$ are all states of $\LLL$.  Hence by Learning Consistency $\{u_1$, $u_2$, \dots, $u_{k-1}$, $u_{k}$, $v_k\}$ is also a state, which we denote by $L$.  On the other hand, by Learning Smoothness, there is a sequence $L_{k+1}$, $L_{k+2}$, \dots, $L_m$ of states with $L_{k+1}=L$, $L_m=Q$, and $|L_i \setminus L_{i-1}|=1$ for $i=k+2$, $k+3$, \dots, $m$.  Taking $w_i$ as the item in  $L_i \setminus L_{i-1}$, we obtain the learning string $u_1\,u_2\,\cdots\,u_{k-1}\,u_{k}\,v_k\,w_{k+2}\,w_{k+3}\,\cdots\,w_{m}$.  Thus Condition~(ii) holds.

Now suppose the strings $u$ and $v$ fulfil the assumption in Condition~(iii).  Thus $\{v_1$, $v_2$, \dots, $v_{k+1}\}$ is a state, that we call $L$.
By Learning Smoothness, there is a sequence $L_{k+1}$, $L_{k+2}$, \dots, $L_m$ of states in $\LLL$ with $L_{k+1}=L$, $L_m=Q$ and $|L_i \setminus L_{i-1}|=1$ for $i=k+2$, $k+3$, \dots, $m$.  With $\{w_i\}=L_i \setminus L_{i-1}$, we obtain the learning string 
$$
u_1\,u_2\,\cdots\,u_{k-1}\,u_{k}\,q\,w_{k+2}\,w_{k+3}\,\cdots\,w_{m}.
$$
Hence Condition~(iii) holds.

\medskip

(Sufficency.) Given a collection $\SSS$ of strings on $Q$ satisfying (i)--(iii), we call $\LLL$ the collection of all prefixes of strings in $\SSS$.  Then $\es$ and $Q$ are in $\LLL$.  Moreover, $\LLL$ clearly satisfies downgradability, that is Condition~(b)  in Theorem~\ref{pro_learning_spaces}.   To establish that $\LLL$ satisfies Condition~(c), let $K$, $K\cup\{q\}$ and $K\cup\{r\}$ be in $\LLL$.  There exist strings $u$ and $v$ such that $\{u_1$, $u_2$, \dots, $u_{|K|}\}=K$ and $\{v_1$, $v_2$, \dots, $v_{|K|+1}\}=K\cup\{q\}$.  Then by Condition~(iii) $u_1\,u_2\,\cdots\,u_{|K|}\,q$ is a prefix of some string in $\SSS$.  A similar argument shows that $u_1\,u_2\,\cdots\,u_{|K|}\,r$ is a string prefix.  Then by Condition~(ii), $u_1\,u_2\,\cdots\,u_{|K|}\,q\,r$ is also a prefix of some string.   So $K\cup\{q,r\}\in\LLL$.   
Hence, by Theorem~\ref{pro_learning_spaces}, $(Q,\LLL)$ is a learning space.  Moreover, the learning strings of $(Q,\LLL)$ constitute exactly $\SSS$ (this derives again from the definition of $\SSS$ together with Condition~(iii)). 
\end{proof}

Here is an example showing that Conditions~(ii) and (iii) in Theorem~\ref{theo_learning_string} are independent.

\begin{example}
On the domain $Q=\{a,b,c,d\}$, the two strings
\begin{align*}
& a\, b\, d\, c,\\
& a\, c\, d\, b
\end{align*}
form a collection which satisfies Condition~(iii) in Theorem~\ref{theo_learning_string} but not Condition~(ii) (take $k=2$, $u_1\,u_2=a\,b$ and $v_1\,v_2=a\,c$).
Conversely, the two strings on the same domain $Q=\{a,b,c,d\}$
\begin{align*}
& a\, b\, c\, d,\\
& b\, a\, d\, c
\end{align*}
form a collection which satisfies Condition~(ii) in Theorem~\ref{theo_learning_string} but not Condition~(iii) (take $k=2$,  $u_1\,u_2=b\,a$ and $v_1\,v_2\,v_3=a\,b\,c$).
\end{example}

The next result is Theorem~2.1 in \authindex{Boyd, E.A.} \authindex{Faigle, U.}\cite{Boyd_Faigle_1990} (compare with Theorem~1.4 in \authindex{Korte, B.}\authindex{Lov\'asz, L.}\authindex{Schrader, R.}\citealp[]{Korte_Lovasz_Schrader_1991}).

\begin{theorem}\label{theo_learning_word}  
\sl Let $Q$ be a finite domain.   
A collection $\WWW$ of words on $Q$ is the collection of all learning words of some learning space on $Q$ if and only if $\WWW$ satisfies the following three conditions:
\begin{enumerate}[\quad\rm(i)]
\item any item from $Q$ appears in at least one word of $\WWW$;
\item any prefix of a word in $\WWW$ also belongs to $\WWW$;
\item if $v$ and $w$ are two words of $\WWW$ with $\widetilde v \not\subseteq \widetilde w$, then for some item $q$ in $\widetilde v \setminus \widetilde w$ the concatenation $w\,q$ is a word again in $\WWW$.
\end{enumerate}
\end{theorem}

\begin{proof}
(Necessity)  Assume $(Q,\LLL)$ is a learning space, and denote by $\WWW$ the collection of all its learning words.  Then by downgradability of $\LLL$ and $Q\in\LLL$, the collection $\WWW$ contains some string, so Condition~(i) is true.  By the definition of a learning word $w$, any prefix of $w$ is also a learning word, so Condition~(ii) holds.  Now take two words $v$ and $w$ as in Condition~(iii).  Then $\widetilde w \cup \widetilde v \in \LLL$ (by Theorem~\ref{big_equivalence}, $\LLL$ is $\cup$-closed). Because of Learning Smoothness, there is a sequence $L_0=\widetilde w$, $L_1$, \dots, $L_\ell=\widetilde w \cup \widetilde v$ with $|L_i \setminus L_{i-1}|=1$ for $i=1$, $2$, \dots, $\ell$.  Let $\{q\}=L_1 \setminus L_0$.  Then $q\in \widetilde v \setminus \widetilde w$ and $w \, q \in \WWW$.

\medskip

(Sufficiency)  Given a collection $\WWW$ of words as in the statement, set $\LLL=\{\widetilde w \st w\in\WWW\}$.  Then $\es\in\LLL$.  Repeatedly applying Conditions~(i) and (iii), we infer that there is a string in $\WWW$, and so $Q \in \LLL$.  By (ii), $\LLL$ is downgradable.  
To conclude that $(Q,\LLL)$ is a learning space it now suffices to prove that $\LLL$ satisfies  Condition~(iii) in Theorem~\ref{pro_learning_spaces}.  Let $K$, $K\cup\{q\}$ and $K\cup\{r\}$ be states in $\LLL$ with $q\neq r$.  
There are then some words $v$ and $w$ in $\WWW$ such that $\widetilde v=K\cup\{q\}$ and $\widetilde w=K\cup\{r\}$.  Because $\widetilde v \setminus \widetilde w=\{q\}$, Condition~(iii) implies $w\,q\in\WWW$.  Now $\widetilde{w\,q}=K\cup\{q,r\}$, and so $K\cup\{q,r\}\in\LLL$.  Finally, it is easily checked that $\WWW$ consists of all learning words of $\LLL$.
\end{proof}

Learning words and strings form a useful tool for the handling of large learning spaces.  For instance, they are implicit in the new representation in Figure~\ref{fig_another_example_bis} of the learning space $\LLL$ from Example~\ref{ex_another_example}: a learning string consists of the letters (representing items) on a path from the vertex representing $\es$ to the vertex representing $Q$.  We call such a representation (with letters displayed only to show addition of a single item) a \indexedterm{learning diagram}.

%%%%%%%%%%% figure
\begin{figure}[ht]
\begin{center}
\begin{tikzpicture}[yscale=1.5]
\node at (0,-0.3) {$\es$};
\node at (0,4.3) {$Q$};
\begin{scope}[every node/.style={circle,scale=0.3,draw}]
\node (es)  at (0,0)  {};
\node (a)   at (-1,1) {};
\node (b)   at ( 1,1) {};
\node (ab)  at (-1,2) {};
\node (bd)  at ( 1,2) {};
\node (abc) at (-2,3) {};
\node (abd) at ( 0,3) {};
\node (bcd) at ( 2,3) {};
\node (Q)   at ( 0,4) {};
\end{scope}
\begin{scope}[every node/.style={font=\small,midway}]
% anchor=west,anchor=text}]
\draw (es) --node[anchor=base east]{\strut$a$} (a);
\draw (es) --node[anchor=base west]{\strut$b$} (b);
\draw (a)  --node[anchor=base east]{\strut$b$} (ab); 
\draw (b)  --node[anchor=base west]{\strut$d$} (bd);;
\draw (b)  --node[anchor=base west]{\strut$a$} (ab); 
\draw (ab) --node[anchor=base east]{\strut$c$} (abc);
\draw (ab) --node[anchor=base east]{\strut$d$} (abd);
\draw (bd) --node[anchor=base east]{\strut$a$} (abd);
\draw (bd) --node[anchor=base west]{\strut$c$} (bcd);
\draw (abc)--node[anchor=base east]{\strut$d$} (Q);
\draw (abd)--node[anchor=base west]{\strut$c$} (Q);
\draw (bcd)--node[anchor=base west]{\strut$a$} (Q);
\end{scope}
\end{tikzpicture}
\end{center}
\caption[The learning diagram of the learning space in Example~\ref{ex_another_example}]{The learning diagram representing the learning space in Example~\ref{ex_another_example} (see also Figure~\ref{fig_another_example}).}
\label{fig_another_example_bis}
\end{figure}
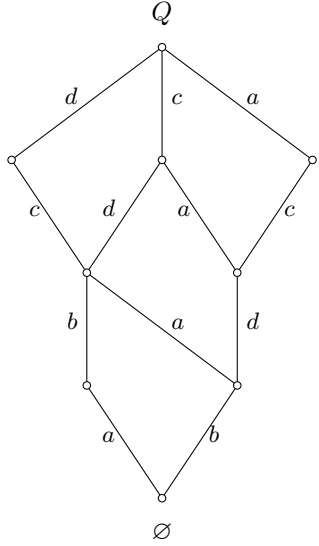
%%%%%%%%%%%%%%%%%%%%%%%%%% end of figure

Figure~\ref{fig_ten_items_learning_diagram} on page~\pageref{fig_ten_items_learning_diagram} shows a similar learning diagram for our ten-item example from Example~\ref{ten_items_example}.  

%%%%%%%%%%%%%%%%%%%%%%%%% figure
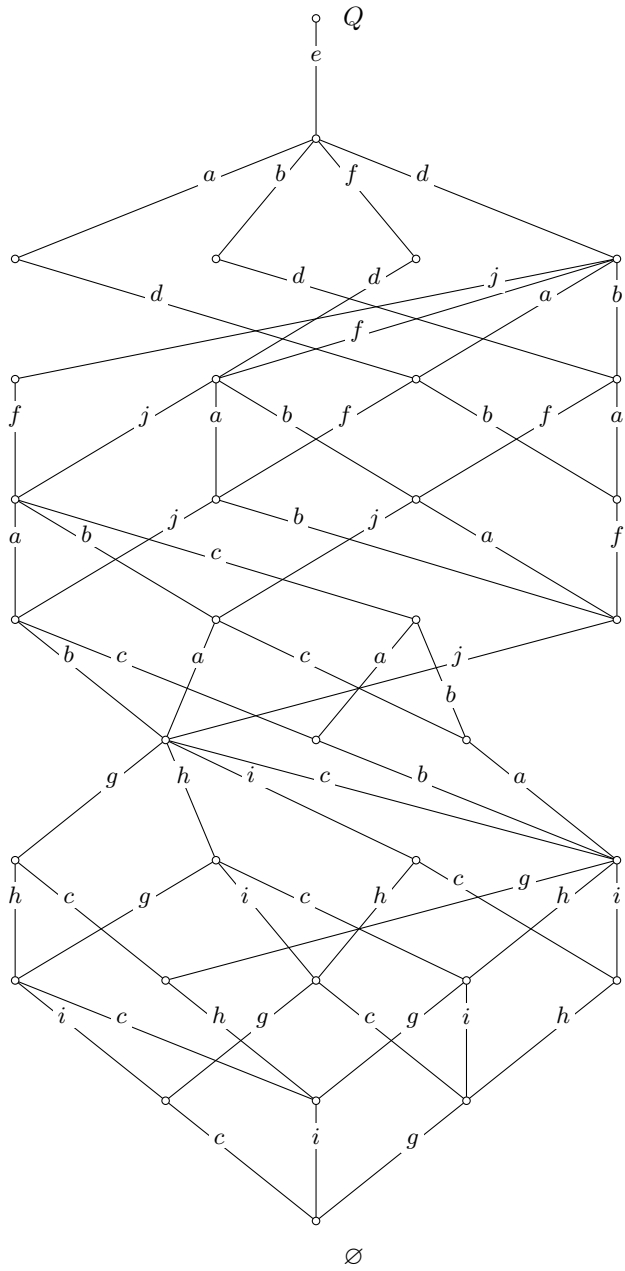
\begin{figure}[t!]
\begin{center}
\begin{tikzpicture}[yscale=1.6]%10-item example

\node at (0.5,-0.3) {$\es$};
\node at (0.5,10) {$Q$};
\begin{scope}[every node/.style={circle,scale=0.3,draw}]
%%% 000
\node (es) at ( 0,0) {};
%%%% 111
\node (c) at (-2,1) {};
\node (i) at ( 0,1) {};
\node (g) at ( 2,1) {};
%%%% 222
\node (ci) at (-4,2) {};
\node (hi) at (-2,2){};
\node (cg) at ( 0,2) {};
\node (gi) at ( 2,2) {};
\node (gh) at ( 4,2) {};

%%%% 333
\node (chi) at (-4,3) {};
\node (cgi) at (-1.33,3) {};
\node (cgh) at ( 1.33,3) {};
\node (ghi) at ( 4,3) {};
%%%% 444
\node (cghi) at (-2,4) {}; % etait (-3.4,4) {};
\node (bghi) at ( 0,4) {};
\node (aghi) at ( 2,4) {}; % etait ( 3.4,4) {};
%%%% 555
\node (bcghi) at (-4,5) {};%
\node (acghi) at (-1.33,5) {};
\node (abghi) at ( 1.33,5) {};
\node (cghij) at ( 4,5) {};  % etait ( 4.5,5) {};
%%%% 666
\node (abcghi) at (-4,6) {};
\node (bcghij) at (-1.33,6) {};%
\node (acghij) at ( 1.33,6) {};
\node (cfghij) at ( 4,6) {};
%%%% 777
\node (abcfghi) at (-4,7) {};
\node (abcghij) at (-1.33,7) {};%
\node (bcfghij) at ( 1.33,7) {};
\node (acfghij) at ( 4,7) {};
%%%% 888
\node (bcdfghij) at (-4,8) {};%
\node (acdfghij) at (-1.33,8) {};
\node (abcdghij) at ( 1.33,8) {};
\node (abcfghij) at ( 4,8) {};
%%%% 999
\node (abcdfghij) at (0,9) {};
%%%%
\node (Q) at (0,10) {};
\end{scope}

\begin{scope}[every node/.style={font=\small, pos=0.65, fill=white,
anchor=base,  inner sep=1.8pt}]
% anchor=base, anchor=text,
%,inner sep=0pt
% pos=0.7,
% near end,
% de 0 ˆ 1
% de 0 ˆ 1
\draw (es) --node{$c$} (c);
\draw (es) --node{$i$} (i);
\draw (es) --node{$g$} (g);
% de 1 ˆ 2
\draw (c) --node{$i$\;\;} (ci);
\draw (c) --node{$g$} (cg);
\draw (i) --node{$h$} (hi);
\draw (i) --node{$g$} (gi);
\draw (g) --node{$c$} (cg);
\draw (i) --node{$c$} (ci);
\draw (g) --node{$i$} (gi);
\draw (g) --node{$h$} (gh);
% de 2 ˆ 3
\draw (ci) --node{$h$} (chi);
\draw (hi) --node{$c$} (chi);
\draw (cg) --node{$i$\;\;} (cgi);
\draw (gi) --node{$c$} (cgi);
\draw (gh) --node[pos=0.8]{$c$} (cgh);
\draw (ci) --node{$g$} (cgi);
\draw (hi) --node[pos=0.8]{$g$} (ghi);
\draw (cg) --node{$h$} (cgh);
\draw (gh) --node{$i$} (ghi);
\draw (gi) --node{$h$} (ghi);
% de 3 ˆ 4
\draw (chi) --node{$g$} (cghi);
\draw (cgi) --node{$h$} (cghi);
\draw (cgh) --node{$i$\;} (cghi);
\draw (ghi) --node{$c$} (cghi);
\draw (ghi) --node{$b$} (bghi);
\draw (ghi) --node{$a$} (aghi);
% de 4 ˆ 5
\draw (cghi) --node{$b$} (bcghi);
\draw (cghi) --node{$a$} (acghi);
\draw (cghi) --node{$j$} (cghij);
\draw (bghi) --node{$a$} (abghi);
\draw (aghi) --node{$c$} (acghi);
\draw (bghi) --node{$c$} (bcghi);
\draw (aghi) --node[pos=0.3]{$b$} (abghi);
% de 5 ˆ 6
\draw (cghij) --node{$a$} (acghij);
\draw (cghij) --node[pos=0.8]{$b$} (bcghij);
\draw (cghij) --node{$f$} (cfghij);
\draw (bcghi) --node[pos=0.8]{$j$} (bcghij);
\draw (bcghi) --node{$a$} (abcghi);
\draw (acghi) --node[pos=0.8]{$j$} (acghij);
\draw (acghi) --node{$b$} (abcghi);
\draw (abghi) --node[pos=0.5]{$c$} (abcghi);
% de 6 ˆ 7
\draw (abcghi) --node{$j$} (abcghij);
\draw (abcghi) --node{$f$} (abcfghi);
\draw (bcghij) --node{$f$} (bcfghij);
\draw (bcghij) --node{$a$} (abcghij);
\draw (acghij) --node{$f$} (acfghij);
\draw (acghij) --node{$b$} (abcghij);
\draw (cfghij) --node{$b$} (bcfghij);
\draw (cfghij) --node{$a$} (acfghij);
% de 7 ˆ 8
\draw (bcfghij) --node{$d$} (bcdfghij);
\draw (bcfghij) --node{$a$} (abcfghij);
\draw (acfghij) --node[pos=0.8]{$d$} (acdfghij);
\draw (acfghij) --node{$b$} (abcfghij);
\draw (abcghij) --node[pos=0.35]{$f$} (abcfghij);
\draw (abcghij) --node[pos=0.8]{$d$} (abcdghij);
\draw (abcfghi) --node[pos=0.8]{$j$} (abcfghij);
% 8 ˆ 9
\draw (abcfghij) --node{$d$} (abcdfghij);
\draw (bcdfghij) --node{$a$} (abcdfghij);
\draw (acdfghij) --node{$b$} (abcdfghij);
\draw (abcdghij) --node{$f$} (abcdfghij);
% de 9 ˆ 10
\draw  (abcdfghij) --node{$e$} (Q);
\end{scope}

\end{tikzpicture}
\end{center}
\caption[The learning diagram of the ten-item example]{The learning diagram of the ten-item learning space $\LLL$ of Example~\ref{ten_items_example} (see also Figure~\ref{fringe_theorem}). 
}
\label{fig_ten_items_learning_diagram}
\end{figure}
%%%%%%%%%%%%%%%%%%%%%%%%% end of figure

Theorem~\ref{theo_learning_string} characterizes learning spaces through their complete collections of learning strings.  
In the same vein as the base which, containing only a relatively small number of states of the knowledge structure, gives us access to the whole collection, we might want to summarize in a similar way the collection of learning strings in a subcollection.  The following definition comes from  \authindex{Eppstein, D.}\citet{Eppstein_2013a}.

\begin{definition}\label{def_strings_to_learning_space}
Let $\SSS$ be a collection of strings on a finite domain $Q$.
Form the collection $\LLL_\SSS$ containing all possible unions of the sets determined by prefixes of strings in $\SSS$.  Then  $(Q,\LLL_\SSS)$ is the learning space \subjindex{encoded (learning space)}\term{encoded} by $\SSS$.
\EDE\end{definition}

That $(Q,\LLL_\SSS)$ indeed forms a learning space is easy to verify (for instance, apply Theorem~\ref{big_equivalence}(ii)).  Now, conversely, given a learning space $(Q,\LLL)$, we may select any nonempty subset $\SSS$ of its collection of learning strings; then, in general, $\LLL_\SSS \subseteq \LLL$ holds, but there is no reason to have equality here (see Theorem~13.5.7 in Eppstein, 2013a, for a criterion).

The case in which we have the equality $\LLL_\SSS = \LLL$ is interesting for algorithmic work on the learning space $\LLL$.  Let us denote as $S_1$, $S_2$, \dots, $S_k$ the strings forming $\SSS$.  Then any state in $\LLL$ is univocally encoded by a list of natural numbers $n_1$, $n_2$, \dots, $n_k$: each of the numbers specifies the length of the prefix we need to extract from the corresponding string in $\SSS$ to get the state at hand as the union of the prefixes. \authindex{Eppstein, D.}\cite{Eppstein_2013a} shows how to exploit the new state encoding for various tasks.  The present context generates the following problem: given a learning space, how do we compute the smallest number of learning strings needed to encode it?  The ensuing invariant was dubbed \indexedterm{convex dimension} by \authindex{Edelman, P.H.}\authindex{Jamison, R.E.}\cite{Edelman_Jamison_1985} 
\authindex{Korte, B.}\authindex{Lov\'asz, L.}\authindex{Schrader, R.}\cite[see also][]{Korte_Lovasz_Schrader_1991}.  
\authindex{Eppstein, D.}\cite{Eppstein_2013b} gives an algorithm to compute it.

\begin{example}
The learning space of Example~\ref{ex_another_example} (see also Figure~\ref{fig_another_example}) is encoded by the following three of its learning strings:
\begin{align*}
& a\, b\, d\, c,\\
& b\, a\, c\, d,\\
& b\, d\, c\, a.
\end{align*}
It is as well encoded by the two strings
\begin{align*}
& a\, b\, c\, d,\\
& b\, d\, c\, a,
\end{align*}
but never by just one string.
\EEX\end{example}

%%%%%%%%%%%%%%%%%%%%%%%%%%%%%%%%%%%%%%%%%%%%%%%%%%%%%%%%%%%%%%%%%%%%%
\section[The Projection Theorem]{The Projection Theorem
\sectionmark{The Projection Theorem}}
\sectionmark{The Projection Theorem}
\label{projection_theorem}
How large is the structure of a real-life learning space? For instance, what is the ratio of the number of knowledge states to the number of possible subsets of the domain?
In the ten-item example of Table~\ref{ten_items} and Figure~\ref{ten_items_figure} we have $34$ knowledge states, which gives the ratio $\frac {34}{2^{10}}\approx .03$. However, this example may be misleading.  In real-life learning spaces, the ratio may become considerably smaller as soon as there are a couple of dozen items.   
As another illustration, we again take the 37 items example in Beginning Algebra.
There are $4\,615$ knowledge states in the corresponding (induced) learning space.  With just $37$ items, the ratio of the number of states to the number of subsets is $\frac {4\,615}{2^{37}}\approx .03\times 10^{-6}$.  We mention in passing that there are $217$ different knowledge states containing exactly $25$ items. Presumably all these knowledge states would be assigned the same psychometric score in classical test theory, while KST treats them as different.
 
As mentioned in Section~\ref{origin}, the full domain of Beginning Algebra in the \subjindex{\aleks}\aleks{} system contains about $650$ items.  The complexity of the resulting learning space is daunting.  
It calls for ways of parsing such huge learning spaces into meaningful components.  One of the goals could be a placement test for which only part of the full collection of items is needed.  A more important reason arises when we need an assessment on the full structure.  In such a case, the number of knowledge states is so large that a straightforward approach becomes infeasible. 

A practical solution has been worked out which consists in suitably partitioning the domain and then carrying on simultaneous, parallel, mutually informative assessments of the resulting substructures, ultimately followed by the computation of the final state.  
We outline this technique in Section~\ref{assessment}.  Here, we define two useful concepts, that of a `projection' and that of `children' of a learning space, given a subset of the domain.  
We show---without proof---that such a projection remains a learning space, while in general children satisfy only some of the requirements in Definition~\ref{def_ls_axioms}.

We begin with an illustration based on a small learning space.

\begin{example}\label{ex_still_another_example}
Let $(Q,\KKK)$ be the learning space on the domain $\{a,b,c,d\}$ whose covering diagram is provided in Figure~\ref{fig_still_another_example}.
%%%%%%%%%%%%%%%%%%%%%%%%%% figure
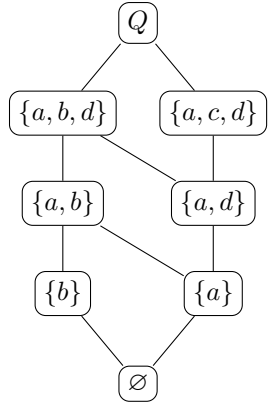
\begin{figure}[ht]
\begin{center}
\begin{tikzpicture}[yscale=1.2,every node/.style={rectangle, draw, rounded corners}]
\node (es)  at (0,0) {$\es$};
\node (a)   at ( 1,1) {$\{a\}$};
\node (b)   at (-1,1) {$\{b\}$};
\node (ab)  at (-1,2) {$\{a,b\}$};
\node (ad)  at ( 1,2) {$\{a,d\}$};
\node (abd) at ( -1,3) {$\{a,b,d\}$};
\node (acd) at ( 1,3) {$\{a,c,d\}$};
\node (Q)   at ( 0,4) {$Q$};
\draw (es) -- (a) -- (ad) -- (acd) -- (Q);
\draw (es) -- (b) -- (ab) -- (abd) -- (Q);
\draw (a) -- (ab);
\draw (ad) -- (abd);
\end{tikzpicture}
\end{center}
\caption[The covering diagram of the learning space in Example~\ref{ex_still_another_example}]{The covering diagram of the learning space in Example~\ref{ex_still_another_example}.}
\label{fig_still_another_example}
\end{figure}
%%%%%%%%%%%%%%%%%%%%%%%%%% end of figure
Consider the subset\linebreak
 $Q'=\{c,d\}$ of the domain $Q$ and form all the `traces' $K \cap Q'$, for $K$ any state in $\KKK$.  The resulting collection, the `projection' of $\KKK$ on $Q'$,
\begin{equation}
\{\, \es,\, \{d\},\, \{c,d\} \,\},
\end{equation}
forms again a learning space.  The general result appears in Theorem~\ref{projection_theo}(i) below.  We summarize the construction in Figure~\ref{fig_patates}. 

%%%%%%%%%%%%%%%%%%%%%%%%%% figure
\begin{figure}[h!]
\begin{center}
\begin{tikzpicture}[scale=0.9]
\begin{scope}[every node/.style={rectangle, draw, rounded corners}]
\node (es)  at (0,0) {$\es$};
\node (a)   at ( 2,2) {$\{a\}$};
\node (b)   at (-2,2) {$\{b\}$};
\node (ab)  at (-2,4) {$\{a,b\}$};
\node (ad)  at ( 2,4) {$\{a,d\}$};
\node (abd) at ( -2,6) {$\{a,b,d\}$};
\node (acd) at ( 2,6) {$\{a,c,d\}$};
\node (Q)   at ( 0,8) {$Q$};
\draw (es) -- (a) -- (ab);
\draw (es) -- (b) -- (ab);
\draw (ad) -- (abd);
\draw (a) -- (ad) -- (acd);
\draw (ab) -- (abd) -- (Q);
\draw (acd) -- (Q);
\end{scope}
\begin{scope}[xshift=45,yshift=-37,rotate=63]
\draw[rounded corners,dashed,thick] (0,0) rectangle (3.7,7);
\end{scope} 
\begin{scope}[xshift=75,yshift=81,rotate=63]
\draw[rounded corners,dashed,thick] (0,0) rectangle (1.4,7);
\end{scope}
\begin{scope}[xshift=80,yshift=136,rotate=63]
\draw[rounded corners,dashed,thick] (0,0) rectangle (2.2,7);
\end{scope} 
\node[draw,thick] (t_es) at (5,-0.75) {$\es$};
\node[draw,thick] (t_d) at (5,2.5) {$\{d\}$};
\node[draw,thick] (t_cd) at (5,4.9) {$\{c,d\}$};
\draw[thick] (2.45,0.45) -- (t_es);
\draw[thick] (2.95,3.5) -- (t_d);
\draw[thick] (3.3,5.75) -- (t_cd);
\end{tikzpicture}
\end{center}
\caption[Two constructions (projection and children)]{An illustration of the two constructions in Example~\ref{ex_still_another_example}
(with $Q'=\{c,d\}$): 
the three equivalence classes are in the rounded, dashed rectangles; the traces are displayed on the right.}
\label{fig_patates}
\end{figure}
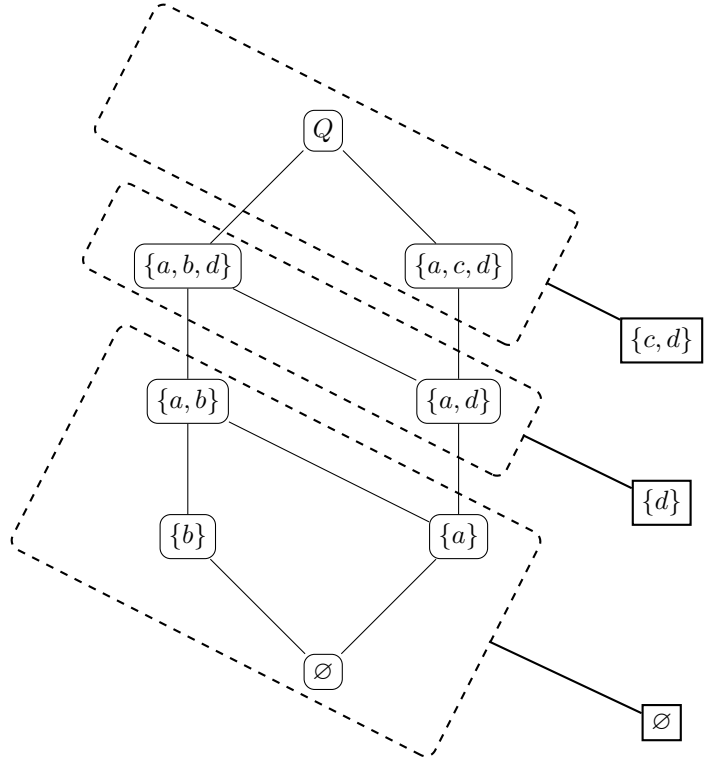
%%%%%%%%%%%%%%%%%%%%%%%%%% end of figure

With the same space we now illustrate another construction.  This time, we sort out the states in $\KKK$ according to their intersections with $Q'=\{c,d\}$.  The resulting equivalence classes are displayed row by row in Table~\ref{table_sorting}.  
Thus the states of $\KKK$ in a same row all have the same trace on $Q'=\{c,d\}$ (shown in the second column).  It happens that they always form a `union-stable', well-graded knowledge structure (see Theorem~\ref{projection_theo}(ii) below).  However, in view of the absence of $\es$, the two last structures do not constitute a learning space.  In the third column, we show the `children'; they are obtained by subtracting from the states (of $\KKK$ shown in that row) their common intersection.
%%%%%%%%%%%%%%%%%%%%%%%%%% table
\begin{table}[h!]  
%\begin{minipage}{300pt}
\caption[Traces and children]{The states of the learning space $(Q,\KKK)$ in Example~\ref{ex_still_another_example} are sorted according to their intersection with $\{c,d\}$.  The third column provides the corresponding children.}
\label{table_sorting}
\begin{equation*}
\renewcommand{\arraystretch}{1.5}
\renewcommand{\tabcolsep}{1mm}
\begin{array}{c@{\quad}c@{\quad}c}
\hline
\text{Classes of states} & \text{Intersections with }\{c,d\} & \text{Children}\\
\hline
\{\{a,c,d\},\, Q \}& \{c,d\}  & \{\es,\, \{b\}\}\\
\{\{a,d\},\, \{a,b,d\}\} & \{d\} & \{\es,\, \{b\}\}\\
\{\es,\, \{a\},\, \{b\},\, \{a,b\}\} & \es & \{\es,\, \{a\},\, \{b\},\, \{a,b\}\}\\
\hline
\end{array}
\end{equation*}
%\end{minipage}
\end{table}
%%%%%%%%%%%%%%%%%%%%%%%%%% end of table
\EEX\end{example}

In the second row of Table~\ref{table_sorting}, all the states (of the original learning space) have in common the items $d$ (by construction) and moreover $a$.  Removing the two common items gives the two sets $\es$ and  $\{b\}$ which altogether form a learning space on the new domain $\{b\}$.  The last assertion is not true in general.

\begin{example}\label{ex_counter_example_child}
Let $(Q,\KKK)$ be the learning space on the domain $\{a,b,c,d\}$ whose covering diagram is provided in Figure~\ref{fig_counter_example_child}.
%%%%%%%%%%%%%%%%%%%%%%%%%% figure
\begin{figure}[ht]
\begin{center}
\begin{tikzpicture}[yscale=1.2,every node/.style={rectangle, draw, rounded corners}]
\node (es)   at ( 0,0) {$\es$};
\node (b)    at ( 0,1) {$\{b\}$};
\node (bc)   at (-2,2) {$\{b,c\}$};
\node (bd)   at ( 2,2) {$\{b,d\}$};
\node (abc)  at (-2,3) {$\{a,b,c\}$};
\node (bcd)  at ( 0,3) {$\{b,c,d\}$};
\node (abd)  at ( 2,3) {$\{a,b,d\}$};
\node (Q)    at ( 0,4) {$Q$};
\draw (es) -- (b) -- (bc) -- (abc) -- (Q);
\draw (bc) -- (bcd) -- (Q);
\draw (b) -- (bd) -- (bcd);
\draw (bd) -- (abd) -- (Q);
\end{tikzpicture}
\end{center}
\caption[The covering diagram of the learning space in Example~\ref{ex_counter_example_child}]{The covering diagram of the learning space in Example~\ref{ex_counter_example_child}.}
\label{fig_counter_example_child}
\end{figure}
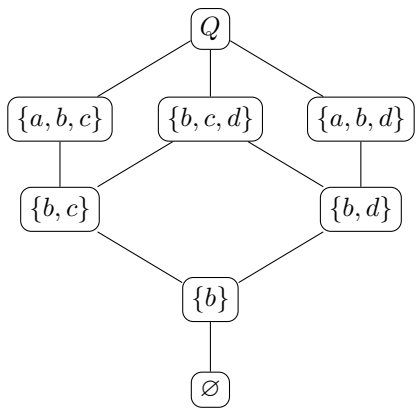
%%%%%%%%%%%%%%%%%%%%%%%%%% end of figure
For $Q'=\{a\}$, one of the two children equals $\{\{c\}$, $\{d\}$, $\{c,d\}\}$; it does not contain the empty set.
\end{example}

We now define the concept of a projection, and then that of children.

\begin{definition} \label{def_projection} 
Suppose that $(Q,\KKK)$ is a knowledge structure with $|Q|\geq 2$,  and let $Q'$ be any proper nonempty subset of $Q$.
For $K$ in $\KKK$, the subset $K \cap Q'$ of $Q'$ is the \indexedterm{trace} of $K$ on $Q'$.  The collection of all traces
\begin{equation}\label{eq KQ'}
\KKK_{|Q'}= \{K\cap Q' \;\st\; K\in\KKK\}
\end{equation}
is the \subjindex{projection (of a knowledge structure)}\term{projection} of $\KKK$ on $Q'$. 
\EDE\end{definition}

Figure~\ref{fig_children_10_items} shows the trace of the ten-item learning space on $\{c,d,g,j\}$.
Note that the sets in $\KKK_{|Q'}$ may be or not be states of $\KKK$.

%%%%%%%%%%%%%%%%%%%%%%%%%% figure
\begin{figure}[t!]
\begin{center}
\begin{tikzpicture}[xscale=1,yscale=1.5]%10-item example
\begin{scope}[every node/.style={rectangle, draw, rounded corners}]
%%%% 000
\node (es)  at ( 0,0) {$\es$};
%%%% 111
\node (c) at (-2,1) {$\{c\}$};
\node (i) at ( 0,1) {$\{i\}$};
\node (g) at ( 2,1) {$\{g\}$};
%%%% 222
\node (ci) at (-4,2) {$\{c,i\}$};
\node (gi) at ( 2,2) {$\{g,i\}$};
\node (hi) at (-2,2) {$\{h,i\}$};
\node (gh) at ( 4,2) {$\{g,h\}$};
\node (cg) at ( 0,2) {$\{c,g\}$};
%%%% 333
\node (chi) at (-4,3) {$\{c,h,i\}$};
\node (cgi) at (-1.33,3) {$\{c,g,i\}$};
\node (cgh) at ( 1.33,3) {$\{c,g,h\}$};
\node (ghi) at ( 4,3) {$\{g,h,i\}$};
%%%% 444
\node (cghi) at (-3.4,4) {$\{c,g,h,i\}$}; % etait (-3.4,4)
\node (bghi) at ( 0,4) {$\{b,g,h,i\}$};
\node (aghi) at ( 3.4,4) {$\{a,g,h,i\}$};  % etait (3.4,4)
%%%% 555
\node (bcghi) at (-4,5) {$\{b,c,g,h, i\}$};
\node (acghi) at (-1.33,5) {$\{a,c,g,h,i\}$};
\node (abghi) at ( 1.33,5) {$\{a,b,g,h,i\}$};
\node (cghij) at ( 4,5) {$\{c,g,h, i,j\}$};
%%%% 666
\node (cfghij) at (4,6) {$\left\{\begin{smallmatrix}\strut c,\, f,\, g,\\ h,\, i,\, j\end{smallmatrix}\right\}$};
\node (bcghij) at (-1.3,6) {$\left\{\begin{smallmatrix}\strut b,\,c,\,g,\\ h,\, i,\,j\end{smallmatrix}\right\}$};%
\node (acghij) at ( 1.3,6) {$\left\{\begin{smallmatrix}\strut a,\,c,\,g,\\ h,\,i,\,j\end{smallmatrix}\right\}$};
\node (abcghi) at (-4,6) {$\left\{\begin{smallmatrix}\strut a,b,c,\\ g,h,i\end{smallmatrix}\right\}$};
%%%% 777
\node (abcfghi) at (-4,7) {$\left\{\begin{smallmatrix}\strut a,\,b,\,c,\,f\\ g,\,h,\,i\end{smallmatrix}\right\}$};

\node (abcghij) at (-1.3,7) {$\left\{\begin{smallmatrix}\strut a,\,b,\,c,\,g,\\ h,\, i,\,j\end{smallmatrix}\right\}$};
\node (bcfghij) at (1.3,7) {$\left\{\begin{smallmatrix}\strut b,\,c,\, f,\,g,\\ h,\, i,\,j\end{smallmatrix}\right\}$};%
\node (acfghij) at ( 4,7) {$\left\{\begin{smallmatrix}\strut a,\,c,\,f,\,g,\\ h,\,i,\,j\end{smallmatrix}\right\}$};
%%%% 888
\node (abcfghij) at ( 4,8) {$\left\{\begin{smallmatrix}\strut a,\,b,\,c,\,f\\ g,\, h, \,i,\,j\end{smallmatrix}\right\}$};
\node (bcdfghij) at (-4,8) {$\left\{\begin{smallmatrix}\strut b,\,c,\,d,\,f\\ g,\,h,\, i,\,j\end{smallmatrix}\right\}$};%
\node (acdfghij) at ( -1.3,8) {$\left\{\begin{smallmatrix}\strut a,\,c,\,d,\,f\\ g,\, h,\,i,\,j\end{smallmatrix}\right\}$};
\node (abcdghij) at ( 1.3,8) {$\left\{\begin{smallmatrix}\strut a,\,b,\,c,\,d\\ g,\,h,\,i,\,j\end{smallmatrix}\right\}$};
%%%% 999
\node (abcdfghij) at (0,9) {$\{a,b,c,d,f,g,h,i,j\}$};
%%%%
\node (Q) at (0,10) {$Q$};
\end{scope}

\draw [rounded corners,dashed,thick] 
(-5.15,7.6) -- (2.4,7.6) -- (2.4,8.3) -- (1.9,9.2) -- (0.3,10.25) -- (-0.3,10.25) -- (-5.15,8.3) -- cycle;
\draw[thick] (-3,9.16) -- (-3.3,9.5) node[above left,draw]{$\{c,d,g,j\}$};

\draw[rounded corners,dashed,thick]   (5.2,8.35) -- (2.9,8.35) -- (2.9,7.4) -- (-2.4,7.4) -- (-2.3,5.6) -- (2.8,5.6) -- (2.8,4.7) -- (5.2,4.7) -- cycle;
\draw[thick] (4,8.35) -- (4.3,8.65) node[above right,draw]{$c,g,j$};

\draw [rounded corners,dashed,thick] (-0.7,1.7) -- (0.7,1.7) -- (1,2.6) -- (2.2,2.8) -- (2.2,3.2) -- (-0.5,3.6) -- (-2.3,3.8) -- (-2.3,4.2) -- (-0.1,4.8) -- (-0.1,5.3) -- (-1.7,5.3) -- (-2.9,5.7) -- (-2.9,7.4) -- (-5.2,7.4) -- (-5.2,4.8) -- (-4.4,3.8) -- (-2.2,3.2) -- (-2.2,2.8) -- (-0.7,2.3) -- cycle;
\draw[thick] (-4.7,4.19) -- (-5,4) node[below left ,draw]{$\{c,g\}$};

\draw [rounded corners,dashed,thick] (1.2,0.7) -- (2.4,0.7) -- (4.9,1.8) -- (4.9,3.2) -- (4.4,4.3) -- (2.6,4.7) -- (2.6,5.25) -- (0.1,5.25) -- (0.1,4.75) -- (-1.1,4.2) -- (-1.1,3.8) -- (3,3.2) -- (3,2.8) -- (1.2,2.2) -- (1.2,1.8) -- cycle;
\draw[thick] (3.5,1.19) -- (4,0.7) node[below right,draw]{$\{g\}$};

\draw [rounded corners,dashed,thick] (-2.5,0.7) -- (-1.4,0.7) -- (-1.4,1.2) -- (-3.1,1.7) -- (-3.1,3.3) -- (-4.9,3.3) -- (-4.9,1.7) -- cycle;
\draw[thick] (-3.3,1.025) -- (-3.8,0.55) node[below left,draw]{$\{c\}$};

\draw [rounded corners,dashed,thick] (-0.5,-0.25) -- (0.5,-0.25) -- (0.5,1.3) -- (-0.1,1.3) -- (-1.3,1.8) -- (-1.3,2.3) -- (-2.7,2.3) -- (-2.7,1.8) -- (-0.5,1) -- cycle;
\draw[thick] (0.47,-0.2) -- (1,-0.5) node[below right,draw]{$\es$};

\end{tikzpicture}
\end{center}
\caption[The children of the ten-item example (projection on $\{c,d,g,j\}$)]{Projection of the ten-item example on $\{c,d,g,j\}$: the dashed lines delineate the equivalence classes, the little black rectangles show the traces.}
\label{fig_children_10_items}
\end{figure}
%%%%%%%%%%%%%%%%%%%%%%%%%% end of figure

\begin{definition} \label{def_children} 
Suppose again that $(Q,\KKK)$ is a knowledge structure with $|Q|\geq 2$, and let $Q'$ be any proper nonempty subset of $Q$.
Define the relation $\sim_{Q'}$ on $\KKK$ by
\begin{align}\label{def_sim_eq}
K \sim_{Q'} L \quad&\EQ\quad K\cap Q' =  L \cap Q'\\
\label{def_sim_eq_2}
&\EQ\quad K\bigtriangleup L \subseteq Q\setminus Q'.
\end{align}
(The equivalence between the right hand sides of (\ref{def_sim_eq}) and (\ref{def_sim_eq_2}) is easily verified.)
Then $\sim_{Q'}$ is an equivalence relation on $\KKK$.  When the context specifies the subset $Q'$,  we may simply write $\sim$ for $\sim_{Q'}$.
We denote by $[K]$ the equivalence class of $\sim$ containing $K$, and by $\KKK_\sim=\{\,[K]\,\st\, K\in\KKK\}$ the partition of $\KKK$ induced by $\sim$ (or by $Q'$).  
For any state $K$ in $\KKK$, we define  the collection
\begin{equation}
\label{def Kk}
\KKK_{[K]}= \{\,L\setminus \cap [K]\, \;\st\; L \in [K]\}.
\end{equation}
The collection $\KKK_{[K]}$ is  a $Q'$-\term{child} of $\KKK$, or simply a \indexedterm{child} of $\KKK$ when the set $Q'$ is made clear by the context.  A child of $\KKK$ may take the form of the singleton $\{\es\}$; it is then the \subjindex{trivial (child)}\term{trivial} child.  We refer to $\KKK$ as the \subjindex{parent (structure)}\term{parent} structure.
\EDE\end{definition}

Because $\es\in\KKK$ we have $\KKK_{[\es]}=[\es]$.  We may have $\KKK_{[K]} = \KKK_{[L]}$ even when $K\not\sim L$. (Examples are easily built: see Table~\ref{table_sorting}.)

\begin{theorem} \label{projection_theo} 
\sl Let $(Q,\KKK)$ be a learning space, with $|Q|\geq 2$. The following two properties hold for any proper nonempty subset $Q'$ of $Q$.
\begin{enumerate}[\quad\rm(i)]
\item The projection $\KKK_{|Q'}$ of $\KKK$ on $Q'$ is a learning space.
\item The children of $\KKK$ are well-graded and $\cup$-stable collections.  The latter means that the union of any nonempty subcollection of the collection also belongs to the 
collection\footnote{Notice a slight difference between ``union-stable'' and ``union-closed''; only the second condition entails that the empty set belongs to the collection.}. Example~\ref{ex_still_another_example} shows that the children are not necessarily learning spaces.
\end{enumerate}
\end{theorem}

We can impose some restricting conditions on the set $Q'$ that guarantee any child to be a learning space, provided that the empty state is added to the collection if necessary \authindex{Doignon, J.-P.}\authindex{Falmagne, J.-Cl.}\citep[see][Definition 2.4.11 and Theorem 2.4.12]{Falmagne_Doignon_LS}.    

The concept of a projection plays an essential role in designing assessment algorithms for realistic learning spaces. In applications, the size of the collection of states may be so prohibitively large that the obvious strategy of gradually narrowing down, by some method or other, the class of states consistent with the assessment results is not practical.  The solution discussed in Subsection~\ref{Modified Response Rule} is to first design a suitable partition of the domain into $N$ manageable classes.  Second, one builds the $N$ projections on these classes.  The assessment procedure then operates in parallel on the $N$ projections.  A combination of the results of the $N$ simultaneous assessments delivers a final knowledge state.    

%%%%%%%%%%%%%%%%%%%%%%%%%%%%%%%%%%%%%%%%%%%%%%%%%%%%%%%%%%%%%%%%%%%%%
\section[Probabilistic Knowledge Structures]{Probabilistic Knowledge Structures
\sectionmark{Probabilistic Knowledge Structures}}
\sectionmark{Probabilistic Knowledge Structures}
\label{Probabilistic Knowledge Structures}
The concept of a learning space is deterministic. As such, it does not provide realistic predictions of subjects' responses to the problems of a test.  Probabilities must enter in at least two ways in
a realistic model. For one, the knowledge states will certainly occur with different frequencies in the population of reference. So, it makes sense to postulate the existence of a probability distribution on the collection of states. For another, a subject's knowledge state does not necessarily specify the observed responses. A subject having mastered an item may be careless in responding, and make an error. Also, in some situations, a subject may be able to guess the correct response to a question not yet mastered\footnote{Such lucky guesses have probability zero or are very unlikely in assessment systems requiring open responses to the items, instead of multiple-choice. \subjindex{\aleks}\aleks{} is one of those systems.}. In general, it makes sense to introduce conditional probabilities of responses, given the states.

\begin{definition}\label{def probabilistic ks}
A \indexedterm{probabilistic (knowledge structure)} $(Q,\KKK,p)$ consists of a finite knowledge structure $(Q,\KKK)$ with a probability distribution $p$ on the collection $\KKK$ of states.  Thus $p(K)$ is a real number with $0\leq p(K) \leq 1$, for any state $K$, and moreover $\sum_{K\in\KKK}p(K)=1$.
A \indexedterm{parametrized (probabilistic knowledge structure)} $(Q,\KKK,p, \beta, \eta)$ is a probabilistic knowledge structure $(Q,\KKK,p)$ equipped with two functions $\beta: Q \to [0,1]:q\mapsto \beta_q$ and $\eta:Q\to [0,1]:q\mapsto \eta_q$. The number $\beta_q$ represents the \indexedterm{careless error probability} to item $q$, and the number $\eta_q$ is the \indexedterm{lucky guess probability} for item $q$. 
\EDE\end{definition}

\begin{definition}\label{def straight ppks}
A parametrized probabilistic knowledge structure $(Q,\KKK$, $p, \beta, \eta)$ is \subjindex{straight (parametrized probabilistic knowledge structure)}\term{straight} if $\beta_q=\eta_q=0$ for all items $q$.  
\EDE\end{definition}

We now describe two types of construction of probability distributions on a set of states. In the first type, the set of states results from a projection.

\begin{definition}\label{def_prob_projection}
Let $(Q,\KKK,p)$ be a probabilistic knowledge structure, and let $\es \neq Q'\subset Q$.  On the projection $\KKK_{|Q'}$, we define the \indexedterm{projected distribution}~$p'$ by setting, for $K'$ a state in $\KKK_{|Q'}$,
$$
p'(K') \;=\; 
\sum\left\{ p(K) \st K\in\KKK \text{ and }K \cap Q'=K' \right\}.
$$ 
Then $\left( Q',\KKK_{Q'},p' \right)$ is the \subjindex{probabilistic projection (of a probabilistic knowledge structure)}\term{probabilistic projection} on $Q'$ of the probabilistic knowledge structure $(Q,\KKK,p)$. 
\end{definition}

In the second case, we start with a probabilistic knowledge structure and extend it on a larger set of states.

\begin{definition}\label{def_prob_extension}
Let $(Q,\KKK)$ be a knowledge structure, and let $\es \neq Q'\subset Q$.  Assume $p'$ is a probability distribution on the projection $\KKK_{|Q'}$.  We define the \indexedterm{extended distribution} $p^+$ to $\KKK$ of $p$ by setting, for $K$ a state in $\KKK$,
$$
p^+(K) \;=\; 
\dfrac
{p'(K \cap Q')}
{\sum\left|\{ L\in\KKK \st L \cap Q' = K \cap Q'\} \right|}.
$$ 
Then $\left(Q,\KKK,p^+ \right)$ is the \subjindex{uniform 
extension (of a probabilistic knowledge structure)}\term{uniform extension} to $(Q,\KKK)$ of the probabilistic knowledge structure $(Q',\KKK',p')$. 
\end{definition}

%%%%%%%%%%%%%%%%%%%%%%%%%%%%%%%%%%%%%%%%%%%%%%%%%%%%%%%%%%%%%%%%%%%%%%
\section[The Stochastic Assessment Algorithm]{The Stochastic Assessment Algorithm
\sectionmark{The Stochastic Assessment Algorithm}}
\sectionmark{The Stochastic Assessment Algorithm}
\label{assessment}
The general idea of the assessment algorithm is to gradually update, after each response of the subject, the distribution of probabilities on the collection of states.  On each step of the assessment, the system selects an item, and presents a randomly chosen instance of that item to the student.  The student's response is evaluated and classified as ``correct'' or ``false''.  The result serves to update the probability distribution on the set of states.  The new distribution is the starting point of the next step.  Ultimately, only one or a few states will remain with a high probability.  The system then chooses the final state.
 
The assessment algorithm we just sketched is applicable in the `straightforward situation', that is, when the learning space (or the knowledge structure for that matter) is moderately large, with a domain not exceeding $50$ items.  Such learning spaces can serve in the design of some placement tests, for example. In Subsection~\ref{parallel}, we deal with the more usual case of domains having hundreds of items.  
On such large domains, the application of the algorithm requires 
considerably more sophistication because the number of states becomes so large that operating on the probability distribution on the set of states is unmanageable. The solution outlined in Subsection~\ref{parallel} is to build a suitable partition of the domain and to perform parallel (simultaneous, mutually informative) 
assessments on the projections on all the subdomains. 
The final state is constructed by combining the outcomes obtained in each of the parallel assessments.

%---------------------------------------------------------------------
\subsection{Sketch of the algorithm in the straightforward situation}
%%%%%%%%%%%%%%%%%%%%%%%%%% figure
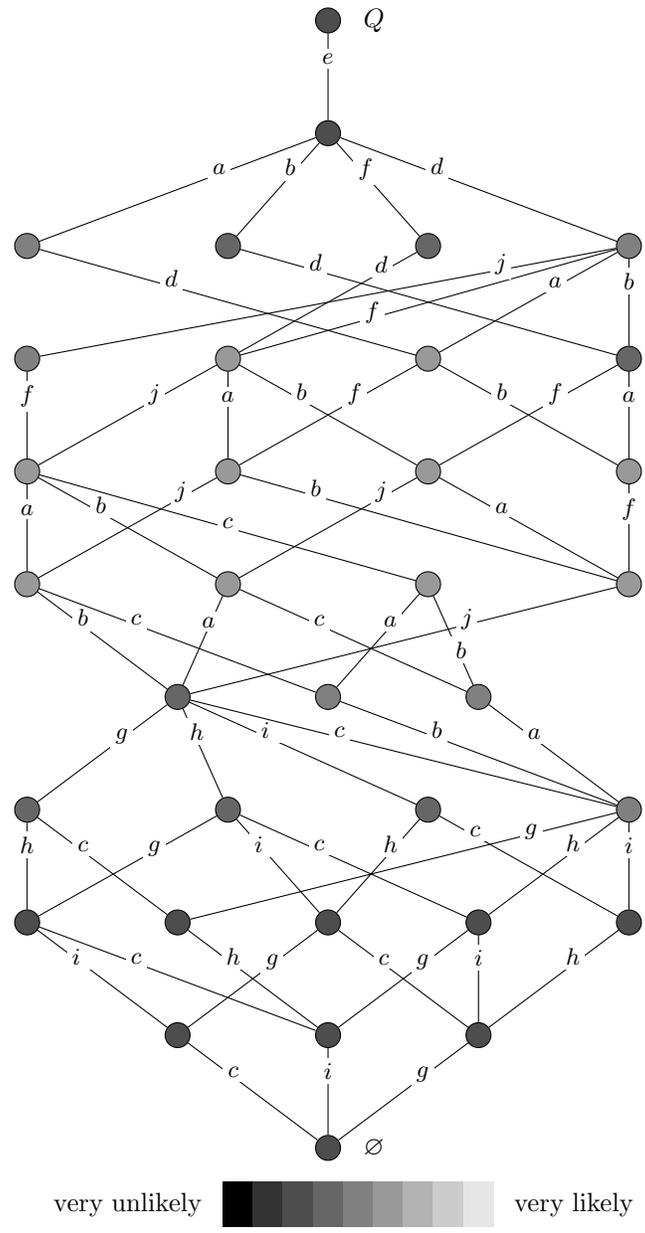
\begin{figure}[t!]
\begin{center}
\begin{tikzpicture}[yscale=1.5]%10-item example
\begin{scope}[every node/.style={circle,scale=1,draw}]
% etat initial
%..........
%%% 000
\node[fill=black!70] (es) at ( 0,0) {};
%%%% 111
\node[fill=black!70] (c) at (-2,1) {};
\node[fill=black!70] (i) at ( 0,1) {};
\node[fill=black!70] (g) at ( 2,1) {};
%%%% 222
\node[fill=black!70] (ci) at (-4,2) {};
\node[fill=black!70] (gi) at ( 2,2) {};
\node[fill=black!70] (hi) at (-2,2){};
\node[fill=black!70] (gh) at ( 4,2) {};
\node[fill=black!70] (cg) at ( 0,2) {};
%%%% 333
\node[fill=black!60] (chi) at (-4,3) {};
\node[fill=black!60] (cgi) at (-1.33,3) {};
\node[fill=black!60] (cgh) at ( 1.33,3) {};
\node[fill=black!50] (ghi) at ( 4,3) {};
%%%% 444
\node[fill=black!60] (cghi) at (-2,4) {}; % etait (-3.4,4)
\node[fill=black!50] (bghi) at ( 0,4) {};
\node[fill=black!50] (aghi) at ( 2,4) {};  % etait ( 3.4,4)
%%%% 555
\node[fill=black!40] (bcghi) at (-4,5) {};%
\node[fill=black!40] (acghi) at (-1.33,5) {};
\node[fill=black!40] (abghi) at ( 1.33,5) {};
\node[fill=black!40] (cghij) at ( 4,5) {};
%%%% 666
\node[fill=black!40] (abcghi) at (-4,6) {};
\node[fill=black!40] (bcghij) at (-1.33,6) {};%
\node[fill=black!40] (acghij) at ( 1.33,6) {};
\node[fill=black!40] (cfghij) at ( 4,6) {};
%%%% 777
\node[fill=black!50] (abcfghi) at (-4,7) {};
\node[fill=black!40] (abcghij) at (-1.33,7) {};%
\node[fill=black!40] (bcfghij) at ( 1.33,7) {};
\node[fill=black!60] (acfghij) at ( 4,7) {};
%%%% 888
\node[fill=black!50] (bcdfghij) at (-4,8) {};%
\node[fill=black!60] (acdfghij) at (-1.33,8) {};
\node[fill=black!60] (abcdghij) at ( 1.33,8) {};
\node[fill=black!50] (abcfghij) at ( 4,8) {};
%%%% 999
\node[fill=black!70] (abcdfghij) at (0,9) {};
%%%%
\node[fill=black!70] (Q) at (0,10) {};
\end{scope}%..........

\begin{scope}[every node/.style={font=\small, pos=0.65, fill=white,
anchor=base,  inner sep=1.8pt}]
% anchor=base, anchor=text,
%,inner sep=0pt
% pos=0.7,
% near end,
% de 0 ˆ 1
% de 0 ˆ 1
\draw (es) --node{$c$} (c);
\draw (es) --node{$i$} (i);
\draw (es) --node{$g$} (g);
% de 1 ˆ 2
\draw (c) --node{$i$\;\;} (ci);
\draw (c) --node{$g$} (cg);
\draw (i) --node{$h$} (hi);
\draw (i) --node{$g$} (gi);
\draw (g) --node{$c$} (cg);
\draw (i) --node{$c$} (ci);
\draw (g) --node{$i$} (gi);
\draw (g) --node{$h$} (gh);
% de 2 ˆ 3
\draw (ci) --node{$h$} (chi);
\draw (hi) --node{$c$} (chi);
\draw (cg) --node{$i$\;\;} (cgi);
\draw (gi) --node{$c$} (cgi);
\draw (gh) --node[pos=0.8]{$c$} (cgh);
\draw (ci) --node{$g$} (cgi);
\draw (hi) --node[pos=0.8]{$g$} (ghi);
\draw (cg) --node{$h$} (cgh);
\draw (gh) --node{$i$} (ghi);
\draw (gi) --node{$h$} (ghi);
% de 3 ˆ 4
\draw (chi) --node{$g$} (cghi);
\draw (cgi) --node{$h$} (cghi);
\draw (cgh) --node{$i$\;} (cghi);
\draw (ghi) --node{$c$} (cghi);
\draw (ghi) --node{$b$} (bghi);
\draw (ghi) --node{$a$} (aghi);
% de 4 ˆ 5
\draw (cghi) --node{$b$} (bcghi);
\draw (cghi) --node{$a$} (acghi);
\draw (cghi) --node{$j$} (cghij);
\draw (bghi) --node{$a$} (abghi);
\draw (aghi) --node{$c$} (acghi);
\draw (bghi) --node{$c$} (bcghi);
\draw (aghi) --node[pos=0.3]{$b$} (abghi);
% de 5 ˆ 6
\draw (cghij) --node{$a$} (acghij);
\draw (cghij) --node[pos=0.8]{$b$} (bcghij);
\draw (cghij) --node{$f$} (cfghij);
\draw (bcghi) --node[pos=0.8]{$j$} (bcghij);
\draw (bcghi) --node{$a$} (abcghi);
\draw (acghi) --node[pos=0.8]{$j$} (acghij);
\draw (acghi) --node{$b$} (abcghi);
\draw (abghi) --node[pos=0.5]{$c$} (abcghi);
% de 6 ˆ 7
\draw (abcghi) --node{$j$} (abcghij);
\draw (abcghi) --node{$f$} (abcfghi);
\draw (bcghij) --node{$f$} (bcfghij);
\draw (bcghij) --node{$a$} (abcghij);
\draw (acghij) --node{$f$} (acfghij);
\draw (acghij) --node{$b$} (abcghij);
\draw (cfghij) --node{$b$} (bcfghij);
\draw (cfghij) --node{$a$} (acfghij);
% de 7 ˆ 8
\draw (bcfghij) --node{$d$} (bcdfghij);
\draw (bcfghij) --node{$a$} (abcfghij);
\draw (acfghij) --node[pos=0.8]{$d$} (acdfghij);
\draw (acfghij) --node{$b$} (abcfghij);
\draw (abcghij) --node[pos=0.35]{$f$} (abcfghij);
\draw (abcghij) --node[pos=0.8]{$d$} (abcdghij);
\draw (abcfghi) --node[pos=0.8]{$j$} (abcfghij);
% 8 ˆ 9
\draw (abcfghij) --node{$d$} (abcdfghij);
\draw (bcdfghij) --node{$a$} (abcdfghij);
\draw (acdfghij) --node{$b$} (abcdfghij);
\draw (abcdghij) --node{$f$} (abcdfghij);
% de 9 ˆ 10
\draw  (abcdfghij) --node{$e$} (Q);
\end{scope}

\node[anchor=west] at (Q) {\quad$Q$};
\node[anchor=west] at (es) {\quad$\es$};

\begin{scope}[xshift=-1cm,
yshift=-0.7cm,scale=0.2]
\fill[black] (-2,0) rectangle node [anchor=east] {very unlikely\quad~} (0,2) ;
\fill[black!80] ( 0,0) rectangle ( 2,2);
\fill[black!70] ( 2,0) rectangle ( 4,2) ;
\fill[black!60] ( 4,0) rectangle ( 6,2);
\fill[black!50] ( 6,0) rectangle ( 8,2);
\fill[black!40] ( 8,0) rectangle (10,2) ;
\fill[black!30] (10,0) rectangle (12,2);
\fill[black!20] (12,0) rectangle (14,2);
\fill[black!10] (14,0) rectangle node [black,anchor=west] {\quad very likely} (16,2);
\end{scope}

\end{tikzpicture}
\end{center}
\caption[Initial probabilities of all the states in an assessment]{\small\rm Initial probabilities of all the states at the beginning of the assessment. The probabilities are marked by the shading of the circles representing the states. Dark grey means very unlikely, and bright grey very likely.}
\label{initial}
\end{figure}
%%%%%%%%%%%%%%%%%%%%%%%%%% end of figure
We suppose that, at the outset of the algorithm execution, there exists some probability distribution on the collection of 
states\footnote{This is one of the reasons why the algorithm cannot be readily applied in the case of large domains: we cannot manage a probability distribution on a set containing billions of states.}. Such a probability distribution may be inferred from some information on the population that the testee belongs to. Failing such information, we may simply assume that the distribution is uniform.
The assessment is adaptive. On each of its main steps, the algorithm applies a Bayesian type updating operator of the current probability distribution, producing thus for each knowledge state an estimate of  the probability that the student be in this state.
 
We keep illustrating our discussion by our example of the 10 items in Beginning Algebra and give in Figure~\ref{initial} a picture of the learning space at the beginning of the assessment.  Notice that the probability values for the states are suggested by the strength of gray (the lighter the disk is, the higher the probability value is).

Figure~\ref{transition_diagram_ter}
summarizes the sequence of events on each step.  We call such steps \subjindex{trial}\term{trials}.  At the beginning of each trial, the algorithm picks the `most informative item'\footnote{For the technical meaning of this term, see the Questioning Rule on page~\pageref{Questioning Rule}.} to give the student; for doing so, it relies on the current, estimated probabilities of the states.  Next, it randomly chooses an instance of the item.  
For example, if the item selected is 
 \begin{align*}
  a:\quad &\text{Quotients of expressions involving exponents}\\[2mm]
  \noalign{the instance could be}
  &\text{Simplify the expression}\quad \frac{a^4b^5}{5a^6b}\quad \text{as much as possible}.
  \end{align*}
The algorithm then proposes the instance to the student. It records the response and checks the correctness.  To complete the trial, the algorithm uses the information just collected in order to update the probability distribution on the collection of states.  Of course, we still need to specify how the item selection and the probability updating are performed (see the Questioning Rule and the Updating Rule below, page~\pageref{Questioning Rule}).

%%%%%%%%%%%%%%%%%%%%%%%%%% figure
\begin{figure}[h]
\begin{center}
\begin{tikzpicture}[scale=0.8]
%%%%%%%
\node[draw,rectangle] (proba) at (0,0) 
{State probabilities on trial $n$};
\node[draw,rectangle] (instance) at (2.1,-3) 
{Selected instance};
\node[draw,rectangle] (correctness) at (4.2,-6) 
{Response correctness};
\node[draw,rectangle] (new proba) at (0,-10) 
{State probabilities  on trial $n+1$};
\node[rectangle, draw, rounded corners=3mm] (select) at (8,-1.5) 
{Select item instance};
\node[draw,double,rounded corners] (quest rule) at (8,-0.5)
{Questioning Rule};
\node[draw,rectangle, rounded corners=3mm] (check) at (8,-4.5) 
{Check student's response};

\node[rectangle, draw, rounded corners=3mm] (update) at (8,-8.5) 
{Update state probabilities};
\node[draw,double,rounded corners] (updating rule) at (8,-7.5)
{Updating Rule};
%%%%%

\node[draw,circle,fill=white] (junction) at (0,-8.5) {};

\draw[->,thick,>=triangle 45] (proba) -- (instance);
\draw[->,thick,>=triangle 45] (instance) -- (correctness);
\draw (correctness) -- (junction);
\draw (proba) -- (junction);
\draw (instance) -- (junction);
\draw[->,thick,>=triangle 45] (junction) -- (new proba);

\draw[decorate,decoration=snake] (1.05,-1.5) -- (select);
\node[draw,circle,fill=white] at (1.05,-1.5) {};

\draw[decorate,decoration=snake] (3.15,-4.5) -- (check);
\node[draw,circle,fill=white] at (3.15,-4.5) {};

\draw[decorate,decoration=snake] (junction) -- (update);

\draw (quest rule) -- (select);
\draw (updating rule) -- (update);

\end{tikzpicture}
\end{center}
\caption[A representation of the assessment from trial $n$ to trial $n+1$]{Top to bottom, a schematic representation of the assessment from trial $n$ to trial $n+1$: on the left the data acted upon, on the right the general instructions  executed according to specific rules.}
\label{transition_diagram_ter}
\end{figure}
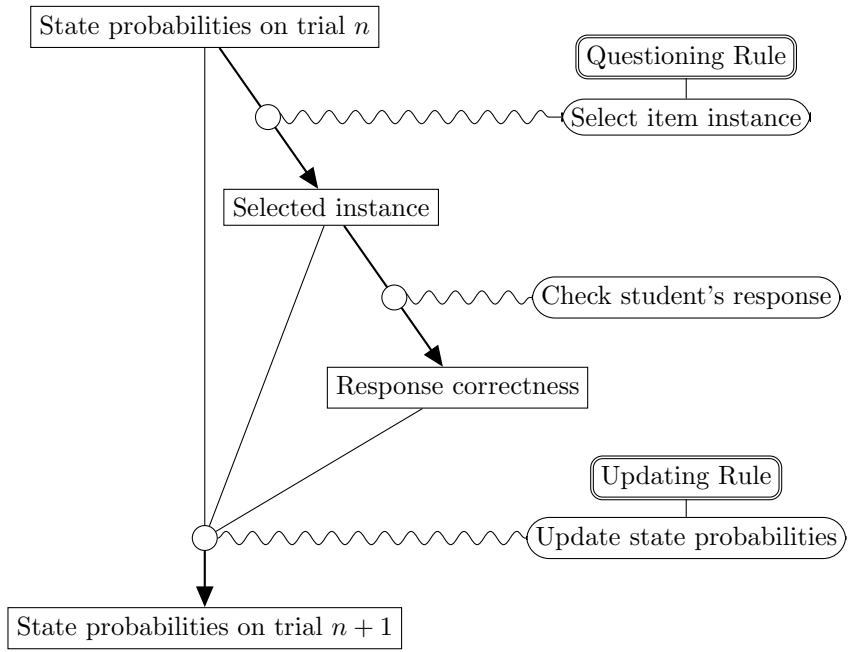
%%%%%%%%%%%%%%%%%%%%%%%%%% end of figure

The basic idea of the assessment algorithm is to ensure that, in the course of the assessment, the probability distribution becomes gradually concentrated on a knowledge state, or on a few knowledge states which are close together\footnote{From the standpoint of their symmetric difference.}. If several states end up with the same high probability, the system chooses randomly between them.  
In practical situations, the assessment terminates in about 25--35 trials with the algorithm we will describe.  We first provide an illustration and then introduce the required notation.

Figure~\ref{initial} pictured a (possible) probability distribution on the set of states at the beginning of the assessment.  The brightest shades indicate the high probability states, the darkest ones represent the lowest probability states.
We infer from the picture that the student may have master about 4--7 items.  On page \pageref{three_parallel}, Figure~\ref{three_parallel} shows the graphs of three later situations.  Suppose the student gives during the first trial a correct response to item~$a$.  The graph on the left shows the updated probability values: note the  increase of the probability values for states containing~$a$, and the decrease of the values for the states not containing~$a$.  
The middle graph sketches the estimated distribution at the end of trial $2$, after the student gave next a false response to item~$f$.
The last graph represents the  typical situation at the end of the assessment: only one state remains with a high probability, which is here $\{a$, $c$, $g$, $h$, $i$, $j\}$.  The algorithm would choose that state as representing the competence of the student in that part of  Beginning Algebra.

%%%%%%%%%%%%%%%%%%%%%%%%%% figure
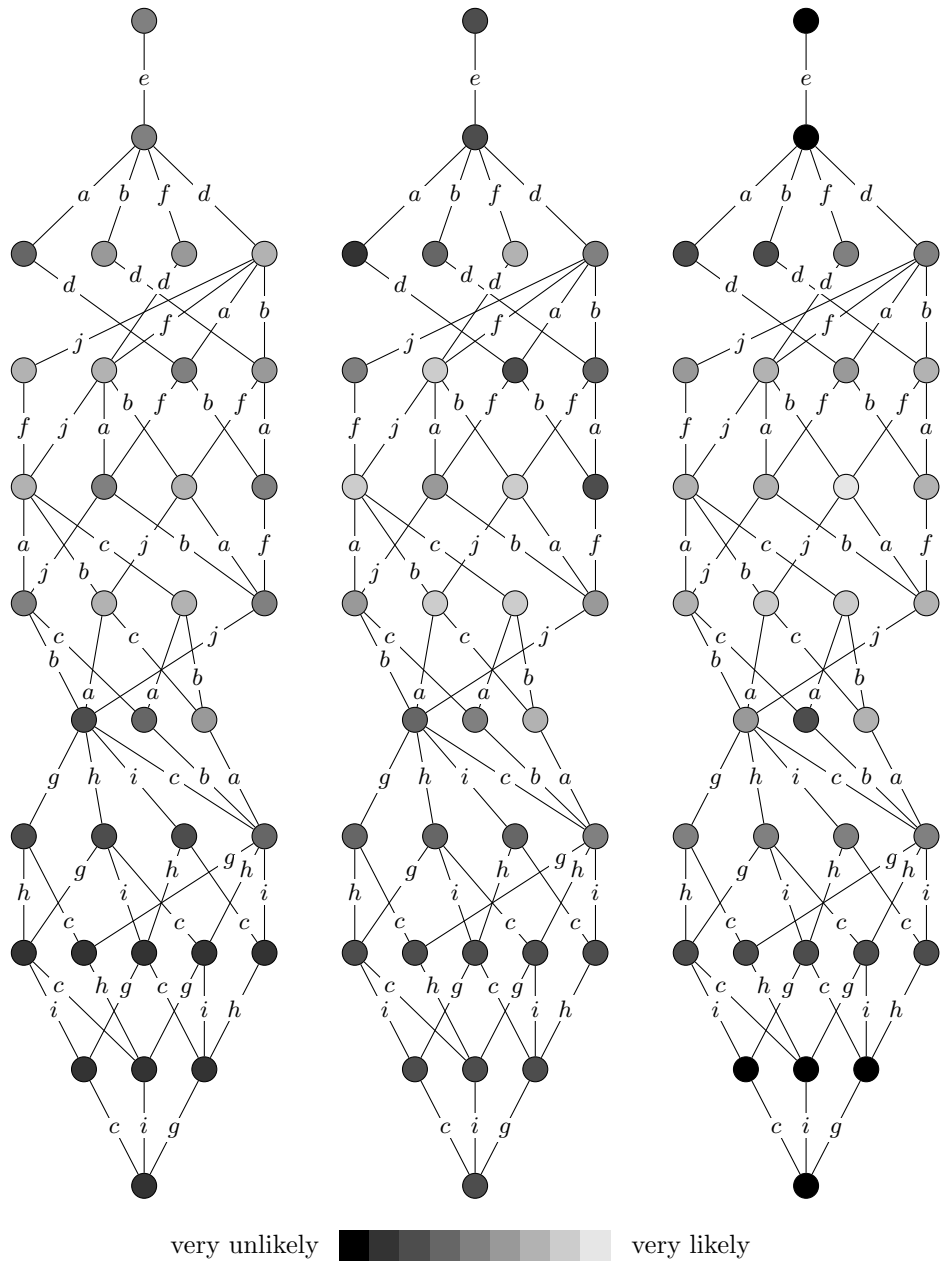
\begin{figure}[t!]
\begin{tikzpicture}
%--------------------- to the left ---------------------------------
\begin{scope}[yscale=1.55, %10-item example
xscale=0.4,every node/.style={circle,scale=1,draw}]%..........
%%% 000
\node[fill=black!80] (es) at ( 0,0) {};
%%%% 111
\node[fill=black!80] (c) at (-2,1) {};
\node[fill=black!80] (i) at ( 0,1) {};
\node[fill=black!80] (g) at ( 2,1) {};
%%%% 222
\node[fill=black!80] (ci) at (-4,2) {};
\node[fill=black!80] (hi) at (-2,2){};
\node[fill=black!80] (cg) at ( 0,2) {};
\node[fill=black!80] (gi) at ( 2,2) {};
\node[fill=black!80] (gh) at ( 4,2) {};
%%%% 333
\node[fill=black!70] (chi) at (-4,3) {};
\node[fill=black!70] (cgi) at (-1.33,3) {};
\node[fill=black!70] (cgh) at ( 1.33,3) {};
\node[fill=black!60] (ghi) at ( 4,3) {};
%%%% 444
\node[fill=black!70] (cghi) at (-2,4) {};
\node[fill=black!60] (bghi) at ( 0,4) {};
\node[fill=black!40] (aghi) at ( 2,4) {};
%%%% 555
\node[fill=black!50] (bcghi) at (-4,5) {};
\node[fill=black!30] (acghi) at (-1.33,5) {};
\node[fill=black!30] (abghi) at ( 1.33,5) {};
\node[fill=black!50] (cghij) at ( 4,5) {};
%%%% 666
\node[fill=black!30] (abcghi) at (-4,6) {};
\node[fill=black!50] (bcghij) at (-1.33,6) {};%
\node[fill=black!30] (acghij) at ( 1.33,6) {};
\node[fill=black!50] (cfghij) at ( 4,6) {};
%%%% 777
\node[fill=black!30] (abcfghi) at (-4,7) {};
\node[fill=black!30] (abcghij) at (-1.33,7) {};%
\node[fill=black!50] (bcfghij) at ( 1.33,7) {};
\node[fill=black!40] (acfghij) at ( 4,7) {};
%%%% 888
\node[fill=black!60] (bcdfghij) at (-4,8) {};%
\node[fill=black!40] (acdfghij) at (-1.33,8) {};
\node[fill=black!40] (abcdghij) at ( 1.33,8) {};
\node[fill=black!30] (abcfghij) at ( 4,8) {};

%%%% 999
\node[fill=black!50] (abcdfghij) at (0,9) {};
%%%%
\node[fill=black!50] (Q) at (0,10) {};
\end{scope}%..........................................................
\begin{scope}[every node/.style={font=\small,
fill=white, pos=0.5,
inner sep=1.7pt,
anchor=mid}]
% de 0 ˆ 1
\draw (es) --node{$c$} (c);
\draw (es) --node{$i$} (i);
\draw (es) --node{$g$} (g);
% de 1 ˆ 2
\draw (c) --node{$i$} (ci);
\draw (c) --node[font=\small,fill=white,pos=0.75,inner sep=1.7pt,
anchor=mid]{$g$} (cg);
\draw (i) --node[font=\small,fill=white,pos=0.75,inner sep=1.7pt,
anchor=mid]{$h$} (hi);
\draw (i) --node[font=\small,fill=white,pos=0.75,inner sep=1.7pt,
anchor=mid]{$g$} (gi);
\draw (g) --node[font=\small,fill=white,pos=0.75,inner sep=1.7pt,
anchor=mid]{$c$} (cg);
\draw (i) --node[font=\small,fill=white,pos=0.75,inner sep=1.7pt,
anchor=mid]{$c$} (ci);
\draw (g) --node{$i$} (gi);
\draw (g) --node{$h$} (gh);
% de 2 ˆ 3
\draw (ci) --node{$h$} (chi);
\draw (hi) --node[font=\small,fill=white,pos=0.2,inner sep=1.8pt,
anchor=mid]{$c$} (chi);
\draw (cg) --node{$i$} (cgi);
\draw (gi) --node[font=\small,fill=white,pos=0.2,inner sep=1.8pt,
anchor=mid]{$c$} (cgi);
\draw (gh) --node[font=\small,fill=white,pos=0.2,inner sep=1.8pt,
anchor=mid]{$c$} (cgh);
\draw (ci) --node[font=\small,fill=white,pos=0.75,inner sep=1.8pt,
anchor=mid]{$g$} (cgi);
\draw (hi) --node[font=\small,fill=white,pos=0.85,inner sep=1.8pt,
anchor=mid]{$g$} (ghi);
\draw (cg) --node[font=\small,fill=white,pos=0.75,inner sep=1.8pt,
anchor=mid]{$h$} (cgh);
\draw (gh) --node{$i$} (ghi);
\draw (gi) --node[font=\small,fill=white,pos=0.75,inner sep=1.8pt,
anchor=mid]{$h$} (ghi);
% de 3 ˆ 4
\draw (chi) --node{$g$} (cghi);
\draw (cgi) --node{$h$} (cghi);
\draw (cgh) --node{$i$} (cghi);
\draw (ghi) --node{$c$} (cghi);
\draw (ghi) --node{$b$} (bghi);
\draw (ghi) --node{$a$} (aghi);
% de 4 ˆ 5
\draw (cghi) --node{$b$} (bcghi);
\draw (cghi) --node[pos=0.15]{$a$} (acghi);
\draw (cghi) --node[near end]{$j$} (cghij);
\draw (bghi) --node[pos=0.15]{$a$} (abghi);
\draw (aghi) --node[near end]{$c$} (acghi);
\draw (bghi) --node[near end]{$c$} (bcghi);
\draw (aghi) --node[pos=0.3]{$b$} (abghi);
% de 5 ˆ 6
\draw (cghij) --node{$a$} (acghij);
\draw (cghij) --node{$b$} (bcghij);
\draw (cghij) --node{$f$} (cfghij);
\draw (bcghi) --node[pos=0.2]{$j$} (bcghij);
\draw (bcghi) --node{$a$} (abcghi);
\draw (acghi) --node[pos=0.5]{$j$} (acghij);
\draw (acghi) --node[pos=0.2]{$b$} (abcghi);
\draw (abghi) --node{$c$} (abcghi);
% de 6 ˆ 7
\draw (abcghi) --node{$j$} (abcghij);
\draw (abcghi) --node{$f$} (abcfghi);
\draw (bcghij) --node[pos=0.75]{$f$} (bcfghij);
\draw (bcghij) --node{$a$} (abcghij);
\draw (acghij) --node[pos=0.75]{$f$} (acfghij);
\draw (acghij) --node[pos=0.75]{$b$} (abcghij);
\draw (cfghij) --node[pos=0.75]{$b$} (bcfghij);
\draw (cfghij) --node{$a$} (acfghij);
% de 7 ˆ 8
\draw (bcfghij) --node[pos=0.75]{$d$} (bcdfghij);
\draw (bcfghij) --node{$a$} (abcfghij);
\draw (acfghij) --node[pos=0.85]{$d$} (acdfghij);
\draw (acfghij) --node{$b$} (abcfghij);
\draw (abcghij) --node[pos=0.37]{$f$} (abcfghij);
\draw (abcghij) --node[pos=0.8]{$d$} (abcdghij);
\draw (abcfghi) --node[pos=0.2]{$j$} (abcfghij);
% 8 ˆ 9
\draw (abcfghij) --node{$d$} (abcdfghij);
\draw (bcdfghij) --node{$a$} (abcdfghij);
\draw (acdfghij) --node{$b$} (abcdfghij);
\draw (abcdghij) --node{$f$} (abcdfghij);
% de 9 ˆ 10
\draw  (abcdfghij) --node{$e$} (Q);
\end{scope}
%--------------------- in the middle --------------------------------
\begin{scope}[yscale=1.55,xscale=0.4,every node/.style={circle,scale=1,draw},xshift=11cm]%..........
%%% 000
\node[fill=black!70] (es) at ( 0,0) {};
%%%% 111
\node[fill=black!70] (c) at (-2,1) {};
\node[fill=black!70] (i) at ( 0,1) {};
\node[fill=black!70] (g) at ( 2,1) {};
%%%% 222
\node[fill=black!70] (ci) at (-4,2) {};
\node[fill=black!70] (gi) at ( 2,2) {};
\node[fill=black!70] (hi) at ( -2,2){};
\node[fill=black!70] (gh) at ( 4,2) {};
\node[fill=black!70] (cg) at ( 0,2) {};
%%%% 333
\node[fill=black!60] (chi) at (-4,3) {};
\node[fill=black!60] (cgi) at (-1.33,3) {};
\node[fill=black!60] (cgh) at (1.33,3) {};
\node[fill=black!50] (ghi) at ( 4,3) {};
%%%% 444
\node[fill=black!60] (cghi) at (-2,4) {};
\node[fill=black!50] (bghi) at ( 0,4) {};
\node[fill=black!30] (aghi) at ( 2,4) {};
%%%% 555
\node[fill=black!40] (bcghi) at (-4,5) {};%
\node[fill=black!20] (acghi) at (-1.33,5) {}; 
\node[fill=black!20] (abghi) at ( 1.33,5) {};
\node[fill=black!40] (cghij) at ( 4,5) {};
%%%% 666
\node[fill=black!20] (abcghi) at (-4,6) {};
\node[fill=black!40] (bcghij) at (-1.33,6) {};%
\node[fill=black!20] (acghij) at ( 1.33,6) {};
\node[fill=black!70] (cfghij) at ( 4,6) {};
%%%% 777
\node[fill=black!50] (abcfghi) at (-4,7) {};
\node[fill=black!20] (abcghij) at (-1.33,7) {};%
\node[fill=black!70] (bcfghij) at ( 1.33,7) {};
\node[fill=black!60] (acfghij) at ( 4,7) {};
%%%% 888
\node[fill=black!80] (bcdfghij) at (-4,8) {};%
\node[fill=black!60] (acdfghij) at ( -1.33,8) {};
\node[fill=black!30] (abcdghij) at ( 1.33,8) {};
\node[fill=black!50] (abcfghij) at ( 4,8) {};
%%%% 999
\node[fill=black!70] (abcdfghij) at (0,9) {};
%%%%
\node[fill=black!70] (Q) at (0,10) {};
\end{scope}%..........................................................
\begin{scope}[every node/.style={font=\small,
fill=white, pos=0.5,
inner sep=1.7pt,
anchor=mid}]
% de 0 ˆ 1
\draw (es) --node{$c$} (c);
\draw (es) --node{$i$} (i);
\draw (es) --node{$g$} (g);
% de 1 ˆ 2
\draw (c) --node{$i$} (ci);
\draw (c) --node[font=\small,fill=white,pos=0.75,inner sep=1.7pt,
anchor=mid]{$g$} (cg);
\draw (i) --node[font=\small,fill=white,pos=0.75,inner sep=1.7pt,
anchor=mid]{$h$} (hi);
\draw (i) --node[font=\small,fill=white,pos=0.75,inner sep=1.7pt,
anchor=mid]{$g$} (gi);
\draw (g) --node[font=\small,fill=white,pos=0.75,inner sep=1.7pt,
anchor=mid]{$c$} (cg);
\draw (i) --node[font=\small,fill=white,pos=0.75,inner sep=1.7pt,
anchor=mid]{$c$} (ci);
\draw (g) --node{$i$} (gi);
\draw (g) --node{$h$} (gh);
% de 2 ˆ 3
\draw (ci) --node{$h$} (chi);
\draw (hi) --node[font=\small,fill=white,pos=0.2,inner sep=1.8pt,
anchor=mid]{$c$} (chi);
\draw (cg) --node{$i$} (cgi);
\draw (gi) --node[font=\small,fill=white,pos=0.2,inner sep=1.8pt,
anchor=mid]{$c$} (cgi);
\draw (gh) --node[font=\small,fill=white,pos=0.2,inner sep=1.8pt,
anchor=mid]{$c$} (cgh);
\draw (ci) --node[font=\small,fill=white,pos=0.75,inner sep=1.8pt,
anchor=mid]{$g$} (cgi);
\draw (hi) --node[font=\small,fill=white,pos=0.85,inner sep=1.8pt,
anchor=mid]{$g$} (ghi);
\draw (cg) --node[font=\small,fill=white,pos=0.75,inner sep=1.8pt,
anchor=mid]{$h$} (cgh);
\draw (gh) --node{$i$} (ghi);
\draw (gi) --node[font=\small,fill=white,pos=0.75,inner sep=1.8pt,
anchor=mid]{$h$} (ghi);
% de 3 ˆ 4
\draw (chi) --node{$g$} (cghi);
\draw (cgi) --node{$h$} (cghi);
\draw (cgh) --node{$i$} (cghi);
\draw (ghi) --node{$c$} (cghi);
\draw (ghi) --node{$b$} (bghi);
\draw (ghi) --node{$a$} (aghi);
% de 4 ˆ 5
\draw (cghi) --node{$b$} (bcghi);
\draw (cghi) --node[pos=0.15]{$a$} (acghi);
\draw (cghi) --node[near end]{$j$} (cghij);
\draw (bghi) --node[pos=0.15]{$a$} (abghi);
\draw (aghi) --node[near end]{$c$} (acghi);
\draw (bghi) --node[near end]{$c$} (bcghi);
\draw (aghi) --node[pos=0.3]{$b$} (abghi);
% de 5 ˆ 6
\draw (cghij) --node{$a$} (acghij);
\draw (cghij) --node{$b$} (bcghij);
\draw (cghij) --node{$f$} (cfghij);
\draw (bcghi) --node[pos=0.2]{$j$} (bcghij);
\draw (bcghi) --node{$a$} (abcghi);
\draw (acghi) --node[pos=0.5]{$j$} (acghij);
\draw (acghi) --node[pos=0.2]{$b$} (abcghi);
\draw (abghi) --node{$c$} (abcghi);
% de 6 ˆ 7
\draw (abcghi) --node{$j$} (abcghij);
\draw (abcghi) --node{$f$} (abcfghi);
\draw (bcghij) --node[pos=0.75]{$f$} (bcfghij);
\draw (bcghij) --node{$a$} (abcghij);
\draw (acghij) --node[pos=0.75]{$f$} (acfghij);
\draw (acghij) --node[pos=0.75]{$b$} (abcghij);
\draw (cfghij) --node[pos=0.75]{$b$} (bcfghij);
\draw (cfghij) --node{$a$} (acfghij);
% de 7 ˆ 8
\draw (bcfghij) --node[pos=0.75]{$d$} (bcdfghij);
\draw (bcfghij) --node{$a$} (abcfghij);
\draw (acfghij) --node[pos=0.85]{$d$} (acdfghij);
\draw (acfghij) --node{$b$} (abcfghij);
\draw (abcghij) --node[pos=0.37]{$f$} (abcfghij);
\draw (abcghij) --node[pos=0.8]{$d$} (abcdghij);
\draw (abcfghi) --node[pos=0.2]{$j$} (abcfghij);
% 8 ˆ 9
\draw (abcfghij) --node{$d$} (abcdfghij);
\draw (bcdfghij) --node{$a$} (abcdfghij);
\draw (acdfghij) --node{$b$} (abcdfghij);
\draw (abcdghij) --node{$f$} (abcdfghij);
% de 9 ˆ 10
\draw  (abcdfghij) --node{$e$} (Q);
\end{scope}
%--------------------- to the right ----------------------------------
\begin{scope}[yscale=1.55,xscale=0.4,every node/.style={circle,scale=1,draw},xshift=22cm]%..........
%%% 000
\node[fill=black!100] (es) at     ( 0,0) {};
%%%% 111
\node[fill=black!100] (c) at     (-2,1) {};
\node[fill=black!100] (i) at     ( 0,1) {};
\node[fill=black!100] (g) at     ( 2,1) {};
%%%% 222
\node[fill=black!70] (ci) at     (-4,2) {};
\node[fill=black!70] (hi) at     (-2,2){};
\node[fill=black!70] (cg) at     ( 0,2) {};
\node[fill=black!70] (gi) at     ( 2,2) {};
\node[fill=black!70] (gh) at     ( 4,2) {};

%%%% 333
\node[fill=black!50] (chi) at     (-4,3) {};
\node[fill=black!50] (cgi) at     (-1.33,3) {};
\node[fill=black!50] (cgh) at     ( 1.33,3) {};
\node[fill=black!50] (ghi) at     ( 4,3) {};
%%%% 444
\node[fill=black!40] (cghi) at     (-2,4) {};
\node[fill=black!70] (bghi) at     ( 0,4) {};
\node[fill=black!30] (aghi) at     ( 2,4) {};
%%%% 555
\node[fill=black!30] (bcghi) at     (-4,5) {};
\node[fill=black!20] (acghi) at     (-1.33,5) {};
\node[fill=black!20] (abghi) at     ( 1.33,5) {};
\node[fill=black!30] (cghij) at     ( 4,5) {};
%%%% 666
\node[fill=black!30] (abcghi) at     (-4,6) {};
\node[fill=black!30] (bcghij) at     (-1.33,6) {};
\node[fill=black!10] (acghij) at     ( 1.33,6) {};
\node[fill=black!30] (cfghij) at     ( 4,6) {};
%%%% 777
\node[fill=black!40] (abcfghi) at     (-4,7) {};
\node[fill=black!30] (abcghij) at     (-1.33,7) {};
\node[fill=black!40] (bcfghij) at     ( 1.33,7) {};
\node[fill=black!30] (acfghij) at     ( 4,7) {};
%%%% 888
\node[fill=black!70] (bcdfghij) at     (-4,8) {};
\node[fill=black!70] (acdfghij) at     (-1.33,8) {};
\node[fill=black!50] (abcdghij) at     ( 1.33,8) {};
\node[fill=black!50] (abcfghij) at     ( 4,8) {};
%%%% 999
\node[fill=black!100] (abcdfghij) at     (0,9) {};
%%%%
\node[fill=black!100] (Q) at     (0,10) {};
\end{scope}%..........................................................
\begin{scope}[every node/.style={font=\small,
fill=white, pos=0.5,
inner sep=1.7pt,
anchor=mid}]
% de 0 ˆ 1
\draw (es) --node{$c$} (c);
\draw (es) --node{$i$} (i);
\draw (es) --node{$g$} (g);
% de 1 ˆ 2
\draw (c) --node{$i$} (ci);
\draw (c) --node[font=\small,fill=white,pos=0.75,inner sep=1.7pt,
anchor=mid]{$g$} (cg);
\draw (i) --node[font=\small,fill=white,pos=0.75,inner sep=1.7pt,
anchor=mid]{$h$} (hi);
\draw (i) --node[font=\small,fill=white,pos=0.75,inner sep=1.7pt,
anchor=mid]{$g$} (gi);
\draw (g) --node[font=\small,fill=white,pos=0.75,inner sep=1.7pt,
anchor=mid]{$c$} (cg);
\draw (i) --node[font=\small,fill=white,pos=0.75,inner sep=1.7pt,
anchor=mid]{$c$} (ci);
\draw (g) --node{$i$} (gi);
\draw (g) --node{$h$} (gh);
% de 2 ˆ 3
\draw (ci) --node{$h$} (chi);
\draw (hi) --node[font=\small,fill=white,pos=0.2,inner sep=1.8pt,
anchor=mid]{$c$} (chi);
\draw (cg) --node{$i$} (cgi);
\draw (gi) --node[font=\small,fill=white,pos=0.2,inner sep=1.8pt,
anchor=mid]{$c$} (cgi);
\draw (gh) --node[font=\small,fill=white,pos=0.2,inner sep=1.8pt,
anchor=mid]{$c$} (cgh);
\draw (ci) --node[font=\small,fill=white,pos=0.75,inner sep=1.8pt,
anchor=mid]{$g$} (cgi);
\draw (hi) --node[font=\small,fill=white,pos=0.85,inner sep=1.8pt,
anchor=mid]{$g$} (ghi);
\draw (cg) --node[font=\small,fill=white,pos=0.75,inner sep=1.8pt,
anchor=mid]{$h$} (cgh);
\draw (gh) --node{$i$} (ghi);
\draw (gi) --node[font=\small,fill=white,pos=0.75,inner sep=1.8pt,
anchor=mid]{$h$} (ghi);
% de 3 ˆ 4
\draw (chi) --node{$g$} (cghi);
\draw (cgi) --node{$h$} (cghi);
\draw (cgh) --node{$i$} (cghi);
\draw (ghi) --node{$c$} (cghi);
\draw (ghi) --node{$b$} (bghi);
\draw (ghi) --node{$a$} (aghi);
% de 4 ˆ 5
\draw (cghi) --node{$b$} (bcghi);
\draw (cghi) --node[pos=0.15]{$a$} (acghi);
\draw (cghi) --node[near end]{$j$} (cghij);
\draw (bghi) --node[pos=0.15]{$a$} (abghi);
\draw (aghi) --node[near end]{$c$} (acghi);
\draw (bghi) --node[near end]{$c$} (bcghi);
\draw (aghi) --node[pos=0.3]{$b$} (abghi);
% de 5 ˆ 6
\draw (cghij) --node{$a$} (acghij);
\draw (cghij) --node{$b$} (bcghij);
\draw (cghij) --node{$f$} (cfghij);
\draw (bcghi) --node[pos=0.2]{$j$} (bcghij);
\draw (bcghi) --node{$a$} (abcghi);
\draw (acghi) --node[pos=0.5]{$j$} (acghij);
\draw (acghi) --node[pos=0.2]{$b$} (abcghi);
\draw (abghi) --node{$c$} (abcghi);
% de 6 ˆ 7
\draw (abcghi) --node{$j$} (abcghij);
\draw (abcghi) --node{$f$} (abcfghi);
\draw (bcghij) --node[pos=0.75]{$f$} (bcfghij);
\draw (bcghij) --node{$a$} (abcghij);
\draw (acghij) --node[pos=0.75]{$f$} (acfghij);
\draw (acghij) --node[pos=0.75]{$b$} (abcghij);
\draw (cfghij) --node[pos=0.75]{$b$} (bcfghij);
\draw (cfghij) --node{$a$} (acfghij);
% de 7 ˆ 8
\draw (bcfghij) --node[pos=0.75]{$d$} (bcdfghij);
\draw (bcfghij) --node{$a$} (abcfghij);
\draw (acfghij) --node[pos=0.85]{$d$} (acdfghij);
\draw (acfghij) --node{$b$} (abcfghij);
\draw (abcghij) --node[pos=0.37]{$f$} (abcfghij);
\draw (abcghij) --node[pos=0.8]{$d$} (abcdghij);
\draw (abcfghi) --node[pos=0.2]{$j$} (abcfghij);
% 8 ˆ 9
\draw (abcfghij) --node{$d$} (abcdfghij);
\draw (bcdfghij) --node{$a$} (abcdfghij);
\draw (acdfghij) --node{$b$} (abcdfghij);
\draw (abcdghij) --node{$f$} (abcdfghij);
% de 9 ˆ 10
\draw  (abcdfghij) --node{$e$} (Q);
\end{scope}
\begin{scope}[xshift=3cm,yshift=-1cm,scale=0.2]
\fill[black] (-2,0) rectangle node [anchor=east] {very unlikely\quad~} (0,2) ;
\fill[black!80] ( 0,0) rectangle ( 2,2);
\fill[black!70] ( 2,0) rectangle ( 4,2) ;
\fill[black!60] ( 4,0) rectangle ( 6,2);
\fill[black!50] ( 6,0) rectangle ( 8,2);
\fill[black!40] ( 8,0) rectangle (10,2) ;
\fill[black!30] (10,0) rectangle (12,2);
\fill[black!20] (12,0) rectangle (14,2);
\fill[black!10] (14,0) rectangle node [black,anchor=west] {\quad very likely} (16,2);
\end{scope}
\end{tikzpicture}
\caption[Three successive situations along an assessment]{Left to right, three successive situations along the assessment: 
(i)~after a correct response to $a$; (ii)~next, after a false answer to $f$; (iii)~later, at the end of the assessment.}
\label{three_parallel} 
\end{figure}
%%%%%%%%%%%%%%%%%%%%%%%%%% end of figure

We now give a precise mathematical meaning to the concepts outlined above, starting with a list of notations based on a knowledge structure $(Q,\KKK)$.

%\newpage % 
%\vskip3mm

\clearpage

\noindent{\bf Mathematical notation.}\label{math_concepts}\newline
\begin{tabular}{l@{\quad}p{109mm}}
$n$ & the step number, or \indexedterm{trial number}, with $n=1$, $2$, \ldots;\\
$\KKK_q$ & the subcollection of $\KKK$ formed by the states containing $q$;\\
$\Lambda_+$ & the set of all positive probability distributions on $\KKK$;\\
$\mathbf L_n$ & a positive random probability distribution on $\KKK$:\newline we have $\mathbf L_n=L_n\in\Lambda_+$ if $L_n$ is the probability distribution on $\KKK$ at the beginning of trial $n$ (so $L_n>0$);\\
$\mathbf L_n(K)$ & a random variable (r.v.)~evaluating 
the probability of state $K$ on trial $n$;\\
$\mathbf Q_n$ & a r.v.~representing the question asked on trial $n$:\newline we have $\mathbf Q_n = q$ if $q$ in $Q$ is the question asked on trial $n$;\\
$\mathbf R_n$ & a r.v.~coding the response on trial $n$:\newline 
$\mathbf R_n = 
\begin{cases} 
1 & \text{if the response is correct},\\
0 & \text{otherwise};
\end{cases}$\\
$\mathbf W_n$ & the random history of the process from trial $1$ to trial $n$;\\
$\iota_A$ & the indicator function of a set $A$:\quad $\iota_A(q) = 
\begin{cases}  1 &\text{if } q\in A,\\
               0 &\text{if }  q\notin A;
  \end{cases}$\\
$\zeta_{q,r}$ & a collection of parameters defined for $q \in Q$, $r \in \{0,1\}$ and satisfying $1 < \zeta_{q,r}$ (see Updating Rule~[U] below). 
\end{tabular}

\vskip2mm

The formal rules governing the assessment algorithm are as follows (where $n$ is the trial number).\\[1mm]

Using the `Questioning Rule', the algorithm picks a most informative item, that is, an item which, on the basis of the current probability distribution on the set of states, has a probability of being responded to correctly as close to $.5$ as possible.  There may be more than one such item, in which case a uniform, random selection is made.
Formally, in terms of the actual probability distribution $L_n$:\\[1mm]

\noindent[Q]\,\textbf{Questioning Rule.}\label{Questioning Rule}~For all $q \in Q$ and all positive integers,
\begin{equation}\label{random_q_picking}
\mathbb{P} (\mathbf Q_n = q \;\st\; \mathbf L_n , \mathbf W_{n-1} ) =  \frac{\iota_{S(L_n)}(q)}{|S(L_n)|}
\end{equation}
where  $S(L_n)$ is the subset of $Q$ containing all those
items  $q$ minimizing  
$|2L_n(\KKK_q) -1|.$  Under this questioning rule, which is called \subjindex{half-split rule}\term{half-split}, we must have $\mathbf Q_n \in S(L_n)$ with a
probability equal to one.  The questions in the set $S(L_n)$ are then
chosen with equal probability. \\[1mm]

The `Response Rule' formalized below states that the student's response to an instance of an item is correct with probability 1 if the item belongs to the student's knowledge state $K_0$, and false with probability 1 otherwise.  While this rule plays no role in the assessment algorithm per se, it is essential in some simulation, and so is its elaborated version in the form of the `Modified Response Rule' (see page~\pageref{Modified Response Rule} below).  These rules are used, in particular, in the `Extra problem method' (see Subsection 0.18.1, page~\pageref{extra_problem}).\\[1mm]

\noindent[R]\,\textbf{Response Rule.}~There is some state $K_0$ such that, for all positive integers $n$,
\begin{equation}
\mathbb{P}(\mathbf R_n = \iota_{K_0}(q) \;\st\; \mathbf Q_n = q,\mathbf L_n,\mathbf W_{n-1})= 1\,.
\end{equation}
The state $K_0$ is the \indexedterm{latent state} representing the set of all the items currently mastered by the student.  It is the state that the assessment algorithm aims to uncover.\\[1mm]

\noindent[U]\,\textbf{Updating Rule.}~We have $\mathbb{P} (\mathbf L_1 = L) = 1$ (with $L$ the initial probability distribution).  Moreover there are real parameters $\zeta_{q,r}$ (where $q \in Q$, $r \in \{0,1\}$ and $1 < \zeta_{q,r}$) such that for any positive integer $n$, if $\mathbf L_n = L_n$, $\mathbf Q_n = q$, $\mathbf R_n = r$ and 
\begin{equation}\label{cmeqg}
\zeta_{q,r}^K = 
\begin{cases}
   1 \quad    &\text{if } \iota_K(q) \neq r,\\
   \zeta_{q,r}&\text{if } \iota_K(q) = r,
\end{cases}
\end{equation}
then
\begin{equation}
\label{cmeqgA}
L_{n+1}(K) \;=\; \frac {\zeta_{q,r}^K L_n(K)}
{\sum_{K'\in \KKK}\zeta_{q,r}^{K'} L_n(K')}.
\end{equation}
\nopagebreak[4]
This updating rule is called \subjindex{multiplicative updating rule}
\term{multiplicative with parameters} $\zeta_{q,r}$.
The operator mapping $L_n$ to $L_{n+1}$ has two important, related properties. First, it is commutative: the order of the successive items of the assessment has no effect on the result, that is, on the final probability distribution on the set of states.  Second, the operator is essentially Bayesian (this was proved by \authindex{Koppen, M.}Mathieu Koppen\footnote{Personnal communication; see Subsection 13.4.5 in \authindex{Doignon, J.-P.}\authindex{Falmagne, J.-Cl.}\citet{Falmagne_Doignon_LS}.}).\\[1mm]

It can be shown that, under the above Rules~[Q], [R] and [U], the latent state is recoverable in the sense that:
\begin{equation}
\mathbf L_n(K_0) \overset {\text{a.s.}} {\longrightarrow} 1.
\end{equation}
We recall that ``a.s.'' stands for ``almost surely''.
For a proof, see Theorem  1.6.7 in \authindex{Doignon, J.-P.}\authindex{Falmagne, J.-Cl.}\citet{Falmagne_Doignon_LS}. Note that this theorem assumes that the careless errors and the lucky guesses have probability zero (the straight case as in Definition~\ref{def straight ppks}). 

While Rules~[U], [Q], and [R] are fundamental to the assessment,
a straightforward implementation of these rules in an assessment engine is not feasible for two reasons.
One is that students often make careless errors. A solution is to amend Rule~[R] by introducing a `careless error parameter'. This would result, for example, in the modified rule:\\[1mm]

\noindent
[R']\,\textbf{Modified Response Rule.}\label{Modified Response Rule}~For all positive integers $n$,
\begin{equation}
\Prob(\mathbf R_n = \iota_{K_0}(q) \;\st\; \mathbf Q_n = q,\mathbf L_n,\mathbf W_{n-1})=\begin{cases} 1-\beta_q \quad&\text{if } q\in K_0\\
0&\text{if }q\notin K_0\,,
\end{cases}
\end{equation}
in which $\beta_q$ is the probability of making an error in responding to the item~$q$ which lies in the latent state~$K_0$. The parameters $\beta_q$ can be estimated from the data (see Subsection~\ref{extra_problem}, page \pageref{extra_problem}).
In some cases, `lucky guess' parameters may also be used, for example to deal with cases of the multiple choice paradigm. 
However, in real life situations, the occurrence of lucky guesses would render the result of the assessment so unreliable as to render it practically useless.
Multiple choice paradigms are cheap but their results are questionable and they should be avoided.
The \subjindex{\aleks}\aleks{} assessment system, which is based on the three rules~[U], [Q], and [R'], rarely use multiple choice.  When it does, the number of possible responses is so large that the probability of a lucky guess is negligible.

The second reason is that, in many real life applications, the domain formalizing an actual curriculum is typically very large, containing several hundred items.  The learning space can still be constructed but its collection of states is huge, counting many million states.  It means that the assessment algorithm cannot proceed in the straightforward manner outlined above: the probability distribution on such a large collection of states cannot be managed.  We now expose a solution.

%---------------------------------------------------------------------
\subsection{The parallel algorithm for large domains}
\label{parallel}
In case the domain $Q$ of the `parent learning space' $\LLL$ is large, the algorithm first partitions $Q$ into $N$ subdomains $Q^1$, $Q^2$, \dots, $Q^N$.  Here, the subdomains are approximately of equal sizes, and they are manageable because $N$ is chosen sufficiently large.  They are representatives of the parent domain $Q$, in the sense, for example, that each one of them could serve as a placement test. 
By the Projection Theorem~\ref{projection_theo}(i) (page~\pageref{projection_theo}), these $N$ subdomains determine $N$ projections $\LLL^1$, $\LLL^2$, \dots, $\LLL^N$ which are all learning spaces.  The algorithm then manages $N$ simultaneous\footnote{or maybe better said: alternating.} assessments on the $N$ projected learning spaces $\left(Q^i,\LLL^i\right)$ (where $i=1$, $2$, \dots, $N$).  The key difficulty is to immediately transfer information obtained from a response to some item~$q$ in the domain $Q^j$ of the projected learning space $\LLL^j$, to all the other $N-1$ learning spaces with appropriate updates of their state probabilities.  

The main features of the algorithm are described in \authindex{Falmagne, J.-Cl.}\authindex{Albert, D.}\authindex{Doble, C.W.}\authindex{Eppstein, D.}\authindex{Hu, X.}\citet[][see Section 8.8, pages 143--144]{Falmagne_Albert_Doble_Eppstein_Hu}. The basic ideas are as follows.
\begin{enumerate}[\quad\rm1)] 
\item On each trial, scan the $N$ learning spaces and choose the ``most informative'' item to be presented to the student. Suppose that this item is $q$, belonging to the subdomain $Q^i$ of the projected learning space $(Q^i,\LLL^i)$.
\item Record the student's response to item $q$.  Update the state probabilities of the learning space $\LLL^i$. 
\item Update the state probabilities of all the other learning subspaces $\LLL^j$, $j\neq i$. This is achieved by temporarily adding item $q$ to the other learning spaces $\LLL^j$
and working with extended distributions (cf.~Definition~\ref{def_prob_extension}).
\item At the end of the assessment, combine the information gathered in all the $N$ learning subspaces into a final knowledge state in $\LLL$. 
For details, see \authindex{Falmagne, J.-Cl.}\authindex{Albert, D.}\authindex{Doble, C.W.}\authindex{Eppstein, D.}\authindex{Hu, X.}\citet[][Subsection 8.8.1, page~144]{Falmagne_Albert_Doble_Eppstein_Hu}.
\end{enumerate}

%%%%%%%%%%%%%%%%%%%%%%%%%%%%%%%%%%%%%%%%%%%%%%%%%%%%%%%%%%%%%%%%%%%%
\section[About Building Knowledge Spaces or Learning Spaces]{About Building Knowledge Spaces or Learning Spaces
\sectionmark{About Building Knowledge Spaces or Learning Spaces}}
\sectionmark{About Building Knowledge Spaces or Learning Spaces}
\label{sec_Building} 
Building a knowledge space on a given domain is a highly demanding, data driven enterprise.  Building a learning 
space is an even much more intricate process.  Chapters 15 and 16 of \authindex{Doignon, J.-P.}\authindex{Falmagne, J.-Cl.}\citet{Falmagne_Doignon_LS} describe a method in details. We only sketch the main ideas here, proceeding in two steps.
We first describe some theoretical constructions which are instrumental for building knowledge spaces. We then show how these tools can be amended in the case of learning spaces.

The basic algorithm is the `\query' routine, which is due to \authindex{Koppen, M.}\citet{koppen_93} and  \authindex{M\"uller, C.}\citet{mueller:roskam} \authindex{Dowling, C.E.}\citep[see also][]{dowling:fischer}.  
The routine manages the following type of queries, which are either put to expert teachers or have responses inferred from assessment data. 

\begin{enumerate}[{[Q]}]
\item \textsl{Suppose that a student under examination has just provided wrong responses to all the items in some set $A$. Is it practically certain that this student will also fail item~$q$? Assume that the conditions are ideal in the sense that errors and lucky guesses are excluded.}
\end{enumerate}
 
Such a \indexedterm{query} is summarized as the pair $(A,q)$, with $\es \neq A\subset Q$ and $q\in Q$.  The \query{} routine gradually builds the relation $\PPP$ from $2^Q$ to $Q$ consisting of all the queries $(A,q)$ which receive a positive response.  For $m = |Q|$, there are $m\left(2^{m-1}-1\right)$ useful queries (asking a query $(A,q)$ with $q\in A$ is useless: the answer must be affirmative).  So their number is superexponential in $m$.  Fortunately, in most cases, only a small fraction of them need to be asked: as we explain below, the responses to some queries may be inferred from the response to previously asked queries. The importance of the queries regarding the problem at hand---building a knowledge space or a learning space---lies in the following theorem \authindex{Koppen, M.}\authindex{Doignon, J.-P.}\citep{koppen_doignon_1990}.

\begin{theorem}\label{thm_K_D}
\sl For a finite domain $Q$, there is a one-to-one correspondence between the collection of knowledge spaces $\KKK$ on $Q$ and the set of relations $\PPP$ from $2^Q\setminus\{\es\}$ to $Q$ satisfying for all $q$ in $Q$ and $A$, $B$ in $2^Q\setminus\{\es\}$:
\begin{enumerate}[\quad\rm(i)]
\item if $q\in A$, then $A \PPP q$;
\item if $A \PPP b$ for all $b$ in $B$ and $B \PPP q$, then $A \PPP q$.
\end{enumerate}
One such correspondence is defined by the two equivalences, where $r\in Q$, $K$, $B\subseteq Q$,  
\begin{align}\label{eq_entailment1}
K \in \KKK &\EQ \big(\forall (A,q)\in \PPP:\, A \cap K = \es \implies q \notin K\big),\\
\label{eq_entailment2}
(B,r)\in \PPP &\EQ \big(\forall K\in\KKK:\, B \cap K = \es \implies r \notin K\big).
\end{align}
\end{theorem}  

In view of Formula~\eqref{eq_entailment1}, a positive answer to the query $(A,q)$ may lead to remove some subsets from the collection of potential states (namely, all the subsets $K$ satisfying both $A \cap K = \es$ and $q \in K$).  Negating both sides of Formula~\eqref{eq_entailment1} gives us the equivalent formula
\begin{equation}\label{eq_entailment1_minus}
K \notin \KKK\EQ \big(\exists (A,q)\in \PPP: \; A \cap K = \es \text{ and } q \in K\big),
\end{equation} 
which clearly shows that any non-state $K$ is ruled out by a positive answer to some query.  Thus Theorem~\ref{thm_K_D} is instrumental for building a knowledge space---which may or may not be a learning space. We outline below (starting on page~\pageref{page_adapted}) how we can adapt the implementation of Theorem~\ref{thm_K_D} so that 
learning spaces are generated. 

In the \query{} procedure, Block~$j$ consists of all the queries $(A,q)$ with $j=|A|$.  The procedure starts with Block~1, followed by Block~2, etc.
In principle an expert teacher is able to provide the answer to any query in Block~1, because such a query takes the following simple form.  

\begin{enumerate}[{[Q1]}]
\item \textsl{Suppose that a student under examination has just provided a false response to question $q$. Is it practically certain that this student will also fail item~$r$? Assume that the conditions are ideal in the sense that errors and lucky guesses are excluded.}
\end{enumerate}
Collecting the responses to such queries $(\{r\},q)$, or $(r,q)$ for short, with $r$, $q$ in $Q$ and $r \neq q$ amounts to constructing a relation $R$ on the set $Q$.  But not all such queries need to be asked. Assuming that the expert is consistent, we can infer the responses of some queries from other queries. For instance if the responses to the queries $(r,t)$ and $(t,q)$ are positive, then the response to $(r,q)$ should also be positive.  In general, we only need to ask the expert a relatively small set of queries of the type [Q1].  Their responses yield some relation~$R$. The  \indexedterm{transitive closure}\footnote{That is, the smallest transitive relation on $Q$ that includes $R$.} $t(R)$ of the relation $R$ is a quasi order on~$Q$.  By Birkhoff's Theorem~\ref{Birkhoff} (page \pageref{Birkhoff}), \authindex{Birkhoff, G.} the quasi order $t(R)$ uniquely defines a quasi ordinal knowledge space on $Q$, that is, a knowledge structure closed under union and intersection.  The design of the \query{} algorithm ensures that this quasi order is in fact a partial order.
\label{page_45}  Thus the quasi ordinal structure is an ordinal space $\LLL_1$, \label{L_1} namely, a learning space closed under intersection.  The learning space $\LLL_1$ contains 
all the actual knowledge states of the learning space under construction.  However, it also typically contains a possibly very large number of false states, which are due to the closure of intersection of $\LLL^1$.

While human experts are capable of providing useful responses to queries of the type [Q1], their responses to queries of higher blocks are less reliable.  Fortunately, despite the presence of false states, the learning space $\LLL_1$ is sufficiently informative to be used in the schools and colleges\footnote{The likely reason is that any false state $L$ may be close to some true state $K$ (in the sense of the symmetric difference distance $|K\bigtriangleup L| = |(L\setminus K) \cup (K\setminus L)|$).}. The data collected by assessments using $\LLL_1$ can then be used to simulate human expert responses to queries of higher block numbers, such as Block~2 or Block~3.

Here is how, taking Block~2 as our example.   Assuming that we have a very large collection of assessment data\footnote{Which is the typical case in the \subjindex{\aleks}\aleks{} system for example.}, we can estimate the conditional probabilities that failing both items~$r$ and $t$ entails failing also item~$q$.
Using a convenient abuse of notation, we write for these conditional probabilities\footnote{Technically, $q$, $r$, and $t$ denote items and not random variables}:
\begin{equation}
\Prob(q=0 \,\st\, r=t=0).
\end{equation}
Choosing a suitably large threshold value $\theta$, with  $0<\theta < 1$, we can construct the relation $\PPP$ for Block 2 by the rule:  
\begin{equation}\label{prob_query}   
\{r,t\}\PPP q   \qquad\text{exactly when}\qquad 
\Prob(q=0 \,\st\, r=t=0) > \theta.
\end{equation}
Using such estimates and Formula~\eqref{eq_entailment1_minus} help implementing Block~2 of the \query\ routine.  However, the resulting knowledge space is not necessarily a learning space.  Consequently, not all positive responses $\{r,t\}\PPP q$ should be implemented and result in the elimination of states.  This remark also applies in the general case, to the queries of any block.

\label{page_adapted}
The \query\ routine has been adapted for the construction of learning spaces.  As suggested by our  example of Block~2,  the general idea is to verify that the removal of a state by the application of Theorem~\ref{thm_K_D} would not result in a violation of the learning space axioms.  The key result is Theorem~16.1.6 in \authindex{Doignon, J.-P.}\authindex{Falmagne, J.-Cl.}\citet{Falmagne_Doignon_LS}, which is restated as Theorem \ref{16_1_6} here. It relies on the concept of a `hanging state'.

\begin{definition}\label{criticalhanging} 
A nonempty state $K$ in a knowledge structure $(Q,\KKK)$ is \subjindex{hanging (state)}\term{hanging} if its inner fringe $K^\III=\{q\in Q \,\st\, K\setminus \{q\}\in \KKK\}$ (cf.~Definition~\ref{fringe_def}) is empty.
The state $K$ is \subjindex{almost hanging (state)}\term{almost hanging} if it contains more than one item, but its inner fringe consists of a single item.
\end{definition}

\begin{example}\label{ex_hanging}
The knowledge space
$$
\LLL=\{\es,\{a\},\{b\},\{a,b\},\{a,c\},\{a,d\},\{a,b,c\},\{a,b,d\},\{a,c,d\},Q\},
$$
depicted in Figure~\ref{ex_hanging}
has no hanging states, and has two almost hanging states, which are $\{a,c\}$ and $\{a,d\}$. According to Theorem~\ref{thm_no_hanging} below, $\LLL$ must be a learning space. 
\hfill\EEX\end{example}

\begin{figure}[ht]
\begin{center}
\begin{tikzpicture}[yscale=1.2,baseline=60pt,every node/.style={rectangle, draw, rounded corners}]
\node[label= $\LLL$,shape=coordinate] at (-2.8,3.8) {};
\node (es)  at (0,0) {$\es$};
\node (a)   at (-1,1) {$\{a\}$};
\node (b)   at ( 1,1) {$\{b\}$};
\node (ad)  at (-2,2) {$\{a,d\}$};
\node (ab)  at ( 0,2) {$\{a,b\}$};
\node (ac)  at ( 2,2) {$\{a,c\}$};
\node (abd) at (-2,3) {$\{a,b,d\}$};
\node (acd) at ( 0,3) {$\{a,c,d\}$};
\node (abc) at ( 2,3) {$\{a,b,c\}$};
\node (Q)   at ( 0,4) {$Q$};
\draw (es) -- (a) -- (ad) -- (abd) -- (Q);
\draw (es) -- (b) -- (ab) -- (abc) -- (Q);
\draw (a) -- (ab) -- (abd);
\draw (a) -- (ac) -- (abc);
\draw (ad) -- (acd) -- (Q);
\draw (ac) -- (acd);
\end{tikzpicture}
\end{center}
\caption[A learning space with two hanging states]{The covering diagram of the knowledge space in Example~\ref{ex_hanging}.}
\label{fig_hanging}
\end{figure}
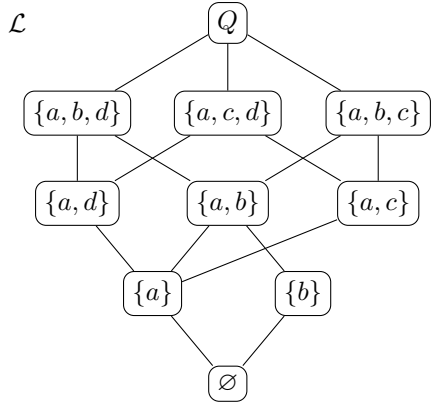

The following result is straightforward.

\begin{theorem}\label{thm_no_hanging}
A finite knowledge space is a learning space if and only if it has no hanging state.
\end{theorem}

Theorem~\ref{thm_no_hanging} indicates that, in the application of the \query\ routine, we must avoid the creation of hanging states (which could result from the removal of states by a positive response to a query).  We need one more tool to analyze the effect of such a positive response.

\begin{definition}\label{DKKK}
Let  $(Q,\KKK)$ be a knowledge space and let $(A,q)$ be any query such that $\es\neq A\subset Q$ and $q\in Q\setminus A$.  For any subfamily $\FFF$ of $\KKK$, we define 
\begin{equation}\label{DFAq}
\DDD_\FFF(A,q) = \{K\in\FFF \;\st\; A \cap K = \es \text{ and } q \in K\}.
\end{equation}
\end{definition}

Thus, $\DDD_\KKK(A,q)$ is the subfamily of all those states of $\KKK$ that would be removed by a positive response $A\PPP q$ to the query $(A,q)$ in the framework of the {\tt QUERY} routine.

\begin{theorem}\label{16_1_6} 
\sl For any knowledge space $\KKK$ and any query $(A,q)$, the family of sets $\KKK\setminus\DDD_\KKK(A,q)$ is a knowledge space. 
If $\KKK$ is a learning space, then $\KKK\setminus\DDD_\KKK(A,q)$ is a learning space if and only if there is no almost hanging state $L$ in $\KKK$ such that $A \cap L = L^\III$ and $q \in L$.
\end{theorem}

The \index{adapted \query{} routine}\term{adapted} \query{} routine for building a learning space is based on Theorem~\ref{16_1_6}.  As we described above (page~\pageref{L_1}), we obtain a learning space (actually an ordinal space) $\LLL_1$ at the end of Block~1. The next query from Block~2 is of the form $(\{r,t\},q)$.  A positive response $\{r,t\}\PPP q$ induced from the estimated conditional probabilities of Statement~\eqref{prob_query} would lead to the removal from $\LLL_1$ of any state $K$ such that
\begin{equation}
\{r,t\}\cap K \;=\; \es\quad\text{and}\quad q\in K,
\end{equation}
provided $\LLL_1\setminus \{K\}$ does not contain any hanging state (cf.\ Formula~\eqref{eq_entailment1_minus}).  But there is one more subtlety. The fact that $\LLL_1\setminus \{K\}$  contains some hanging state does not necessary mean that the query  $(\{r,t\},q)$ must be discarded for ever.  Indeed, it may be that, at a later stage of the execution of the algorithm, after the removal of some other states, implementing $\{r,t\} \PPP q$ no longer creates hanging states.  The adapted \query{} routine takes care of such details
\authindex{Doignon, J.-P.}\authindex{Falmagne, J.-Cl.}\citep[see][Chapter 15 and 16]{Falmagne_Doignon_LS}. 

So far, only Block~2 of the adapted \query\ routine has been implemented to built the learning space $\LLL_2$ in realistic situations involving several hundred items, such as Beginning Algebra as the subject is taught in the US.  Remarkably, in these cases, it was observed that $99\%$ of the false states in $\LLL_1$ are removed to form $\LLL_2$.  It is highly plausible that even better, smaller learning spaces would result from implementing Block~3 to $\LLL_2$, and maybe even higher blocks later on.  We have been told by the \subjindex{\aleks}\aleks{} team\footnote{Personnal communication.} that this costly enterprise is still at the project stage.

\medskip

\label{page_2014}
\authindex{Doignon, J.-P.}\cite{Doignon_2014} recently proposed another modification of the \query{} routine to build a learning space.  He uses a fundamental property of the collection of all the knowledge structures on the domain $Q$, a property which follows from results in \authindex{Caspard, N.}\authindex{Monjardet, B.}\cite{caspard_monjardet_2004}.  Here, a knowledge structure $\KKK$ on the domain $Q$ is included in a knowledge structure $\LLL$ also on $Q$ when any state in $\KKK$ is also a state in $\LLL$ (that is, $\KKK \subseteq \LLL$).

\begin{theorem}\label{theo_largest_learning_space}
For a given finite knowledge space $(Q,\KKK)$, there are two cases.  Either there is no learning space on $Q$ included in $\KKK$, or among all the learning spaces on $Q$ included in $\KKK$, there is one which includes all the other ones.
\end{theorem}

When the second case in Theorem~\ref{theo_largest_learning_space} occurs, we denote by $\KKK^\triangle$ the largest learning space included in the knowledge space $\KKK$.  
Theorem~\ref{theo_KKK_m} below provides a description of $\KKK^\triangle$ in terms of `gradations' (Theorems~\ref{theo_largest_learning_space} and \ref{theo_KKK_m} first appeared  in \citealp{Doignon_2014}).

\begin{definition}\label{gradation}
Let $(Q,\KKK)$ be a finite knowledge structure.  A \indexedterm{learning path} in $\KKK$ is a maximal chain of states.  A \indexedterm{gradation} $\CCC$ is a learning path which is accessible (Definition~\ref{space_wg_downgrad}), that is:  for any nonempty state $K$ in $\CCC$ there is some item $q$ in $K$ for which  $K\setminus\{q\}\in\KKK$.
\end{definition}

Notice that a learning path necessarily contains both states $\es$ and $Q$.  A gradation always consists of $1+|Q|$ states of respective sizes $0$, $1$, \dots, $|Q|$, and it forms itself a learning space on $Q$. There is an obvious correspondence between gradations and learning strings (Definition~\ref{def_learning_string}).

%\clearpage % ??

\begin{theorem}\label{theo_KKK_m}
Let $\KKK$ be a knowledge space on the finite domain $Q$.  If $\KKK$ includes at least one learning space on $Q$, then the largest learning space $\KKK^\triangle$ on $Q$ included in $\KKK$ is formed by all the states in all gradations in $\KKK$.
\end{theorem}

As the classical \query{} routine, the \indexedterm{adjusted \query{} routine} proposed in \authindex{Doignon, J.-P.}\cite{Doignon_2014} repeatedly asks queries 
and maintains a collection $\LLL$ of subsets of $Q$ which, although it is decreasing, always forms a learning space on $Q$.  When the query $(A,q)$ prompts a positive answer from the expert (or the database), the routine builds the collection $\KKK=\LLL\setminus\DDD_\LLL(A,q)$.  By Theorem~\ref{16_1_6}, the collection $\KKK$ is again a knowledge space.  Now there are two cases:
\begin{enumerate}[\qquad(i)~]
\item $\KKK$ includes some learning space (in other words: there is at least one gradation in $\KKK$, cf.~Theorem~\ref{theo_KKK_m}). Then, in view of Theorem~\ref{theo_largest_learning_space}, $\KKK$ includes for sure the learning space $\KKK^\triangle$.
The routine then replaces $\LLL$ with the latter learning space and asks a new query, if any is left unanswered. 
\item $\KKK$ does not include any gradation.  Then the routine exits execution, delivering the actual collection $\LLL$.
\end{enumerate}

Killing the routine in Case~(ii) above makes sense, because there is no learning space that complies with the answers collected to queries.  To the contrary, when the query answers are coherent in the sense that they reflect some latent learning space, Case~(ii) does not occur. Even more, the routine uncovers the latent learning space.

\begin{theorem}\label{prop_adjusted_query_fundamental}
If $\LLL$ is a latent learning space on the finite domain $Q$ and the query answers are truthful with respect to $\LLL$, then the adjusted \query{} routine will ultimately uncover $\LLL$.
\end{theorem}

In any case, the adjusted \query{} routine always produces a learning space (even in case it reached Case~(ii) above and therefore exited execution).  We will not go into the details of the implementation of the adjusted \query{} routine, nor report any comparison of its performances with those of the adapted \query{} routine.

%%%%%%%%%%%%%%%%%%%%%%%%%%%%%%%%%%%%%%%%%%%%%%%%%%%%%%%%%%%%%%%%%%%%%
\section[Some Applications---The \subjindex{\aleks}\aleks{} System]{Some Applications---The \aleks{} System
\sectionmark{Some Applications---The \aleks{} System}}
\sectionmark{Some Applications---The \aleks{} System}
\label{applications}
In psychometrics, the concepts of ``validity'' and ``reliability'' are distinct\footnote{See \authindex{Anastasi, A.}\authindex{Urbina, S.}\cite{Anastasi_1997} for explanations.} 
and dealt with separately, which is justified because a psychometric test is a measurement instrument which is not automatically predictive of a criterion.  In the applications of Learning Space Theory (LST) such as \subjindex{\aleks}\aleks, however, the items 
potentially\footnote{Potentially, in the sense that any item of the domain could be part of the test.} used in any assessment are, by construction, a fully comprehensive coverage of a curriculum, typically based on the consultation of numerous standard textbooks.  This implies that if a LST-type assessment is reliable, it is presumably also valid, and vice versa  \authindex{Doignon, J.-P.}\authindex{Falmagne, J.-Cl.}\citep[cf.~Section 17.1 in][]{Falmagne_Doignon_LS}. 
  
In our introduction, we have presented LST as a more predictive, or valid, alternative to psychometrics. We describe here four analysis of data which all vindicate that position.  The material is taken from 
\authindex{Cosyn, E.}\authindex{Doble, C.W.}\authindex{Falmagne, J.-Cl.}\authindex{Lenoble, A.}\authindex{Thi\'ery, N.}\authindex{Uzun, H.}\citet{Cosyn_D_F_T_U_2013}.
However, the statistical analysis is so complex and detailed that we can only provide a summary here.
   
%-------
\subsection{The extra problem method} 
\label{extra_problem}
The first inquiry investigates whether the knowledge state uncovered at the end of an assessment is predictive of the responses to questions that were not asked. For example, there are $416$ items in the elementary school mathematics domain used in the \subjindex{\aleks}\aleks{} system, of which at most  $35$ are used in any assessment. Can we reliably 
predict the student's responses to the $381=416-35$  remaining questions? 
   
To find out, an \indexedterm{extra problem} is asked in any assessment, the response to which is not used in uncovering the state.  The authors of \authindex{Cosyn, E.}\authindex{Doble, C.W.}\authindex{Falmagne, J.-Cl.}\authindex{Lenoble, A.}\authindex{Thi\'ery, N.}\authindex{Uzun, H.}\citet{Cosyn_D_F_T_U_2013} have evaluated the correlation between the response to the extra problem---coded as $1/0$ for correct/false---and a couple of predictive indices obtained from or during the assessment. 
   
One predictive index is the dichotomic variable coding the observation that the extra problem is either in or out of the final state. Accordingly, the data takes the form of the $2$ by $2$ matrix in Table~\ref{cor_matrix}.

\begin{table}[h] 
\caption[Basic data matrix for the extra problem]
{Basic data matrix for the computation of the correlation between the cases ``in or out of the final state'' and the student's response coded as $1/0$ for correct/false. So, theoretically, $z$ stands for the number of lucky guesses, and $y$ for the number of careless errors.}
\label{cor_matrix}
\renewcommand{\arraystretch}{1.5}
\renewcommand{\tabcolsep}{4mm}
\begin{center}
\begin{tabular}{cccc}
&&\multicolumn{2}{c}{Response}\\
&& 1(correct) & 0 (false)\\
\cline{3-4}
\multirow{2}{*}{State}
& in  & $x$ & $y$\\
& out & $z$ & $w$\\
\cline{3-4}
\end{tabular}
\end{center}
\end{table} 

Because no correlation coefficient is available that would be fully adequate for the situation, two potentially useful ones were computed by Cosyn and his colleagues for these data:   the tetrachoric and the phi coefficient.  The results are given in Table~\ref{cor_matrix_bis} for the median values of the coefficients and the grouped data. These data pertain to $125{,}786$ assessments using $324$ problems out of the $370$ problems mentioned above\footnote{Forty-six problems were discarded because the relevant data were not sufficient to provide reliable estimates of the coefficients.}.  The tetrachoric values 
are in the left column of the table. The median of the distribution is around .68. The grouped data, obtained from gathering the 324 individual $2 \times 2$ matrices into one, yields a higher correlation of about .80. These are high values, but the tetrachoric coefficient is regarded as biased upward \authindex{Greer, T.}\authindex{Dunlap, W.P.}\authindex{Beaty, G.O.}\citep[however, see][]{greer03}. The right column of the table contains similar values for the phi coefficient. These values are much lower, yielding a median of .43 (contrasting with the .68 value obtained for the tetrachoric) and a grouped data value of .58 (instead of .80).
  
%%%%%%%%%%%%%%%%%%%%%%%%%% table
\begin{table}[h]
%\begin{minipage}{300pt}
\caption[Values of correlation coefficients (extra problem)]{Values of the coefficients correlating the $1/0$ variables coding the in/out of the final state and the correct/false response in the two cases: median value of the correlation distribution and grouped data.}
\label{cor_matrix_bis}
\renewcommand{\arraystretch}{1.5}
\renewcommand{\tabcolsep}{5mm}
\begin{center}
\begin{tabular}{r|cc|}
\multicolumn{1}{c}{}&tetrachoric&\multicolumn{1}{c}{phi}\\
\cline{2-3}
median&$.68$& $.43$\\
grouped data&$.80$&  $.58$\\
\cline{2-3}
\end{tabular}
\end{center}
%\end{minipage}
\end{table} 
%%%%%%%%%%%%%%%%%%%%%%%%%% end of table

%----------
\def\rba{{\rm\bf a}}
\subsection{Correcting for careless errors} 
The relatively low correlation values obtained for the phi coefficient in the above analysis may be due in part to the occurrence of careless errors. However, the basic data matrix of Table~\ref{cor_matrix_bis} may be revised to include such careless errors. Instead of the dichotomic \textsl{in/out (of the final state)} variable, \authindex{Cosyn, E.}\authindex{Doble, C.W.}\authindex{Falmagne, J.-Cl.}\authindex{Lenoble, A.}\authindex{Thi\'ery, N.}\authindex{Uzun, H.}\citet{Cosyn_D_F_T_U_2013} define the new variable 
\begin{equation}
\mathbf S_{\rba} = \begin{cases} 1 - \epsilon_{\rba} &\text{if the final state  contains the extra question $\rba$,}\\
0&\text{otherwise,}
\end{cases}
\end{equation}
in which $\epsilon_{\rba}$ stands for the probability of committing a careless error to item~$\rba$.  
So, if the extra question $\rba$ is contained in the uncovered state, $\mathbf S_{\rba}$ is the probability of not committing a careless error to item~$\rba$. The careless errors probabilities $\epsilon_{\rba}$ were estimated from those cases in which, by chance, the same item~$\rba$ appears twice in an assessment, once as the extra problem and the other as one of the other items, and moreover, the response in at least one of the two instances of $\rba$ is correct. The relative frequency of cases in which a false response is given to the other instance of $\rba$ provides an estimate of $\epsilon_\rba$.  

The variable $\mathbf S_{\rba}$ is neither exactly continuous nor exactly discrete\footnote{For example, the distribution of $\mathbf S_{\rba}$ vanishes in a positive neighborhood of $0$, but is positive at the point $0$ itself.}. Nevertheless, for the purpose of comparison with similar analyses performed in psychometric situations, the authors have used the point biserial coefficient $r_{pbis}$ to compute the correlation between the variables $\mathbf S_{\rba} $ and the 1/0 variable coding the correct/incorrect responses to the extra problem.   The value reported for $r_{pbis}$ was $.67$, noticeably higher than the $.58$ obtained for the phi coefficient for the same grouped data. 

The authors compare this $.67$ value for the point biserial coefficient with those disclosed for the same coefficient in the \citet{ETSfeb08} report\footnote{Produced for the California Department of Education (Test and Assessment Division). See {\tt http://www.cde.ca.gov/ta/tg/sr/documents/csttechrpt07.pdf}.} for the Algebra I California Standards Test (CST), which covers approximately the same curriculum as the \subjindex{\aleks}\aleks{} assessment for elementary algebra and is given to more than 100\,000 students each year. 
This test consists of 65 multiple choice questions (items) and is built and scored in the framework of Item Response Theory (IRT), for which the point biserial coefficient is a standard measure.  In particular, for each of the 65 items, a point biserial item-test correlation was computed, which measured the relationship between the dichotomous variable giving the $1/0$ item score (correct/incorrect) and the continuous variable giving the total test score (see p.~397 of the ETS report referenced above).
For the 2007 administration of the Algebra I CST, the mean point biserial coefficient for the 65 items was $.36$, and the median was $.38$  (see Table~7.2, p.~397 of the ETS report). The minimum coefficient obtained for an item was $.10$ and the maximum was $.53$ (Table 7.A.4, pp.~408--9, of the ETS report).  The averages for preceding years on the test were similar, namely, the mean point biserial coefficients were $.38$ in 2005 and $.36$ in 2006 (see Table 10.B.3, p.~553, of the same report)\footnote{These low correlation values were obtained even though some items with low item test correlations  were removed from the test in a preliminary analysis.}. 

%----------
\subsection{Learning readiness}
\label{outerfringe_success} 
A student using the \subjindex{\aleks}\aleks{} system is given regular assessments. At the end of each assessment\footnote{Except the final one, at the end of the course.}, the system is giving the student the choice of the next item to learn, such items being located in the outer fringe of the student's knowledge state\footnote{See Section~\ref{fringe_theorem} for the concept of outer fringe.}. This makes sense because the items in the outer fringe are exactly those that the student is ready to learn. So, if the validity of the assessment is high, the probability of successfully learning an item chosen in the outer fringe should also be high. \authindex{Cosyn, E.}\authindex{Doble, C.W.}\authindex{Falmagne, J.-Cl.}\authindex{Lenoble, A.}\authindex{Thi\'ery, N.}\authindex{Uzun, H.}\citet{Cosyn_D_F_T_U_2013} have estimated the probability that a student successfully learns an item chosen in the outer fringe of his or her state. For elementary school mathematics, the median of the distribution of the estimated (conditional) success probabilities was $.93$.  
This estimated was based on $1{,}940{,}473$ learning occasions. 
  
%-----------
\subsection{\subjindex{\aleks}\aleks{} based placement at the University of Illinois}\label{aleks_illinois}
Until 2006, the students entering the University of Illinois had to take a mathematics placement test based on the ACT\footnote{Originally, an acronym for \textsl{American College Testing}.}.  The results were not satisfactory because many students lack an adequate preparation for the course they were advised to take and ended up withdrawing.  Beginning in 2007, the ACT was replaced by \subjindex{\aleks}\aleks.  \authindex{Ahlgren, A.}\authindex{Harper, M.}\citet{Ahlgren_Harper_2013} reports a comparison of the two situations, which we briefly sum up here \authindex{Ahlgren, A.}\authindex{Harper, M.}\citep[see also][]{Ahlgren_2011}.

The consequence of a withdrawal from a course may have dire consequences.  For example, we read in the introductory section of \authindex{Ahlgren, A.}\authindex{Harper, M.}\citet{Ahlgren_Harper_2013}:

  \begin{quote}\sl
  ``At the University of Illinois, the standard introductory calculus course (Calc: Calculus~I) is a five credit course, the withdrawal from which beyond the add-deadline may reduce students to a credit total below full-time status, resulting in the loss of tuition benefits, health benefits, scholarships, and athletic eligibility.'' 
  \end{quote} 
 
The placement program of the Department of Mathematics at the University of Illinois deal with four courses:  Preparation for Calculus (\PreCalc),  Calculus 1 (\Calc), Calculus 1 for student with experience (\CalcExp), and Business Calculus (\BusCalc).  Because the \PreCalc{} course was only offered for the first time in 2007, it is not included in the comparison statistics below.
   
The new placement exam is an \subjindex{\aleks}\aleks{} assessment focused on the module\footnote{A particular learning space used in the \subjindex{\aleks}\aleks{} system.} `\texttt{Prep for Calculus}', with some items removed that are not part of the relevant course at the University of Illinois. The students take the non-proctored assessment either at home or on campus. Students failing to reach the required score for placement in one of the four courses have  the options of taking the placement test again or using the \subjindex{\aleks}\aleks{} learning module (or some other method) to improve their results. The students can keep retaking assessments until the course add-deadline. 
  
Table~\ref{w_decrease} on page \pageref{w_decrease} gives the percentages of the  decrease in the numbers of withdrawals observed in 2007 and 2008 for the \BusCalc{}, \Calc{} and \CalcExp{} courses, in comparison with the numbers of withdrawals from the same courses in 2006.  

\begin{table}
\caption[Percentages of decrease in withdrawals from courses]{Percentages of decrease in withdrawals from three courses between Fall 2006 (before \subjindex{\aleks}\aleks) and Fall 2007, Fall 2008.}
\label{w_decrease}

\vskip2mm

\renewcommand{\arraystretch}{1.5}
\renewcommand{\tabcolsep}{.5cm}
\begin{center}
\begin{tabular}{|l|c|c|c|}
\cline{2-4}
\multicolumn{1}{c}{}&\multicolumn{3}{|c|}{Percentage of Decrease in Withdrawals}\\
\hline
Course&\BusCalc&\Calc&\CalcExp\\
\hline
2006--2007&{\scriptsize $\searrow$\,}57\%&{\scriptsize $\searrow$\,}49\%&{\scriptsize $\searrow$\,}67\%\\[2mm]
2006--2008&{\scriptsize $\searrow$\,}19\%&{\scriptsize $\searrow$\,}81\%&{\scriptsize $\searrow$\,}42\%\\
\hline
\end{tabular}
\end{center}
\end{table}
The decrease in the percentages of withdrawals is remarkable. The author also note a shift in the enrollment---not reported here---from
the least advanced course, \BusCalc, to the more advanced one, which is \CalcExp. Presumably, this may be due to the combination of two factors: the accuracy of the assessment in the \subjindex{\aleks}\aleks{} system, which may improve the placement, and 
the fact the students were given the possibility of using the system to bridge some gaps in their expertise. 
  
%%%%%%%%%%%%%%%%%%%%%%%%%%%%%%%%%%%%%%%%%%%%%%%%%%%%%%%%%%%%%%%%%%%%%
\section[Bibliographical Notes]{Bibliographical Notes
\sectionmark{Bibliographical Notes}}
\sectionmark{Bibliographical Notes}
\label{bibliography} 
The first paper on knowledge spaces was \authindex{Doignon, J.-P.}\authindex{Falmagne, J.-Cl.}\citet{Doignon_Falmagne_1985}, which contains some of the combinatorial basis of the theory, and in particular one of two key properties, namely, the closure under union: a knowledge space is a knowledge structure closed under union.  \authindex{Doignon, J.-P.}\citet{doignon:88} was a technical follow up. The stochastic part of the theory was presented in the two papers by \authindex{Falmagne, J.-Cl.}\citet{falmagne:88a,falmagne:88b}. The second paper also introduces a second key property, wellgradedness.  A comprehensive, non-technical description of these basic ideas is contained in \citet{falma90}. \citep[See also][for other introductory papers.]{doignon:ganter,falma89a,doignon:albert} 

Even though, taken together, the closure under union and wellgradedness form a legitimate basis of the theory, it is not obvious that these axioms are pedagogically sound.  Around 2002, Falmagne proposed a reaxiomatization in the form of the conditions labelled here as [L1] and [L2] (see Subsection~\ref{def_ls_axioms} on page \pageref{def_ls_axioms}), which he called a \textsl{learning space}. These axioms were unpublished at the time, but they were 
communicated to Eric Cosyn and Hasan Uzun, who then proved that a knowledge structure is a learning space if and only if it is a well-graded knowledge space \authindex{Cosyn, E.}\authindex{Uzun, H.}\citep[see][and our Theorem~\ref{big_equivalence} on page \pageref{big_equivalence}]{Cosyn_Uzun_2009}.  A similar result is contained in \authindex{Korte, B.}\authindex{Lov\'asz, L.}\authindex{Schrader, R.}\citet{Korte_Lovasz_Schrader_1991}.  In fact, learning spaces are exactly the same structures as antimatroids or convex geometries.  The origin of the term antimatroid (in the setting of intersection-closed families of subsets) goes back to \authindex{Jamison, R.E.}\cite{Jamison_1980} and \cite{Jamison_1982}; see also \authindex{Edelman, P.H.}\cite{Edelman_1986} for a lattice-theoretic approach, and \authindex{Edelman, P.H.}\authindex{Jamison, R.E.}\cite{Edelman_Jamison_1985}. 
 
In time, other scientists became interested in knowledge space theory, notably \authindex{Dowling, C.E.}\authindex{M\"uller, C.}Cornelia Dowling\footnote{Formerly Cornelia  M{\"u}ller; see this name also for references.}, \authindex{Heller, J.}J\"urgen Heller and \authindex{Suck, R.}Reinhard Suck in Germany, \authindex{Albert, D.}Dietrich Albert and \authindex{Hockemeyer, C.}Cord Hockemeyer in Austria, \authindex{Koppen, M.}Mathieu Koppen in The Netherlands, and \authindex{Stefanutti, L.}Luca Stefanutti in Italy. 
The literature on the subject grew and now contains several hundred titles. The data base at

\noindent\texttt{http://liinwww.ira.uka.de/bibliography/Ai/knowledge.spaces.html}

\noindent contains an extensive list of reference on knowledge spaces. It is maintained by Cord Hockemeyer at the University of Graz \authindex{Hockemeyer, C.}\citep[see also][]{hockemeyer:tr01c}.  

The monograph by \authindex{Doignon, J.-P.}\authindex{Falmagne, J.-Cl.}\citet{Doignon_Falmagne_KS} gives a comprehensive description of knowledge space theory at that time. Beginning in 1994, an internet-based educational software based on learning space theory was developed at the University of California, Irvine (UCI), by a team of software engineers including (among many others) Eric Cosyn,  Damien Lauly\authindex{Lauly, D.}, David Lemoine\authindex{Lemoine, D.} and Nicolas Thi{\'e}ry\authindex{Thi{\'e}ry, N.}, under the supervision of Falmagne.  This was supported by a large NSF grant to UCI.  The end product of this software development was called \subjindex{\aleks}\aleks, which is also the name of a company created in 1996 by Falmagne and his graduate students\footnote{In particular Cosyn, Lauly and Thi{\'e}ry. \subjindex{\aleks}\aleks{} Corporation was sold to McGraw-Hill Education in 2013.}.  Around 1999, it was thought that the educational software had matured enough for a fruitful application in the schools and universities. This led to further developments in the form of statistical analysis of the assessment data and new mathematical results. \authindex{Falmagne, J.-Cl.}\authindex{Cosyn, E.}\authindex{Doignon, J.-P.}\authindex{Thi\'ery, N.}\citet{falma06} is a not-too-technical introduction to the topic at that time. The monograph by \authindex{Doignon, J.-P.}\authindex{Falmagne, J.-Cl.}\citet[][a much expanded reedition of Doignon and Falmagne, 1999]{Falmagne_Doignon_LS} is a comprehensive and (almost) 
up-to-date technical presentation of the theory.
 
Several assessment systems are founded on Knowledge Space Theory and Learning Space Theory, the most prominent ones being \subjindex{\aleks}\aleks{} and \rath. (The \rath{} system was developed by a team of researchers at the University of Graz, in Austria; see
\authindex{Hockemeyer, C.}\citealp{hocke97} and \authindex{Held, Th.}\authindex{Albert, D.}\citealp{hockemeyer_98_calisce}).

%%%%%%%%%%%%%%%%%%%%%%%%%%%%%%%%%%%%%%%%%%%%%%%%%%%%%%%%%%%%%%%%%%%%

%\backmatter

%\bibliographystyle{cambridgeauthordate}
%\bibliography{jpd}
%\label{refs}

% for a single index
\renewcommand{\indexname}{Index}
\printindex 

\end{document}